\documentclass{article}
\usepackage{fullpage}
\usepackage[utf8]{inputenc}
\usepackage{subfigure}
\usepackage[round]{natbib}
\usepackage{amsfonts,amssymb,amsmath,dsfont,amsthm}
\usepackage{enumitem}
\usepackage{xspace}
\usepackage{graphicx}
\usepackage[colorlinks=true,citecolor=blue]{hyperref}
\usepackage{bm}
\usepackage{comment}

\newtheorem{theorem}{Theorem}
\newtheorem{proposition}{Proposition}
\newtheorem{lemma}{Lemma}

\newtheorem{definition}{Definition}

\newcommand{\bigO}{\mathcal{O}}
\newcommand{\exactbigO}{\Theta}
\newcommand{\smallo}{o}
\newcommand{\integers}{\mathds{Z}}
\newcommand{\reals}{\mathds{R}}
\newcommand{\naturals}{\integers_{\geq 0}}
\newcommand{\complex}{\mathds{C}}

\newcommand{\mA}{\mathcal{A}}
\newcommand{\mB}{\mathcal{B}}
\newcommand{\mC}{\mathcal{C}}
\newcommand{\mD}{\mathcal{D}}
\newcommand{\mF}{\mathcal{F}}
\newcommand{\mH}{\mathcal{H}}

\newcommand{\mK}{\mathcal{K}}
\newcommand{\mN}{\mathcal{N}}
\newcommand{\mM}{\mathcal{M}}
\newcommand{\mL}{\mathcal{L}}

\newcommand{\vect}[1]{\bm{#1}}
\newcommand*{\eg}{\textit{e.g.}\@\xspace}
\newcommand*{\ie}{\textit{i.e.}\@\xspace}
\newcommand*{\aka}{\textit{a.k.a.}\@\xspace}
\newcommand*{\resp}{resp.\@\xspace}
\newcommand{\Real}{\operatorname{Re}}

\newcommand{\sg}{\operatorname{SG}}
\newcommand{\mg}{\operatorname{MG}}
\newcommand{\csg}{\operatorname{CSG}}
\newcommand{\cmg}{\operatorname{CMG}}
\newcommand{\core}{\operatorname{Core}}
\newcommand{\mcore}{\operatorname{MCore}}
\newcommand{\patch}{\operatorname{Patch}}
\newcommand{\kernel}{\operatorname{Kernel}}
\newcommand{\mcycle}{\operatorname{MCycle}}
\newcommand{\mk}{\operatorname{MK}}
\newcommand{\sk}{\operatorname{SK}}
\newcommand{\cmk}{\operatorname{CMK}}
\newcommand{\csk}{\operatorname{CSK}}
\newcommand{\LD}{\operatorname{LD}}
\newcommand{\ld}{\operatorname{ld}}
\newcommand{\IE}{\operatorname{IE}}
\newcommand{\IEd}{\operatorname{IE}_{<d}}

\newcommand{\mgsimple}{\mg^{\setminus \LD}}
\newcommand{\mgpos}{\mg^{>0}}
\newcommand{\sgpos}{\sg^{>0}}
\newcommand{\mcorepos}{\mcore^{>0}}
\newcommand{\poscore}{\core^{>0}}
\newcommand{\pospatch}{\patch^{>0}}

\newcommand{\supp}{\operatorname{Supp}}

\newcommand{\vk}{\vect{k}}
\newcommand{\vn}{\vect{n}}
\newcommand{\vm}{\vect{m}}
\newcommand{\vr}{\vect{r}}
\newcommand{\vx}{\vect{x}}
\newcommand{\vy}{\vect{y}}
\newcommand{\vz}{\vect{z}}
\newcommand{\vzero}{\vect{0}}
\newcommand{\vone}{\vect{1}}
\newcommand{\vlambda}{\vect{\lambda}}
\newcommand{\vtheta}{\vect{\theta}}
\newcommand{\vzeta}{\vect{\zeta}}

\author{\'Elie de Panafieu\thanks{Email: depanafieuelie[at]gmail.com. This work was partially founded by 
Sorbonne Universit\'es, UPMC Univ Paris 06, CNRS, LIP6 UMR 7606, France;
the Austrian Science Fund (FWF) grant F5004;
the Amadeus program; the PEPS HYDrATA; and the LINCS \texttt{www.lincs.fr}.}\\
{\normalsize Bell Labs France, Nokia}}

\title{Analytic combinatorics of connected graphs\footnote{A short version of this work has been published by \cite{EdP16}.}}

\begin{document}
\maketitle

\begin{abstract}
We enumerate the connected graphs that contain a number of edges growing linearly with respect to the number of vertices. So far, only the first term of the asymptotics and a bound on the error were known. Using analytic combinatorics, \ie generating function manipulations, we derive a formula for the coefficients of the complete asymptotic expansion.
The same result is derived for connected multigraphs.

\noindent \textbf{keywords.} connected graphs, analytic combinatorics, generating functions, asymptotic expansion
\end{abstract}

\section{Introduction}
\label{sec:introduction}
This article analyzes the asymptotics of the number $\csg_{n,k}$ of connected graphs 
with $n$ vertices and $n+k$ edges.
Following the definition of \cite{JKLP93}, the quantity $k$,
equal to the difference between the numbers of edges and vertices, 
is called the \emph{excess} of the graph.

    \subsection{Related works}

The enumeration of connected graphs according 
to their number of vertices and edges has a long history.
We have chosen to present it not chronologically,
but from the sparsest to the densest graphs,
\ie according to the speed growth of the excess
with respect to the number of vertices.

Trees are the simplest connected graphs,
and reach the minimal excess $-1$.
They were enumerated in 1860 by Borchardt,
and his result, known as \emph{Cayley's formula},
is $\csg_{n,-1} = n^{n-2}$.
\cite{R59} then derived the formula
for $\csg_{n,0}$, which corresponds to connected graphs 
that contain exactly one cycle,
and are called \emph{unicycles}.
Proofs of those two last results, based on analytic combinatorics,
are available in \cite{FS09}.
\cite{W80} applied generating function techniques
and a combinatorial argument based on \emph{$3$-cores}, or \emph{kernels},
to derive the asymptotics of connected graphs 
when $k$ is a constant, or is slowly going to infinity (k = $\smallo(n^{1/3})$).
%
\cite{FSS04} derived a complete asymptotic expansion for connected graphs with fixed excess,
following a purely analytic approach, discussed in Section~\ref{sec:outline}.

\cite{L90} obtained the asymptotics of $\csg_{n,k}$
when $k$ goes to infinity while $k = \smallo(n)$.
\cite{BCM90} derived the asymptotics for all $k$.
Their proof was based on the differential equations obtained by Wright,
involving the generating functions of connected graphs indexed by their excesses.
%
Since then, two simpler proofs were proposed.
The proof of \cite{HS06} used probabilistic methods, 
analyzing a breadth-first search on a random graph.
The proof of \cite{PW05} relied on the enumeration 
of graphs with minimum degree at least $2$.
The present work follows the same global approach.
The main difference is that, contrary to Pittel and Wormald
who worked at the level of the sequences enumerating graph families,
we use the powerful setting of generating functions to represent those families.
This enables us to shorten the proofs, and to derive more terms in the asymptotics.

\cite{ER60} proved that almost all graphs are connected
when $2k/n - \log(n)$ tends to infinity.
As a corollary, the asymptotics of connected graphs
with those parameters is equivalent to 
the total number of graphs $\binom{n(n-1)/2}{n+k}$.

    \subsection{Motivations and contributions}

Our main result is Theorem~\ref{th:csg_asymptotics},
which provides a complete asymptotic expansion
for the number of connected graphs with a number of edges
growing linearly with the number of vertices,
of the form
\[
    \csg_{n,k} = D_{n,k} \left( c_0 + c_1 k^{-1} + \cdots + c_{d-1} k^{-(d-1)} + \bigO(k^{-d}) \right).
\]
Expressions are provided for $D_{n,k}$ and $c_0$ in this theorem.
We explain how to compute the other $(c_r)$ coefficients in Appendix~\ref{sec:computation},
and provide there the expression of $c_1$.
We thank an anonymous referee for providing us with
large tables of numbers of connected graphs.
Part of them are presented in Figure~\ref{fig:numerical_verification}.
The correct digits obtained by the asymptotic expansion
of order $1$ or $2$ are highlighted.

After three proofs of the asymptotics of connected graphs
when the excess grows linearly with the number of vertices,
what is the point of deriving yet another one?
A first reason is that each proof introduces new techniques,
which can then be applied to investigate other graph families.
In our case, those techniques are the following.
\begin{itemize}
\item
It was already observed by \cite{FKP89} and \cite{JKLP93}
that multigraphs (loops and double edges allowed)
are better suited for generating function manipulations than simple graphs.
In Section~\ref{sec:def_multigraphs},
we improve their model to make it more compatible
with the formalism of the symbolic method (\cite{FS09})
and of species theory (\cite{BLL97}).
\item
The generating functions of graphs with degree constraints
were recently computed by \cite{EdPR16},
and we apply and improve this result
to enumerate graphs and multigraphs with minimum degree at least $2$.
\item
We apply an inclusion-exclusion technique to remove 
loops and double edges from multigraphs, turning them into graphs.
\cite{EdPCGGR17} have recently extended this new approach
to enumerate graphs with forbidden subgraphs,
and to count the occurrences of subgraphs from a given family in random graphs.
\item
New exact expressions for the generating functions
of interesting families of (multi)graphs are derived, 
including multigraphs with a given excess and degree constraints
(Proposition~\ref{th:deg_constraints}),
and graphs and multigraphs without trees and unicycles
(Propositions~\ref{th:mgpos} and~\ref{th:sgpos}).
\end{itemize}
Two other interesting techniques are applied:
a multivariate saddle-point method, Theorem~\ref{th:large_powers},
strongly influenced by the results of \cite{PW13},
and a divergent series analysis, Lemma~\ref{th:bender_like},
borrowed and slightly modified from \cite{Bo16}
(who was influenced by the work of \cite{Be75}).
Those tools are respectively developed in Appendix~\ref{sec:saddle_point} and~\ref{sec:divergent_series}.

A second motivation is the analysis
of the typical structure of random graphs.
\cite{ER60} started this study,
following a probabilistic approach.
One of their most striking result is
that a typical random graph
with $n$ vertices and $m = \exactbigO(n)$ edges
\begin{itemize}
\item contains only trees and unicycles if $\lim_{n \to + \infty} m/n < 1/2$,
\item contains only trees, unicycles, and a unique \emph{giant component}
if $\lim_{n \to + \infty} m/n > 1/2$.
\end{itemize}
In the first case, the graph is said to be \emph{sub-critical},
and \emph{super-critical} in the second case.
Precise results were derived by \cite{JKLP93} in the \emph{critical case},
which corresponds to $m/n = 1/2 + \bigO(n^{-1/3})$.
They proved that those graphs contain only
components of bounded excess, with high probability.
This work, based on analytic combinatorics,
used the expressions of the generating functions
of connected graphs with a fixed excess
obtained by \cite{W80}.
However, the excess of the giant component typically
grows linearly with its number of vertices,
and the generating function of such component was not known
in a form allowing asymptotic analysis
(this point is also discussed in Section~\ref{sec:outline}).
Thus, the structure of super-critical random graphs
has not been yet investigated using analytic combinatorics.
One of the contributions of the present paper is the derivation
of such a generating function (Theorem~\ref{th:csg_exact}),
and the tools to analyze it.
In a future contribution, we plan to extend the present work and derive precise results
on the structure of super-critical random graphs.

A third motivation for the derivation
of a new proof for the asymptotics of connected graphs
is that each proof might be extended
to various generalizations of the classical graph model.
We are currently working on the structure
of random graphs with degree constraints
(to extend the work of \cite{EdPR16}),
of non-uniform hypergraphs (\cite{EdPhp15}),
and of inhomogeneous graphs (\cite{EdP15, PR14}).

Finally, our result is more precise than the previous ones:
we derive an asymptotic expansion,
\ie a potentially infinite number of error terms.
This is characteristic of the analytic combinatorics approach,
where the generating functions capture all the combinatorial information,
and loss occurs only at the asymptotic extraction.
However, it should be noted that the formula
for the coefficients of this asymptotic expansion
is rather long (see Appendix~\ref{sec:computation}).

The proofs of this paper are based on analytic combinatorics,
which classically follows two steps.
First, the combinatorial structures of the families of graphs we are interested in
are translated into generating function relations.
This is achieved applying tools developed
by species theory (\cite{BLL97}) and the symbolic method (\cite{FS09}).
A short introduction to those tools is provided by Section~\ref{sec:analytic_combinatorics}.
Then the asymptotic expansions of the cardinality of those families are extracted.
%
We chose to work more on the combinatorial part,
deriving the generating functions in a ``nice'' form,
so that the asymptotic extractions are achieved using ``black box theorems''
(Lemma~\ref{th:bender_like}, closely related to the results of \cite{Be75} and \cite{Bo16},
and Lemma~\ref{th:mgpos_asymptotics}, a corollary of the work of \cite{PW13}).

    \subsection{Structure of the article}

The graph and multigraph models are presented in Section~\ref{sec:models},
as well as an outline of the forthcoming proof.
Section~\ref{sec:cmg} focuses on multigraphs.
The main result is Theorem~\ref{th:cmg_asymptotics},
where the asymptotic expansion of the number of connected multigraphs
with $n$ vertices and excess $k$, proportional to $n$, is computed.
The corresponding result for simple graphs is derived
in Theorem~\ref{th:csg_asymptotics}, from Section~\ref{sec:csg}.
The classical results on the asymptotics of connected graphs and multigraphs with fixed excess
are recalled in Section~\ref{sec:fixed_k}.
This article relies on two technical tools.
The first one is the multivariate saddle-point method,
presented in Appendix~\ref{sec:saddle_point}.
The second one concerns the asymptotic analysis
of the coefficients of divergent series,
available in Appendix~\ref{sec:divergent_series}.
The main result of this article is the asymptotic expansion of connected graphs
with an excess growing linearly with the number of vertices.
Instructions for the computation of the coefficients of this expansion
are provided in Appendix~\ref{sec:computation}.
As an illustration, the first two coefficients are computed.

    \subsection{Analytic combinatorics} \label{sec:analytic_combinatorics}

To make this paper more self-contained,
we present a brief introduction to the \emph{symbolic method} of analytic combinatorics, without proofs.
\cite{FS09} provide a more rigorous and complete presentation.
The reader already familiar with those notions can skip this section.

A \emph{labeled object} $f$ is a graph-like object
where the nodes are labeled with distinct consecutive integers starting at $1$.
The number of nodes, also equal to the largest label,
is always assumed to be finite, and is the \emph{size} $|f|$ of $f$.
For example, a rooted labeled tree is a labeled object.
It would be convenient that a set, or a sequence,
or any structured collection of labeled objects
would itself be a labeled object.
That way, we could for example investigate pairs or sets of rooted trees.
However, this is not the case, because such a collection
contains, in general, several nodes wearing the same label,
which is forbidden by the definition.
To solve this problem, the notion of \emph{relabeling} is introduced.
It looks technical at first, but is in fact natural and easy to apply.
A relabeling of a sequence $f = (f_1, \ldots, f_t)$ of labeled objects
is a labeled object $g = (g_1, \ldots, g_t)$,
such that for all $1 \leq i \leq t$,
$g_i$ is equal to $f_i$ up to an increasing relabeling of its nodes.
Hence, there is a strictly increasing function $\alpha$
that sends the labels of $f$ on the labels of $g$.
This implies, in particular, that the nodes of the $g_i$ have distinct labels,
and that
\[
    |(g_1, \ldots, g_t)| = |g_1| + \cdots + |g_t|.
\]
An example is provided in Figure~\ref{fig:relabeling}.

\begin{figure}
\begin{center}
\includegraphics[scale=0.7]{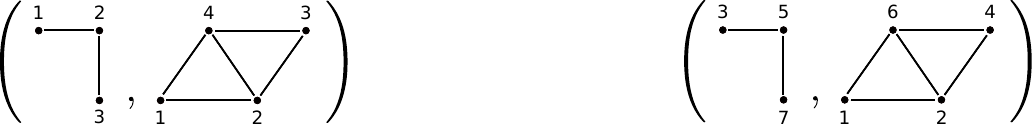}
\caption{Left, a pair of labeled objects.
This pair is not a labeled object itself,
because the same label appears on several nodes.
Right, one of the $\binom{7}{3}$ relabeling of the left pair.
Observe that the relative order of the labels of each object is preserved.
The right pair is a labeled object.}
\label{fig:relabeling}
\end{center}
\end{figure}

A \emph{labeled combinatorial family} $\mF$ is a collection of labeled objects,
such that for any $n \geq 0$, the number $F_n$ of objects of size $n$ in $\mF$ is finite.
The first principle of analytic combinatorics
is to associate to $\mF$ a \emph{generating function},
which is the formal sum
\[
    F(z) = \sum_{f \in \mF} \frac{z^{|f|}}{|f|!} = \sum_{n \geq 0} F_n \frac{z^n}{n!},
\]
where the third expression has been obtained from the second by grouping the terms of same size $n$.
For example, denoting by $T(z)$ the generating function of rooted trees,
and $T_n$ the number of rooted trees on $n$ vertices, we have
\[
    T(z) = \sum_{n \geq 0} T_n \frac{z^n}{n!},
    \quad \text{ thus } \quad
    T_n = n! [z^n] T(z).
\]

The second principle of analytic combinatorics
is to translate the combinatorial structure known on the family $\mF$
into equations that characterize its generating function $F(z)$.
Information on the coefficients $(F_n)$,
such as exact expressions and asymptotics,
are then extracted from those equations.
The translation is achieved by application of a \emph{dictionary},
regrouping some classical combinatorial operations.
Consider two labeled combinatorial families $\mA$ and $\mB$
with generating functions $A(z)$ and $B(z)$.
\begin{itemize}
\item \textbf{Disjoint union.}
If $\mA$ and $\mB$ do not intersect, then the disjoint union
$\mC = \mA \uplus \mB$ has generating function
\[
    C(z) = A(z) + B(z).
\]
\item \textbf{Relabeled Cartesian product.}
The set of all relabeled pairs of objects $(a,b)$ with $a \in \mA$ and $b \in \mB$
is denoted by $\mC = \mA \times \mB$, and has generating function
\[
    C(z) = A(z) B(z).
\]
\item \textbf{Sequence.}
Let us assume that $\mA$ does not contain any empty object (\ie of size $0$).
The combinatorial family that contains
the relabeled sequences of an arbitrary number
of objects from $\mA$ is denoted by
\[
    \mC = \operatorname{Seq}(\mA) = \biguplus_{n \geq 0} \mA^n.
\]
Its generating function is
\[
    C(z) = \frac{1}{1-A(z)} = \sum_{n \geq 0} A(z)^n.
\]
\item \textbf{Set.}
Assuming again that $\mA$ does not contain any empty object,
the combinatorial family of relabeled sets of objects from $\mA$,
denoted by $\mC = \operatorname{Set}(\mA)$, has generating function
\[
    C(z) = e^{A(z)} = \sum_{n \geq 0} \frac{A(z)^n}{n!}.
\]
\item \textbf{Marking a node.}
The family of objects from $\mA$ where one node is distinguished
is denoted by $\mC = \mA'$, and has generating function
\[
    C(z) = z A'(z).
\]
\item \textbf{Composition}
The standard way to compose an object $a \in \mA$
with a sequence $(b_1, \ldots, b_{|a|})$ of $|a|$ objects from $\mB$
is to replace each node $i$ of $a$ with the object $b_i$,
and to relabel $(b_1, \ldots, b_{|a|})$ into $(c_1, \ldots, c_{|a|})$ in a way
that ensures that for all $i$,
the smallest label of $c_i$ is smaller than the smallest label of $c_{i+1}$.
The generating function of all compositions
of objects from $\mA$ by sequences of objects from $\mB$
is denoted by $\mC = \mA(\mB)$, and has generating function
\[
    C(z) = A(B(z)).
\]
\end{itemize}
The same operations can be naturally extended
to the multivariate case,
for example for multigraphs,
which are counted according both to their number of vertices and edges.
Those operations are illustrated in the following classical lemma,
which provides information on the generating function of rooted trees.

\begin{lemma} \label{th:cayley}
The generating function $T(z)$ of labeled rooted trees 
satisfies the two relations
\[
    T(z) = z e^{T(z)},
    \quad \text{ and } \quad
    z T'(z) = \frac{T(z)}{1-T(z)}.
\]
Its radius of convergence is $1/e$.
On its disk of convergence $|z| \leq 1/e$,
the maximum of $|T(z)|$ is $T(1/e) = 1$.
\end{lemma}

\begin{proof}
We sketch the proof, as a more complete version is provided
by Propositions~II.5 and~IV.5 from \cite{FS09}.
A rooted tree can be decomposed as a vertex,
the root, which has generating function $z$,
and a set of rooted trees, its children,
and this set has generating function $e^{T(z)}$.
Applying the symbolic method, this combinatorial description
translates into the generating function relation
\begin{equation} \label{eq:cayley}
    T(z) = z e^{T(z)}.
\end{equation}
The set of rooted trees where one vertex is marked
has generating function $z T'(z)$.
Any such tree $t$ can be uniquely decomposed
as a nonempty sequence of rooted trees,
which roots are on a path from the root of $t$
to the marked vertex.
The generating function of such a nonempty sequence is $\frac{z}{1-z}$,
so the generating function of a nonempty sequence
of rooted trees is $\frac{T(z)}{1-T(z)}$.
Thus, we have proven
\[
    z T'(z) = \frac{T(z)}{1-T(z)}.
\]
This can also be established by derivation of Equation~\eqref{eq:cayley}.
According to Cayley's Formula,
there are $n^{n-1}$ rooted tree with $n$ vertices, so
\[
    T(z) = \sum_{n \geq 1} n^{n-1} \frac{z^n}{n!}.
\]
This formula can be proven
using Lagrange inversion \cite[Proposition~I.5]{FS09} on Equation~\eqref{eq:cayley}.
Application of D'Alembert's criterion then provides
the radius of convergence $1/e$, because
\[
    \lim_{n \to +\infty}
    \frac{(n+1)^n}{(n+1)!}
    \frac{n!}{n^{n-1}}
    =
    \lim_{n \to +\infty}
    \left(1 + \frac{1}{n} \right)^n
    =
    e.
\]
Since $T(z)$ has nonnegative coefficients,
its absolute value reaches its maximum on any disk $|z| \leq \zeta$
where it is defined at the real positive value $\zeta$.
At $1/e$, the limit value of $T(z)$ satisfies Equation~\eqref{eq:cayley}
\[
    T(1/e) = (1/e) e^{T(1/e)},
\]
which implies $T(1/e) = 1$.
\end{proof}

\section{Models and outline of the method}
\label{sec:models}
We introduce the notations used throughout the article,
the classical graph model, and a multigraph model,
better suited for generating function manipulations.
The link between those two models is established in Lemma~\ref{th:multigraphs_to_graphs}.
Finally, we sketch the main steps for deriving
the asymptotic expansion of connected graphs
with $n$ vertices and excess $k$ growing linearly with $n$.

    \subsection{Notations}

A \emph{multiset} is an unordered collection of objects,
where repetitions are allowed.
\emph{Sets}, or \emph{families}, are then multisets without repetitions.
A \emph{sequence}, or \emph{tuple}, is an ordered multiset.
We use the parenthesis notation $(u_1, \ldots, u_n)$ for sequences, 
and the brace notation $\{u_1, \ldots, u_n\}$ for sets and multisets.
The cardinality of a set or multiset $S$ is denoted by $|S|$.
The double factorial notation for odd numbers stands for
\[
  (2k-1)!! = \frac{(2k)!}{2^k k!}.
\]
Given two positive values $\zeta$, $\lambda$,
the closed torus of radii $(\zeta, \lambda)$
denotes the set of pairs of complex numbers
\[
    \{(z,x) \in \complex^2\ |\ |z| \leq \zeta,\ |x| \leq \lambda\}.
\]
The $n$th coefficient in the Taylor expansion of $A(z)$ at $z=0$
is denoted by $[z^n] A(z)$, so that
\[
    [z^n] \sum_{m \geq 0} a_m z^m = a_n.
\]
The derivative of the function $f$ with respect to the variable $x$ is denoted by $\partial_x f(x)$,
or by $f'(x)$ when there is no ambiguity about the variable.
Most of the series we will manipulate have nonnegative coefficients.
The following simple classical lemma
provides a bound on the coefficients of such series.

\begin{lemma} \label{th:saddle_point_bound}
Consider a series $A(z)$ with nonnegative coefficients
and a positive radius of convergence $r$,
a positive value $\zeta < r$,
and a nonnegative integer $n$,
then the $n$th coefficient of $A(z)$ is bounded by
\[
    [z^n] A(z) \leq \frac{A(\zeta)}{\zeta^n}.
\]
\end{lemma}

\begin{proof}
Since the coefficients of $A(z)$ are nonnegative
and $\zeta$ is positive, we have
\[
    A(\zeta)
    =
    \sum_{m \geq 0} [z^m] A(z) \zeta^m
    \geq
    [z^n] A(z) \zeta^n.
\]
The result is obtained after dividing by $\zeta^n$.
\end{proof}

The value of $\zeta$ that provides the best bound
is the one that minimizes $A(\zeta) / \zeta^n$,
and this point is called the \emph{saddle-point}.
More details are available in \cite{FS09}.

    \subsection{Graph model} \label{sec:def_graphs}

We consider in this article the classical model of graphs, 
\aka \emph{simple graphs},
with labeled vertices and unlabeled unoriented edges.
All edges are distinct and no edge links a vertex to itself.
As always in analytic combinatorics and species theory,
the labels are distinct elements that belong to a totally ordered set.
When counting labeled objects (here, graphs), we always assume that 
the labels are consecutive integers starting at $1$.
Another formulation is that we consider two objects as equivalent 
if there exists an increasing relabeling sending one to the other.
We naturally adopt for graphs generating functions
exponential with respect to the number of vertices,
and ordinary with respect to the number of edges
(see \cite{FS09}, or \cite{BLL97}).

\begin{definition} \label{def:graphs}
A graph $G$ is a pair $(V(G), E(G))$,
where $V(G)$ is the labeled set of vertices,
and $E(G)$ is the set of edges.
Each edge is a set of two vertices from $V(G)$.
The number of vertices (resp.~of edges) is $n(G) = |V(G)|$
(resp. $m(G) = |E(G)|$).
The \emph{excess} $k(G)$ is defined as $m(G) - n(G)$.
The number of graphs with $n$ vertices and excess $k$
(hence with $n+k$ edges)
in a graph family $\mF$ is denoted by $F_{n,k}$.
The generating function of $\mF$ is
\[
    F(z,w)
    =
    \sum_{n,m} F_{n,m-n} w^m \frac{z^n}{n!}
    =
    \sum_{G \in \mF}
    w^{m(G)}
    \frac{z^{n(G)}}{n(G)!}.
\]
\end{definition}

A graph is said to be \emph{positive}
if all its components have a positive excess,
\ie are neither trees nor unicycles.
The set of positive graphs from a family $\mF$
is denoted by $\mF^{>0}$.

        \subsection{Multigraph model} \label{sec:def_multigraphs}

As already observed by \cite{FKP89,JKLP93},
\emph{multigraphs} are better suited for generating function manipulations than graphs.
We use the model of \cite{EdPCGGR17}, 
distinct but related to the one used by \cite{FKP89, JKLP93},
and recall the link between the generating functions of graphs 
and multigraphs in Lemma~\ref{th:multigraphs_to_graphs}.


The difference between graphs and multigraphs
is that multigraphs have labeled and oriented edges,
and are permitted loops and multiple edges.
Since vertices and edges are labeled, 
we choose exponential generating functions with respect to both quantities.
Furthermore, a weight $1/2$ is assigned to each edge,
for a reason that will become clear in Lemma~\ref{th:multigraphs_to_graphs}.

\begin{definition} \label{def:multigraphs}
A multigraph $G$ is a pair $(V(G), E(G))$,
where $V(G)$ is the set of labeled vertices,
and $E(G)$ is the set of labeled edges 
(the edge labels are independent from the vertex labels).
Each edge is a triplet $(v,w,e)$,
where $v$, $w$ are vertices,
and $e$ is the label of the edge.
The number of vertices (resp.~number of edges, excess) is $n(G) = |V(G)|$
(resp. $m(G) = |E(G)|$, $k(G) = m(G) - n(G)$).
The set of all multigraphs is denoted by $\mg$.
The number of multigraphs with $n$ vertices and excess $k$
in a multigraph family $\mF$ is denoted by $F_{n,k}$.
The generating function of $\mF$ is
\[
    F(z,w)
    =
    \sum_{n,m} F_{n,m-n} \frac{w^m}{2^m}{m!} \frac{z^n}{n!}
    =
    \sum_{G \in \mF}
    \frac{w^{m(G)}}{2^{m(G)} m(G)!}
    \frac{z^{n(G)}}{n(G)!}.
\]
\end{definition}

In the following, it will always be clear from the context
whether $\mF$ is a graph family or a multigraph family,
and thus whether $F(z,w)$ is defined using the convention
of Definitions~\ref{def:graphs} or~\ref{def:multigraphs}.
As a consequence of the definition,
the generating function of an edge is $w/2$,
while $w$ is the generating function of an edge
that can be oriented in both directions.
%
The generating function of all multigraphs is
\[
    \mg(z,w) = \sum_{n \geq 0} e^{w n^2/2} \frac{z^n}{n!},
\]
because a multigraph on $n$ vertices
is a set of labeled edges,
each chosen among a set of $n^2$ possibilities.
Definition~\ref{def:loops_double_edges} and Figure~\ref{fig:multigraphs_to_graphs}
present examples of multigraphs.

A multigraph is said to be \emph{positive}
if all its components have a positive excess,
\ie are neither trees nor unicycles.
The set of positive multigraphs from a family $\mF$
is denoted by $\mF^{>0}$.

    \paragraph{Difference with the previous model.}

\cite{FKP89} and \cite{JKLP93} defined multigraphs as
graphs where loops and multiple edges are allowed
(\ie labeled vertices, but unlabeled unoriented edges).
They counted multigraphs with a weight, the \emph{compensation factor},
and called \emph{number of multigraphs in the family $\mF$}
the sum of those weights (although it needed not be an integer).
Specifically, given a multigraph $G$ with $m$ edges,
its weight was defined as the number of different ways
to orient and label its edges, divided by $2^m m!$.

This setting leads to the same generating functions as us.
However, our definition brings two improvements.
First, those artificial weights are avoided.
Secondly, more combinatorial operations
translate into generating function relations
in the exponential setting than in the ordinary one.

    \paragraph{Link between graphs and multigraphs.}

A major difference between graphs and multigraphs
is the possibility of loops and multiple edges.

\begin{definition} \label{def:loops_double_edges}
A \emph{loop} (resp.\ \emph{double edge}) of a multigraph $G$ is a subgraph $(V, E)$ 
(\ie $V \subset V(G)$ and $E \subset E(G)$)
isomorphic to the following left multigraph (resp.\ to one of the following right multigraphs).
\begin{center}
\includegraphics[scale=1.]{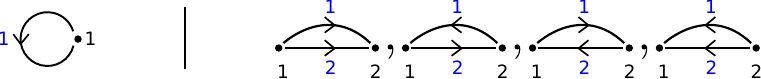}
\end{center}
The set of loops and double edges of a multigraph $G$ is denoted by $\LD(G)$,
and its cardinality by $\ld(G)$.
\end{definition}

In particular, a multigraph that has no double edge contains no multiple edge.
Multigraphs are better suited for generating function manipulations than graphs.
However, we aim at deriving results on the graph model,
since it has been adopted both by the graph theory and the combinatorics communities.
The following lemma, illustrated in Figure~\ref{fig:multigraphs_to_graphs}, 
links the generating functions of both models.

\begin{lemma} \label{th:multigraphs_to_graphs}
Let $\mgsimple$ denote the family of multigraphs
that contain neither loops nor double edges,
and $p$ the projection 
from $\mgsimple$ to the set $\sg$ of graphs,
that erases the edge labels and orientations,
as illustrated in Figure~\ref{fig:multigraphs_to_graphs}.
Let $\mF$ denote a subfamily of $\mgsimple$,
stable by edge relabeling and change of orientations.
Then there exists a family $\mH$ of graphs such that $p^{-1}(\mH) = \mF$.
Furthermore, the generating functions of $\mF$ and $\mH$,
with the respective conventions of multigraphs and graphs,
are equal
\[
    \sum_{G \in \mF}
    \frac{w^{m(G)}}{2^{m(G)} m(G)!}
    \frac{z^{n(G)}}{n(G)!}
    =
    \sum_{G \in \mH}
    w^{m(G)}
    \frac{z^{n(G)}}{n(G)!}.
\]
\end{lemma}

\begin{proof}
Consider a graph $G$ from $\mH$ that contains $m$ edges.
The edges of $G$ can be labeled
and oriented in $2^{m} m!$ different ways,
and $\mF$ is stable by edge relabeling and change of orientation,
so the set $p^{-1}(\{G\})$ of multigraphs from $\mF$
sent by $p$ to $G$ has cardinality $2^{m} m!$.
In the multigraph generating function of $\mF$,
let us group the multigraphs sent by $p$ to the same graph.
Each group corresponding to a graph with $m$ edges
has cardinality $2^m m!$, so
\[
    \sum_{G \in \mF}
    \frac{w^{m(G)}}{2^{m(G)} m(G)!}
    \frac{z^{n(G)}}{n(G)!}
    =
    \sum_{G \in \mH}
    2^{m(G)} m(G)!
    \frac{w^{m(G)}}{2^{m(G)} m(G)!}
    \frac{z^{n(G)}}{n(G)!}
    =
    \sum_{G \in \mH}
    w^{m(G)}
    \frac{z^{n(G)}}{n(G)!}.
\]
\end{proof}

\begin{figure}[htbp]
  \begin{center}
    \includegraphics[scale=0.9]{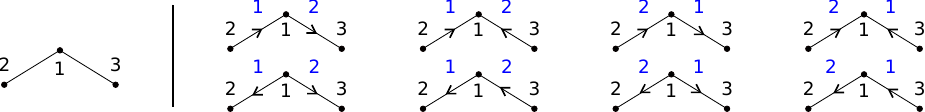}
    \caption{A graph $G$ and the set $\mF$ of multigraphs 
      sent by $p$ (defined in Lemma~\ref{th:multigraphs_to_graphs}) to the graph $G$.
      The generating function of $\{G\}$ is $w^2 \frac{z^3}{3!}$, the generating function of $\mF$ is $8 \frac{w^2}{2^2 2!} \frac{z^3}{3!}$.
      As stated by Lemma~\ref{th:multigraphs_to_graphs}, those generating functions are equal.}
    \label{fig:multigraphs_to_graphs}
  \end{center}
\end{figure}

In some graph families,
the number of edges of a graph
depends only of the number of vertices.
This is the case for trees and cycles,
since a tree with $n$ vertices contains $n-1$ edges,
and a cycle with $n$ vertices contains $n$ edges.
This is more generally true for the graphs of excess $k$
in a graph family $\mF$:
each such graph with $n$ vertices contains $n+k$ edges.
In such cases, we use univariate generating functions for the families,
simply replacing the variable $w$ with $1$
\[
    F(z) = F(z,1).
\]
We use the same convention for multigraph families of fixed excess.

\begin{lemma} \label{th:fixed_excess}
Consider a graph or multigraph family $\mF$
with generating function $F(z,w)$.
The generating function $F_k(z)$
of the graphs (or multigraphs) from $\mF$
of excess $k$ is equal to
\[
    F_k(z) = [y^k] F(z/y,y),
\]
and the generating function of $\mF$
is expressed using $F_k(z)$ as
\[
    F(z,w) = \sum_{k \in \integers} F_k(zw) w^k.
\]
Depending on whether $\mF$ is a graph or multigraph family,
the number of graphs (\resp multigraphs) in $\mF$
with $n$ vertices and excess $k$ is equal to
\[
    F_{n,k} = n! [z^n] F_k(z)
    \quad \text{ or to } \quad
    F_{n,k} = n! 2^{n+k} (n+k)! [z^n] F_k(z).
\]
\end{lemma}

\begin{proof}
We present the proof for a multigraph family $\mF$.
The proof for simple graphs is identical.
According to Definition~\ref{def:multigraphs},
the generating function of $F(z/y,y)$ is equal to
\[
    F(z/y,y) =
    \sum_{G \in \mF}
    \frac{y^{m(G)-n(G)}}{2^{m(G)} m(G)!}
    \frac{z^{n(G)}}{n(G)!}.
\]
Extracting the coefficient $[y^k]$
is thus equivalent to restricting the domain of summation
to the multigraphs of excess $k$ in $\mF$, so
\[
    F_k(z) = [y^k] F(z/y,y).
\]
This relation implies
\[
    F_k(z w) w^k = [y^k] F(z/y, w y).
\]
Formally summing over $k$, we obtain
\[
    \sum_{k \in \integers} F_k(z w) w^k = F(z, w).
\]
Finally, by Definition~\ref{def:multigraphs},
the number of multigraphs with $n$ vertices and excess $k$ in $\mF$ is
\[
    F_{n,k} = n! 2^{n+k} (n+k)! [z^n w^{n+k}] F(z,w),
\]
which is equal to $n! 2^{n+k} (n+k)! [z^n] F_k(z)$
\end{proof}

        \subsection{Outline of the method} \label{sec:outline}

A proof can be represented as a pyramid of statements,
with the main result standing at the top.
In this section, we present the main steps of our proof from top to bottom,
motivating the introduction of the combinatorial objects one by one.
The actual proof is presented the other way around.

    \paragraph{A first exact expression for the number of connected graphs.}

With the conventions of Definition~\ref{def:graphs},
the generating function of all graphs is
\[
  \sg(z,w) = 
  \sum_{n \geq 0}
  (1+w)^{\binom{n}{2}}
  \frac{z^n}{n!},
\]
because a graph with $n$ vertices
has $\binom{n}{2}$ possible edges,
each of which is either absent or present in the graph
(hence there are $\binom{n (n-1) / 2}{m}$ graphs
with $n$ vertices and $m$ edges).
Since a graph is a set of connected graphs,
the generating function of connected graphs $\csg(z,w)$
is characterized by the relation
\[
  \sg(z,w) = e^{\csg(z,w)}.
\]
Taking the logarithm, we obtain the classical closed form
for the generating function of connected graphs
\[
    \csg(z,w) =
    \log \bigg(
    \sum_{n \geq 0}
    (1+w)^{\binom{n}{2}}
    \frac{z^n}{n!}
    \bigg).
\]
Observe that the argument of the logarithm
is a series with a zero radius of convergence.
Therefore, we cannot use any analytic property of the logarithm,
and the only way to treat this expression seems to be
to expand it as a series
\begin{equation} \label{eq:first_exact_csg}
    \csg(z,w) =
    \sum_{q \geq 1}
    \frac{(-1)^{q+1}}{q}
    \bigg(
    \sum_{n \geq 1}
    (1+w)^{\binom{n}{2}}
    \frac{z^n}{n!}
    \bigg)^q.
\end{equation}
This expression was the starting point of the analysis of \cite{FSS04},
who worked on connected graphs with fixed excess.
If we extract the coefficient $n! [z^n w^m]$,
we obtain an exact expression for the number
of connected graphs with $n$ vertices and $m$ edges
\[
    \csg_{n,m-n} =
    \sum_{q = 1}^n
    \frac{(-1)^{q+1}}{q}
    \sum_{\substack{n_1 + \cdots + n_q = n\\ \forall j,\ n_j \geq 1}}
    \binom{n}{n_1, \ldots, n_q}
    \binom{\sum_{j=1}^q \binom{n_j}{2}}{m}.
\]
However, as already observed by those authors,
it is difficult to extract the asymptotics,
because of ``\emph{magical}'' cancellations in the coefficients.
In particular, the dominant contribution to the sum does not come
from the first value $q=1$,
because the summand is then the number of (non-empty) graphs
with $n$ vertices and $m$ edges.
Those graphs are indeed typically not connected,
as they contain many trees and unicycles
(\ie components of excess $-1$ or $0$, see \cite{ER60}).

    \paragraph{Connected graphs and positive graphs.}

Instead of working on this expression
using complicated analysis,
we will derive a different (although similar) expression,
where the dominant contribution is easier to locate.
The main idea, already applied by \cite{PW05},
is to consider the family $\sgpos$
of graphs without trees and unicycles.
We call them \emph{positive graphs},
as their components all have a positive excess.
Their generating function is derived in Proposition~\ref{th:sgpos}.
Using the fact that a positive graph of positive excess $k$
is a set of connected graphs with positive excess
which excesses sum to $k$,
we will obtain in Proposition~\ref{th:csg_exact}
the following expression for the number
of connected graphs with $n$ vertices and excess $k$
\[
    \csg_{n,k} = n! [z^n y^k]
    \log \bigg(
    1 + \sum_{\ell \geq 1} \sgpos_{\ell}(z) y^{\ell}
    \bigg).
\]
This expression looks similar to the previous one.
However, the dominant contribution now comes
from the first terms of the Taylor expansion of the logarithm.
We prove in Lemma~\ref{th:csg_bender_applied}
that the tail of the sum is indeed negligible.
A result on the first term $ n! [z^n] \sgpos_k(z)$
was already proven by\cite{ER60}:
when $k = \exactbigO(n)$, a positive graph
with $n$ vertices and excess $k$ is almost surely connected, which implies
\[
    \csg_{n,k} \sim \sgpos_{n,k}.
\]
Hence, the dominant asymptotics of connected graphs
and positive graphs are the same.
This property was at the foundation of the proof of \cite{PW05}.
This asymptotic relation is not precise enough for our purpose,
as we want to derive an arbitrary number of error terms.
A positive graph is connected with probability tending to $1$,
but to gain more information on the speed of convergence,
we need to consider also the less likely cases
where the positive graph is a set of connected graphs.
Intuitively, it seems clear that the most probable configuration
is that one of those connected graphs
has a large excess, while the others
have a small (constant) excess.
This motivates the derivation, in Proposition~\ref{th:sgpos},
of two expressions for the generating function
of positive graphs of excess $k$:
one suited for the case where $k$ goes to infinity with $n$,
the other for constant values of $k$.
Lemma~\ref{th:csg_bender_applied} translates this intuition
into error terms for $\csg_{n,k}$.

    \paragraph{Positive graphs and cores.}

\begin{figure}
\begin{center}
\includegraphics[scale=0.9]{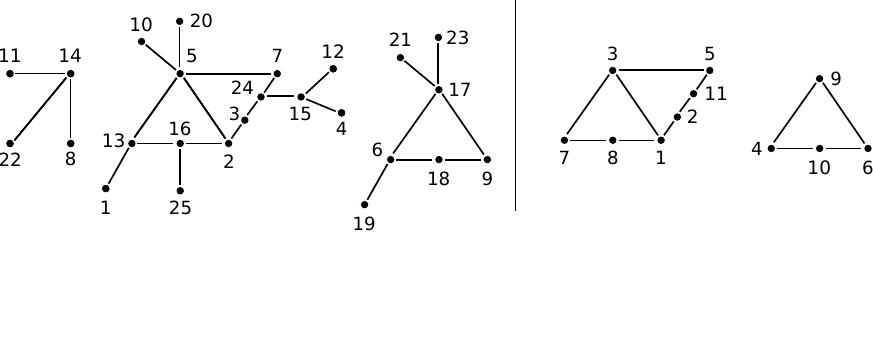}
\caption{A graph $G$ and its core $H$,
obtained after iteratively removing all vertices of degree $0$ and $1$.
Observe that the core contains no tree,
and that the graph $G$ without its tree component
can be obtained from the core $H$
by adding rooted trees to the vertices.}
\label{fig:core} 
\end{center}
\end{figure}

As already observed by \cite{W80} and \cite{PW05},
a convenient way to remove all trees from a graph $G$
is to remove iteratively all vertices of degree $0$ and $1$.
The graph $H$ obtained is then a graph
of minimum degree $2$, called a \emph{core}.
This process is illustrated in Figure~\ref{fig:core}.
A positive core is then a core where all components have a positive excess.
The only components of nonpositive excess with minimum degree at least $2$
are isolated cycles, which have excess $0$.
Reversely, any positive graph of excess $k$
is a core of excess $k$,
where isolated cycles are removed,
and where rooted trees are added to each vertex.
In Proposition~\ref{th:sgpos},
the symbolic method (see Section~\ref{sec:analytic_combinatorics})
is applied to translate this combinatorial description
into an expression for the generating function of positive graphs,
involving the generating function of cores.
Finally, We apply results of \cite{EdPR16} to express
the generating function of cores of a given excess,
in Proposition~\ref{th:core}.

When considering graphs with degree constraints,
such as cores,
it is more convenient to work with multigraphs,
where loops and multiple edges are allowed,
as was already observed in the configuration model from \cite{Wo78, Bo80}.

    \paragraph{From positive multigraphs to positive graphs.}

To analyze positive graphs instead of positive multigraphs,
we will apply Lemma~\ref{th:multigraphs_to_graphs}.
It requires to remove the loops and double edges
from the positive multigraphs.
This is achieved using an inclusion-exclusion technique
(more examples of application of this technique are provided by \cite{FS09}).
Since multigraphs are interesting in themselves,
and are simpler to analyze,
we will first derive in Section~\ref{sec:cmg}
the asymptotic expansion of connected multigraphs
with $n$ vertices and excess $k$ growing linearly with $n$.
This also provides an introduction to the more complex proof
for connected graphs, presented in Section~\ref{sec:csg}.

\section{Connected multigraphs with large excess}
\label{sec:cmg}

In this section, the asymptotic expansion of connected multigraphs
with $n$ vertices and excess $k$, growing linearly with $n$,
is derived.
The outline of the proof from Section~\ref{sec:outline}
motivates the introduction of multicores
(multigraphs with minimum degree at least $2$)
and positive multigraphs
(multigraphs where all components have a positive excess,
\ie containing neither trees nor unicycle components).
Multicores are particular cases
of multigraphs with constraints on their degrees.
The generating functions of such multigraphs
have been derived by \cite{EdPR16},
and the first part of the proof
of the following lemma relies on their work.
The next proposition has been obtained independently by \cite{ElieThesis} and \cite{Bo17}.
The last author applied it to obtain a complete asymptotic expansion
of multicores with weights depending on their vertices degrees.

\begin{proposition} \label{th:deg_constraints}
Given a subset $D$ (finite or infinite) of $\integers_{\geq 2}$,
and its generating function
\[
    \Delta(x) = \sum_{d \in D} \frac{x^d}{d!},
\]
then the generating function of multigraphs of excess $k$
where all vertices have their degree in $D$ is
\[
    \mg^{(D)}_k(z) =
    (2k-1)!! [x^{2k}]
    \frac{1}{\big(1 - z \frac{\Delta(x)}{x^2/2}\big)^{k+1/2}}.
\]
\end{proposition}

\begin{proof}
Let $\mg^{(D)}$ denote the set of multigraphs
where all vertices have their degree in $D$.
We first recall the proof for the expression
of the bivariate generating function $\mg^{(D)}(z,w)$,
obtained by \cite{EdPR16},
then extract the formula for the generating function $\mg^{(D)}_k(z)$
of multigraphs from $\mg^{(D)}$ of excess $k$.
Let us consider a multigraph $G$.
As illustrated in Figure~\ref{fig:half_edges},
each edge $(v,w,\ell) \in E(G)$
of label $\ell$ and linking the vertex $v$ to the vertex $w$
can be replaced by two half-edges,
one attached to $v$ and labeled $2\ell-1$,
the other attached to $w$ and labeled $2\ell$.
The size of the set of half-edges attached to a vertex
is then equal to its degree.
If $G$ is in $\mg^{(D)}$, then the sizes of those sets are in $D$.
Therefore, the multigraph $G$ is now represented
as a set of vertices,
each coming with a set of half-edges of size in $D$,
and the total number of half-edges is twice the number of edges of $G$.
The symbolic method (see Section~\ref{sec:analytic_combinatorics})
translates this combinatorial description
into the following generating function expression
for the generating function of $\mg^{(D)}$
\[
    \mg^{(D)}(z,w) =
    \sum_{m \geq 0} (2m)! [x^{2m}] e^{z \Delta(x)} \frac{w^m}{2^m m!},
\]
where
\begin{itemize}
\item
the variable $x$ is used to mark the half-edges,
\item
$\Delta(x)$ is the generating function of sets of size in $D$,
\item
$e^{z \Delta(x)}$ is the generating function of sets of vertices,
each coming with a number of labeled half-edges that lies in $D$,
\item
the coefficient extraction $(2m)! [x^{2m}]$ fixes the number of half-edges to $2m$,
\item
the product by $\frac{w^m}{2^m m!}$ represents the addition of $m$ edges,
to replace the $2m$ half-edges,
\item
the sum over $m$ corresponds to the fact
that $\mg^{(D)}$ is the disjoint union, for all $m \geq 0$,
of the subsets of $\mg^{(D)}$
of multigraphs having exactly $m$ edges.
\end{itemize}
In this expression, after developing the exponential as a sum over $n$,
applying the change of variable $m \mapsto k+n$,
and replacing $[x^{2m}]$ with $[x^{2k}] x^{-2n}$, we obtain
\[
    \mg^{(D)}(z,w) =
    \sum_{k \geq 0}
    [x^{2k}]
    \sum_{n \geq 0}
    \frac{(2(k+n))!}{2^{k+n} (k+n)!}
    \frac{\left( z w \frac{\Delta(x)}{x^2} \right)^n}{n!}
    w^k.
\]
The sum over $n$ is replaced by its closed form
\[
    \mg^{(D)}(z,w) =
    \sum_{k \geq 0}
    [x^{2k}]
    \frac{(2k-1)!!}
    {\left( 1 - z w \frac{\Delta(x)}{x^2/2} \right)^{k + 1/2}}
    w^k.
\]
The generating function of multigraphs from $\mg^{(D)}$
of excess $k$ is then (see Lemma~\ref{th:fixed_excess})
\[
    \mg^{(D)}_k(z) = [y^k] \mg^{(D)}(z/y,y) =
    (2k-1)!!
    [x^{2k}]
    \frac{1}
    {\left( 1 - z \frac{\Delta(x)}{x^2/2} \right)^{k + 1/2}}.
\]
\end{proof}

\begin{figure}[htbp]
  \begin{center}
    \includegraphics[scale=0.9]{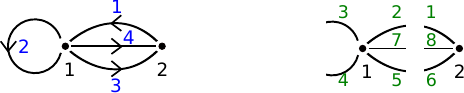}
    \caption{Left, a multigraph.
        Right, the multigraph obtained after cutting each edge
        into two labeled half-edges, as explained in the proof
        of Proposition~\ref{th:deg_constraints}.}
    \label{fig:half_edges}
  \end{center}
\end{figure}

The ``\emph{half-edges}'' idea is reminiscent
of the \emph{configuration model}, introduced by \cite{Bo80} and \cite{Wo78}.
During the proof, we have derived a particular case
of the following general formula
\begin{equation} \label{eq:m_to_k} 
    \sum_{m \geq 0}
    (2m-1)!! [x^{2m}]
    A(z,w,x) e^{z \Delta(x)}
    w^m
    =
    \sum_{k \geq 0}
    (2k-1)!! [x^{2k}]
    \frac{A(z,w,x)}{\big(1 - z w \frac{\Delta(x)}{x^2/2}\big)^{k+1/2}}
    w^k,
\end{equation}
valid for any formal power series $A(z,w,x)$ and $\Delta(x)$
when $\Delta(0) = \Delta'(0) = 0$.
It will be used again in Proposition~\ref{th:core}.
A positive multicore is a multigraph with minimum degree at least $2$,
where all connected components have positive excess.
We now apply the previous proposition to express
their generating function.

\begin{proposition} \label{th:mcorepos}
The generating functions of positive multicores of excess $k$ is
\[
    \mcorepos_k(z)
    =
    (2k-1)!! [x^{2k}] \frac{\sqrt{1-z}}{\big(1 - z \frac{e^x-1-x}{x^2/2}\big)^{k+1/2}}.
\]
\end{proposition}

\begin{proof}
Multicores are multigraphs with minimum degree at least $2$,
so, with the notations of Proposition~\ref{th:deg_constraints},
$\mcore = \mg^{(\integers_{\geq 2})}$.
Since
\[
    \sum_{d \in \integers_{\geq 2}} \frac{x^d}{d!}
    = e^x-1-x,
\]
The proposition provides the following expression
for the generating function of multicores of excess $k$
\[
    \mcore_k(z) =
    (2k-1)!! [x^{2k}] \frac{1}{\big(1 - z \frac{e^x-1-x}{x^2/2}\big)^{k+1/2}}.
\]
The only components of a multicore with nonpositive excess
are isolated cycles.
Thus, any multicore has a unique decomposition as
a positive multicore, and a set of isolated cycles.
Since those isolated cycles have excess $0$,
the multicore and the positive multicore have the same excess, so
\[
    \mcore_k(z) = \mcorepos_k(z) e^{\mcycle(z)},
\]
where $\mcycle(z) = \mcycle_0(z)$ denotes
the univariate generating function of (multigraph) isolated cycles.
There are $(n-1)!$ ways to label the vertices of a cycle of length $n$,
and $2^n n!$ ways to label and orient its edges,
so the bivariate generating function of isolated cycles is equal to
\[
    \mcycle(z,w)
    =
    \sum_{n \geq 0} (n-1)! 2^n n! \frac{w^n}{2^n n!} \frac{z^n}{n!}
    =
    \frac{1}{2} \log \left( \frac{1}{1-zw} \right),
\]
and the univariate generating function is
\[
    \mcycle(z) = [y^0] \mcycle(z/y,y) = \frac{1}{2} \log \left( \frac{1}{1-z} \right).
\]
Combining the last equations, we obtain
\[
    \mcorepos_k(z)
    =
    \mcore_k(z) e^{-\frac{1}{2} \log \left( \frac{1}{1-z} \right)}
    =
    (2k-1)!! [x^{2k}] \frac{\sqrt{1-z}}{\big(1 - z \frac{e^x-1-x}{x^2/2}\big)^{k+1/2}}.
\]
\end{proof}

In the next proposition,
the generating function of positive multicores
is used to express the generating function
of positive multigraphs of a given excess.
Two formulas are provided:
the first one is suited for fixed values of the excess $k$,
as it requires the computation of a polynomial of degree $2k$,
while the other one is suited to excesses
going to infinity with the number of vertices.

\begin{proposition} \label{th:mgpos}
The generating function of positive multigraphs of excess $k$
has the following two expressions
\begin{align*}
    \mgpos_k(z)
    &=
    \frac{\mk_k(T(z))}{(1-T(z))^{3k}},
    \\
    \mgpos_k(z)
    &=
    (2k-1)!! [x^{2k}]
    \sqrt{1-T(z)} B(z,x)^{k+1/2},
\end{align*}
where $\mk_k(T)$ is a polynomial of degree $2k$,
and the expressions of $\mk_k(T)$ and $B(z,x)$ are
\begin{align*}
    \mk_k(T) &=
    (2k-1)!! (1-T)^{2k} [x^{2k}]
    \left( 1 - \frac{T}{1-T} \frac{e^x-1-x-x^2/2}{x^2/2} \right)^{-k-1/2},\\
    B(z,x) &= \left(1 - T(z) \frac{e^x-1-x}{x^2/2}\right)^{-1}.
\end{align*}
\end{proposition}

\begin{proof}
We first prove the second expression of the proposition.
Iteratively removing the vertices of degree $1$
from a positive multigraph $G$
reduces it to a positive multicore $H$.
Observe that the removed vertices form trees,
rooted at the vertices of the multicore,
and that $G$ and $H$ share the same excess.
In fact, any positive multigraph of excess $k$
has a unique decomposition
as a positive multicore of excess $k$,
where a rooted tree is planted at each vertex.
This implies the following generating function relation
\[
    \mgpos_k(z) = \mcorepos_k(T(z)),
\]
where $T(z)$ is the univariate generating function of rooted trees
(see Lemma~\ref{th:cayley}).
This implies, applying Proposition~\ref{th:mcorepos} to express $\mcorepos_k(z)$,
\[
    \mgpos_k(z)
    = (2k-1)!! [x^{2k}] 
    \frac{\sqrt{1-T(z)}}{\big(1 - T(z) \frac{e^x-1-x}{x^2/2}\big)^{k+1/2}}.
\]
To derive the second result of the proposition,
we start with the previous expression of $\mgpos_k(z)$,
where $1-T(z) \frac{e^x-1-x}{x^2/2}$ is replaced by
$(1-T(z)) \left( 1 - \frac{T(z)}{1-T(z)} \frac{e^x-1-x-x^2/2}{x^2/2} \right)$
\[
    \mgpos_k(z) =
    (2k-1)!!
    [x^{2k}]
    \frac{1}{(1-T(z))^{k} \big(1- \frac{T(z)}{1-T(z)} \frac{e^x-1-x-x^2/2}{x^2/2}\big)^{k+1/2}}.
\]
This expression is rewritten
\[
    \mgpos_k(z) = \frac{\mk_k(T(z))}{(1-T(z))^{3k}},
\]
where
\[
    \mk_k(T) =
    (2k-1)!! (1-T)^{2k} [x^{2k}]
    \frac{1}{\big(1- \frac{T}{1-T} \frac{e^x-1-x-x^2/2}{x^2/2}\big)^{k+1/2}}.
\]
Since $\frac{e^x-1-x-x^2/2}{x^2/2}$ has valuation $1$,
$
    [x^{2k}] \big(1- Y \frac{e^x-1-x-x^2/2}{x^2/2}\big)^{-k-1/2}
$
is a polynomial in $Y$ of degree exactly $2k$.
This implies that $\mk_k(T)$ is a polynomial of degree $2k$.
\end{proof}

The generating function of positive multigraphs of excess $k$
has been proven to be a rational function in $T(z)$,
using generating function calculations.
A combinatorial interpretation of this fact is provided in Section~\ref{sec:kernels}.
It is based on a simpler version of a proof of \cite{W80},
stating that the generating function of connected graphs of excess $k$
is a rational function in $T(z)$.
For any fixed positive $k$,
the first expression of $\mgpos_k(z)$
is amenable to asymptotic analysis
using a singularity analysis (see \cite{FS09}).
When $k$ grows linearly with the number of vertices,
the asymptotics is extracted applying
to the second expression the following saddle-point lemma,
proven at the end of Appendix~\ref{sec:saddle_point}.

\begin{lemma} \label{th:mgpos_asymptotics}
Consider a positive integer $d$,
and integers $n$ and $k$ going to infinity such that $\alpha := k/n$ stays in a closed interval $\mK$ of $\reals_{>0}$.
Let $\lambda$ and $\zeta$ denote the unique positive solutions of the equations
\[
    \frac{\lambda}{2} \frac{e^{\lambda}+1}{e^{\lambda}-1} = \alpha + 1,
    \quad
    T(\zeta) = \alpha+1-\frac{\lambda}{2},
\]
$A(z,x)$ a bivariate function 
analytic on the closed torus of radii $(\zeta, \lambda)$,
and
\[
    B(z,x) = \frac{1}{1 - T(z) \frac{e^x-1-x}{x^2/2}}.
\]
Then the following asymptotic expansion holds uniformly for $k/n$ in $\mK$
\[
    [z^n x^{2k}]
    A(z,x)
    B(z,x)^k
    =
    \frac{B(\zeta, \lambda)^k}{2 \pi k \zeta^n \lambda^{2k}}
    \bigg(
    \sum_{r=0}^{d-1}
    c_r k^{-r}
    + \bigO(k^{-d})
    \bigg),
\]
with
\[
    B(\zeta, \lambda) = \frac{\lambda}{2 \alpha},
    \quad
    \zeta = e^{-\alpha-1} \sqrt{(\alpha+1)^2 - (\lambda/2)^2},
    \quad
    c_0 =
    A(\zeta, \lambda)
    \sqrt{\frac{\alpha^3}{\lambda}
    \frac{\lambda/2-\alpha}{(\lambda/2-\alpha)^2+\lambda \alpha}}
\]
and the formula for the other $(c_r)$ is as follows.
There is a biholomorphic function $\psi(x,y) = (\psi_1(x,y), \psi_2(x,y))$
sending $(0,0)$ to $(0,0)$
such that
\[
    - \log \left( \frac{B(\zeta e^{i \psi_1(x,y)}, \lambda e^{i \psi_2(x,y)})}{B(\zeta,\lambda)} \right)
    + i \left( \frac{\psi_1(x,y)}{\alpha} + 2 \psi_2(x,y) \right) =
    \frac{x^2+y^2}{2}.
\]
Its Jacobian matrix is denoted by $J_\psi(x,y)$, and we have
\[
    c_r =
    \sum_{t=0}^r
    (2t-1)!! (2(r-t)-1)!! [x^{2t} y^{2(r-t)}]
    A(\zeta e^{i \psi_1(x,y)}, \lambda e^{i \psi_2(x,y)}) \det(J_{\psi}(x,y)).
\]
Each of $c_r$, $\lambda$ and $\zeta$ is a smooth function of $k/n$.
\end{lemma}

According to Lemma~\ref{th:fixed_excess},
the number of positive multigraphs
with $n$ vertices and excess $k$ is
\[
    \mgpos_{n,k} = n! 2^{n+k} (n+k)! [z^n] \mgpos_k(z).
\]
As a corollary of Proposition~\ref{th:mgpos}
and Lemma~\ref{th:mgpos_asymptotics}
applied with $A(z,x) = \sqrt{(1-T(z)) B(z,x)}$,
the number of positive multigraphs
with $n$ vertices and excess $k$ with $k/n$ in a closed interval of $\reals_{>0}$
has asymptotics
\[
    \mgpos_{n,k}
    \sim
    n! 2^{n+k} (n+k)!
    (2k-1)!!
    \sqrt{(1-T(\zeta))
    \frac{\alpha^3}{\lambda}
    \frac{\lambda/2-\alpha}{(\lambda/2-\alpha)^2+\lambda \alpha}}
    \frac{B(\zeta, \lambda)^{k+1/2}}{2 \pi k \zeta^n \lambda^{2k}}
\]
(using the notations of Lemma~\ref{th:mgpos_asymptotics}).
According to \cite{ER60}, when $k = \exactbigO(n)$,
a positive graph is almost surely connected,
a property used by \cite{PW05}.
The same holds for multigraphs, so
\[
    \cmg_{n,k} \sim \mgpos_{n,k}
\]
(we will not use this property, so it is stated here informally).
Thus, if we were only interested in the asymptotics of connected multigraphs,
we could stop our analysis here.
Deriving an arbitrary number of error terms requires more work.
The next proposition recalls the link between the generating functions
of connected multigraphs, and positive multigraphs.

\begin{proposition} \label{th:cmg_exact}
For any positive $k$, 
the generating function of connected multigraphs of excess $k$
is equal to
\[
    \cmg_k(z) = [y^k] \log \bigg( 1 + \sum_{\ell \geq 1} \mgpos_{\ell}(z) y^{\ell} \bigg).
\]
\end{proposition}

\begin{proof}
A positive multigraph is a set of connected multigraphs of positive excess, so
\[
    \mgpos(z,w) = e^{\cmg^{>0}(z,w)},
\]
where $\cmg^{>0}(z,w)$ denotes the generating function
of connected multigraphs with a positive excess.
Taking the logarithm, replacing $z$ with $z/y$ and $w$ with $y$,
and extracting the coefficient $[y^k]$ for $k \geq 1$,
we obtain the following expression for the generating function
of connected multigraphs of excess $k$ (see Lemma~\ref{th:fixed_excess})
\[
    \cmg_k(z) = [y^k] \cmg(z/y,y) =
    [y^k] \log(\mgpos(z/y,y)).
\]
According to Lemma~\ref{th:fixed_excess}, we have
\[
    \mgpos(z/y,y) = \sum_{\ell \geq 0} \mgpos_{\ell}(z) y^{\ell}.
\]
The only positive multigraph of excess $0$ is the empty multigraph,
and its generating function is equal to $1$.
This is confirmed by Proposition~\ref{th:mgpos}
\[
    \mgpos_0(z)
    = (2\times 0 - 1)!! [x^0] \frac{\sqrt{1-T(z)}}{\big(1 - T(z) \frac{e^x-1-x}{x^2/2}\big)^{0 +1/2}}
    = 1.
\]
Thus, we have
\[
    \cmg_k(z) =
    [y^k] \log \bigg(1 + \sum_{\ell \geq 1} \mgpos_{\ell}(z) y^{\ell}\bigg).
\]
\end{proof}

Given the rapid growth of the asymptotics of $\mgpos_k(z)$
with respect to $k$ for any real $z$ in $]0,1/e[$
($1/e$ is the radius of convergence of $T(z)$, and hence of $\mgpos_k(z)$),
the work of \cite{Be75} comes to mind
to extract the asymptotic expansion of $\cmg_{n,k}$.
Informally, his Theorem~$1$ implies that
when the sequence $(g_k)_{k \geq 1}$ grows fast enough to infinity
and $f(y)$ has a nonzero radius of convergence, then
\[
    [y^k] f\bigg(\sum_{\ell \geq 1} g_{\ell} y^{\ell}\bigg) =
    \sum_{r=0}^d
    g_{k-r}
    [y^r] f'\bigg(\sum_{\ell \geq 1} g_{\ell} y^{\ell}\bigg)
    + \bigO(g_{k-d-1}).
\]
In this expression, observe that there is a finite number of summands,
indexed by $r$, and that
$[y^r] f'(\sum_{\ell \geq 1} g_{\ell} y^{\ell})$
is a finite sum of product of terms from $(g_{\ell})_{\ell \geq 1}$.
Thus, the asymptotics of the $r$th summand when $k$ is large
is driven by $g_{k-r}$.
To apply this theorem, we would set
\[
    f(y) = \log(1+y)
    \quad
    \text{ and }
    \quad
    g_k = \mgpos_k(z).
\]
Because we are dealing with this extra variable $z$,
our problem does not fit as it is in the theorems of \cite{Be75}.
Instead, we apply the following result.
Its proof is provided at the end of Appendix~\ref{sec:divergent_series},
and follows a recent extension of Bender's Theorem,
due to \cite{Bo16}.

\begin{lemma} \label{th:bender_like}
Consider a formal bivariate series
\[
    \sum_{\ell \geq 1} g_{\ell}(z) y^{\ell}
\]
with nonnegative coefficients,
and assume there exist positive constants $E$, $\beta$ and $\zeta$
such that when $\ell$ goes to infinity,
\[
    g_{\ell}(\zeta) = \bigO( E^{\ell} \Gamma(\ell + \beta) ),
\]
assuming that the radius of convergence of each $g_{\ell}(z)$ is greater than $\zeta$.
Let $f(z)$ be a function analytic at the origin,
then for any positive integer $d$, we have
\[
    [z^n y^k] f \bigg( \sum_{\ell \geq 1} g_{\ell}(z) y^{\ell} \bigg) =
    [z^n]
    \sum_{r=0}^{d}
    g_{k-r}(z)
    [y^r]
    f'\bigg( \sum_{\ell = 1}^r g_{\ell}(z) y^{\ell} \bigg)
    +
    \bigO
    \left(
    \frac{E^k}{\zeta^n}
    \Gamma(k-d-1+\beta)
    \right).    
\]
\end{lemma}

Applying the previous lemma to the expression of $\cmg_k(z)$
derived in Proposition~\ref{th:cmg_exact},
we obtain the following expression
for the number of connected multigraphs.

\begin{lemma} \label{th:bender_applied}
Consider a positive integer $d$,
and two integers $n$, $k$ going to infinity
such that $\alpha := k/n$ stays in a closed interval $\mK$ of $\reals_{>0}$.
Then the number of connected multigraphs
with $n$ vertices and excess $k$
has the following asymptotic expansion,
uniformly for $\alpha$ in $\mK$
\[
    \cmg_{n,k}
    =
    n! 2^{n+k} (n+k)! (2k-1)!!
    \bigg(
    \sum_{r=0}^{d}
    \frac{(2(k-r)-1)!!}{(2k-1)!!}
    [z^n x^{2k}]
    A_r(z,x)
    B(z,x)^k
    + \bigO \left(
    k^{-d-1}
    \frac{B(\zeta, \lambda)^k}{\zeta^n \lambda^{2k}}
    \right)
    \bigg),
\]
where the series $B(z,x)$ and $A_r(z,x)$ are equal to
\begin{align*}
    B(z,x) &= \left( 1 - T(z) \frac{e^x-1-x}{x^2/2} \right)^{-1},\\
    A_r(z,x) &= \frac{x^{2r} M_r(T(z)) B(z,x)^{-r+1/2}}{(1-T(z))^{3r-1/2}},\\
    M_r(T) &=
    [y^r] \bigg( 1 +
    \sum_{\ell = 1}^r \mk_{\ell}(T) y^{\ell}
    \bigg)^{-1},
\end{align*}
$M_r(T)$ is a polynomial of degree at most $2r$, and
the formula for the polynomial $\mk_{\ell}(T)$ is provided by Proposition~\ref{th:mgpos}.
\end{lemma}

\begin{proof}
According to Proposition~\ref{th:cmg_exact},
the number of connected multigraphs
with $n$ vertices and excess $k$ is
\[
    \cmg_{n,k}
    =
    n! 2^{n+k} (n+k)! [z^n] \log \bigg( 1 + \sum_{\ell \geq 1} \mgpos_{\ell}(z) \bigg).
\]
Proposition~\ref{th:mgpos} provides the following expression
for the generating function of positive multigraphs of excess $\ell$
\[
    \mgpos_{\ell}(z) =
    (2\ell - 1)!! [x^{2\ell}] \sqrt{1-T(z)} B(z,x)^{\ell + 1/2}.
\]
Using the relation
\[
    (2\ell - 1)!! = \frac{2^{\ell}}{\sqrt{\pi}} \Gamma(\ell + 1/2),
\]
the values $\zeta$, $\lambda$ from Lemma~\ref{th:mgpos_asymptotics},
and the bound from Lemma~\ref{th:saddle_point_bound}, we conclude
\[
    \mgpos_{\ell}(\zeta)
    \leq
    \frac{2^{\ell}}{\sqrt{\pi}} \Gamma(\ell + 1/2)
    \frac{\sqrt{1-T(\zeta)} B(\zeta, \lambda)^{\ell + 1/2}}{\lambda^{2 \ell}}.
\]
The hypothesis of Lemma~\ref{th:bender_like} are satisfied
with $g_{\ell}(z) = \mgpos_{\ell}(z)$,
$f(z) = \log(1+z)$, $\beta = 1/2$,
and $E = 2 B(\zeta, \lambda)/ \lambda^2$,
so $\cmg_{n,k}$ is equal to
\begin{equation} \label{eq:th:cmg_asymptotics}
    n! 2^{n+k} (n+k)!
    \bigg(
    \sum_{r=0}^{d}
    [z^n]
    \mgpos_{k-r}(z)
    [y^r]
    \bigg(
    1 + \sum_{\ell \geq 1} \mgpos_{\ell}(z) y^{\ell}
    \bigg)^{-1}
    + \bigO \left(
    (2(k-d)-3)!!
    \frac{B(\zeta, \lambda)^k}{\zeta^n \lambda^{2k}}
    \right)
    \bigg).
\end{equation}
Proposition~\ref{th:mgpos} provides the expression
\[
    \mgpos_{\ell}(z) = \frac{\mk_{\ell}(T(z))}{(1-T(z))^{3\ell}},
\]
where $\mk_{\ell}(T)$ is a polynomial of degree $2 \ell$.
Injecting this expression and applying
the change of variable $y \mapsto (1-T(z))^3 y$, we obtain
\[
    [y^r]
    \bigg(
    1 + \sum_{\ell \geq 1} \mgpos_{\ell}(z) y^{\ell}
    \bigg)^{-1}
    =
    \frac{
    [y^r] \left(
    1 + \sum_{\ell \geq 1} \mk(T(z)) y^{\ell}
    \right)^{-1}}
    {(1-T(z))^{3r}},
\]
where the numerator is a polynomial in $T(z)$ of degree at most $2 r$,
denoted by $M_r(T)$.
Injecting the second expression of $\mgpos_k(z)$
from Proposition~\ref{th:mgpos},
the expression of $\mg_{n,k}$ from Equation~\eqref{eq:th:cmg_asymptotics} becomes
\[
    n! 2^{n+k} (n+k)!
    \bigg(
    \sum_{r=0}^{d}
    (2(k-r)-1)!!
    [z^n x^{2(k-r)}]
    \frac{B(z,x)^{k-r+1/2} M_r(T(z))}{(1-T(z))^{3r-1/2}}
    + \bigO \left(
    (2(k-d)-3)!!
    \frac{B(\zeta, \lambda)^k}{\zeta^n \lambda^{2k}}
    \right)
    \bigg).
\]
We multiply and divide by $x^{2r}$ and $(2k-1)!!$,
and rearrange the terms to obtain
\[
    \cmg_{n,k}
    =
    n! 2^{n+k} (n+k)! (2k-1)!!
    \bigg(
    \sum_{r=0}^{d}
    \frac{(2(k-r)-1)!!}{(2k-1)!!}
    [z^n x^{2k}]
    A_r(z,x)
    B(z,x)^k
    + \bigO \left(
    k^{-d-1}
    \frac{B(\zeta, \lambda)^k}{\zeta^n \lambda^{2k}}
    \right)
    \bigg),
\]
where $A_r(z,x) = \frac{x^{2r} M_r(T(z)) B(z,x)^{-r+1/2}}{(1-T(z))^{3r-1/2}}$.
\end{proof}

We now extract the coefficient asymptotic expansion
of each summand using Lemma~\ref{th:mgpos_asymptotics},
and obtain the asymptotic expansion of $\cmg_{n,k}$.

\begin{theorem} \label{th:cmg_asymptotics}
Consider a positive integer $d$,
and two integers $n$, $k$ going to infinity
such that $\alpha := k/n$ stays in a closed interval $\mK$ of $\reals_{>0}$.
Then the number of connected multigraphs
with $n$ vertices and excess $k$
has the following asymptotic expansion,
uniformly for $\alpha$ in $\mK$
\[
    \cmg_{n,k} =
    n! 2^{n+k} (n+k)! (2k-1)!!
    \frac{B(\zeta, \lambda)^k}{2 \pi k \zeta^n \lambda^{2k}}
    \bigg(
    \sum_{r=0}^{d-1}
    c_r k^{-r}
    + \bigO(k^{-d})
    \bigg),
\]
where $\zeta$ is the unique positive solution of the equation
\[
    \frac{\lambda}{2} \frac{e^{\lambda}+1}{e^{\lambda}-1} = \alpha + 1,
\]
the values of $B(\zeta, \lambda)$ and $\zeta$ are equal to
\[
    B(\zeta, \lambda) = \frac{\lambda}{2 \alpha},
    \quad
    \zeta = e^{-\alpha-1} \sqrt{(\alpha+1)^2 - (\lambda/2)^2},
\]
the first constant $c_0$ is equal to
\[
    c_0 =
    \frac{\alpha(\lambda/2-\alpha)}{\sqrt{\lambda^2/2 - 2 \alpha^2 - 2 \alpha}},
\]
and the expression for the other $(c_r)$ is given in the proof.
Each of $c_r$, $\lambda$ and $\zeta$ is a smooth function of $\alpha$.
\end{theorem}

\begin{proof}
The starting point is the result of Lemma~\ref{th:bender_applied}
\begin{equation} \label{eq:cmg_asymptotics}
    \cmg_{n,k}
    =
    n! 2^{n+k} (n+k)! (2k-1)!!
    \bigg(
    \sum_{r=0}^{d}
    \frac{(2(k-r)-1)!!}{(2k-1)!!}
    [z^n x^{2k}]
    A_r(z,x)
    B(z,x)^k
    + \bigO \left(
    k^{-d-1}
    \frac{B(\zeta, \lambda)^k}{\zeta^n \lambda^{2k}}
    \right)
    \bigg),
\end{equation}
where the notations of the lemma have been used.
We first derive the asymptotic expansion of $\frac{(2(k-r)-1)!!}{(2k-1)!!}$
up to the order $d$
when $k$ is large and $r$ is a fixed nonnegative integer.
By definition, we have
\[
    \frac{(2(k-r)-1)!!}{(2k-1)!!}
    =
    \prod_{s=0}^{r-1} \frac{1}{2(k-s)-1}
    =
    \frac{1}{(2k)^r}
    \prod_{s=0}^{r-1}
    \frac{1}{1 - \frac{2s+1}{2k}}.
\]
Expanding $\left( 1 - \frac{2s+1}{2k} \right)^{-1}$
as a series in $1/k$, this expression becomes
\[
    \frac{1}{(2k)^r}
    \prod_{s=0}^{r-1}
    \sum_{t=0}^{d-r-1} \left( \frac{2s+1}{2k} \right)^t
    + \bigO(k^{-d})
\]
and has asymptotic expansion
\[
    \frac{(2(k-r)-1)!!}{(2k-1)!!}
    =
    \sum_{t=0}^{d-1}
    a_{r,t} k^{-t}
    + \bigO(k^{-d}),
\]
where each $a_{r,t}$ is equal to
\[
    a_{r,t} =
    2^{-t}
    \sum_{t_0 + \cdots + t_{r-1} = t-r}
    \prod_{s=0}^{r-1}
    (2s+1)^{t_s}.
\]
Observe that $a_{r,t}$ vanishes whenever $t$ is smaller than $r$.
We now turn to the asymptotic expansion of the coefficient extraction
$[z^n x^{2k}] A_r(z,x) B(z,x)^k$,
applying Lemma~\ref{th:mgpos_asymptotics}.
Let $\zeta$, $\lambda$, $\psi(x,y)$ and $J_{\psi}(x,y)$ be defined as in this lemma,
then we have the following asymptotic expansion
\[
    [z^n x^{2k}] A_r(z,x) B(z,x)^k =
    \frac{B(\zeta, \lambda)^k}{2 \pi k \zeta^n \lambda^{2k}}
    \bigg(
    \sum_{s=0}^{d-1}
    b_{r,s} k^{-s}
    + \bigO(k^{-d})
    \bigg),
\]
where each $b_{r,s}$ is computed using the formula
\[
    b_{r,s} =
    \sum_{a=0}^s
    (2a-1)!! (2(s-a)-1)!!
    [x^{2a} y^{2(s-a)}]
    A_r(\zeta e^{i \psi_1(x,y)}, \lambda e^{i \psi_2(x,y)})
    \det(J_{\psi}(x,y)).
\]
Injecting the asymptotic expansions of $\frac{(2(k-r)-1)!!}{(2k-1)!!}$
and $[z^n x^{2k}] A_r(z,x) B(z,x)^k$ into Equation~\eqref{eq:cmg_asymptotics},
we obtain
\[
    n! 2^{n+k} (n+k)! (2k-1)!!
    \bigg(
    \sum_{r=0}^d
    \bigg(
    \sum_{t=0}^{d-1} a_{r,t} k^{-t} + \bigO(k^{-d})
    \bigg)
    \frac{B(\zeta,\lambda)^k}{2 \pi k \zeta^n \lambda^{2k}}
    \bigg(
    \sum_{s=0}^{d-1} b_{r,s} k^{-s} + \bigO(k^{-d})
    \bigg)
    + \bigO \left(k^{-d} \frac{B(\zeta, \lambda)^k}{k \zeta^n \lambda^{2k}} \right)
    \bigg),
\]
which reduces to
\[
    \cmg_{n,k} =
    n! 2^{n+k} (n+k)! (2k-1)!!
    \frac{B(\zeta, \lambda)^k}{2 \pi k \zeta^n \lambda^{2k}}
    \bigg(
    \sum_{s=0}^{d-1}
    c_s k^{-s}
    + \bigO(k^{-d})
    \bigg),
\]
where
\[
    c_s =
    \sum_{r=0}^d
    \sum_{t=0}^s
    a_{r,s-t}
    b_{r,t}.
\]
Since $s < d$ and $a_{r,s-t}$ vanishes when $s-t$ is smaller than $r$,
the coefficient $c_s$ is equal to
\[
    c_s =
    \sum_{t=0}^{s}
    \sum_{r=0}^{s-t}
    a_{r,s-t}
    b_{r,t}.
\]
In particular, $c_0 = b_{0,0}$,
which is equal to
\[
    b_{0,0} = \sqrt{(1-T(\zeta)) B(\zeta, \lambda)
    \frac{\alpha^3}{\lambda}
    \frac{\lambda/2-\alpha}{\lambda^2/4-\alpha^2-\alpha}}
\]
according to Lemma~\ref{th:mgpos_asymptotics}.
Injecting the values
\[
    e^{\lambda} = \frac{\alpha+1+\frac{\lambda}{2}}{\alpha+1-\frac{\lambda}{2}},
    \qquad
    T(\zeta) = \alpha+1-\frac{\lambda}{2}
\]
derived from the characterizations of $\lambda$ and $\zeta$,
we obtain
\[
    c_0 = b_{0,0} =
    \frac{\alpha(\lambda/2-\alpha)}{\sqrt{\lambda^2/2 - 2 \alpha^2 - 2 \alpha}}.
\]
\end{proof}

\section{Connected graphs with large excess}
\label{sec:csg}

Given a positive integer $d$,
our goal is to express the number $\csg_{n,k}$
of connected graphs with $n$ vertices and excess $k$,
up to a negligible term,
as a finite sum of terms of the form $[z^n x^{2k}] A(z,x) B(z,x)^k$,
so that Lemma~\ref{th:mgpos_asymptotics}
can be applied to extract the asymptotic expansion of order $d$.
We follow the same path as in the previous section,
starting with the enumeration of cores,
positive graphs, and finally connected graphs.

In order to shorten the expressions
of the generating functions of those families,
and clarify their structure,
many auxiliary functions are introduced:
$\pospatch_k(z)$, $\sk_k(T)$, $C_{\ell}(z,x)$, $B(z,x)$, $S_r(T)$
and, for the derivation of the coefficients of the asymptotic expansion
of $\csg_{n,k}$, the functions $A_{r,\ell}(z,x)$, $\psi_1(x,y)$, and $\psi_2(x,y)$.
Those functions have rather long and intimidating expressions,
but their only property worth noticing
is that their Taylor expansion at any point
can be computed using a computer algebra language.
This property ensures that the coefficients
of the asymptotic expansion of $\csg_{n,k}$
can be computed.

Our strategy to enumerate graph families will be to first derive
the generating function of the corresponding multigraph family,
where loops and double edges are marked with a variable $u$.
Setting $u=0$ gives access to the generating function
of multigraphs without loops and double edges,
which is equal to the generating function of the graph family,
according to Lemma~\ref{th:multigraphs_to_graphs}.
Therefore, we need a way to mark
the loops and multiple edges in multigraph families.
Our tool to do so is the inclusion-exclusion technique,
in conjunction with the notion of \emph{patchwork}.

    \subsection{Patchworks} \label{sec:patchworks}

The role of patchworks is to capture the complex structures
that loops and double edges produce when they are ``glued'' together.
Recall that $\LD(G)$ denotes the set of loops and double edges of a multigraph $G$,
and $\ld(G)$ is the cardinality of this set.

\begin{definition} \label{def:patchworks}
A \emph{patchwork} with $p$ parts
\[
  P = \{(V_1, E_1), \ldots, (V_p, E_p)\}
\]
is a set of $p$ pairs $(\text{vertices}, \text{edges})$ such that
\[
  \mg(P) = \left( \cup_{i=1}^p V_i, \cup_{i=1}^p E_i \right)
\]
is a multigraph, and each $(V_i, E_i)$ 
is either a loop or a double edge of $\mg(P)$,
\ie $P \subset \LD(\mg(P))$.
The sets $(V_i)$ need not be disjoint,
and neither do the sets $(E_i)$.
The number of parts of the patchwork is $|P|$.
Its number of vertices $n(P)$, edges $m(P)$, and its excess $k(P)$
are the corresponding numbers for $\mg(P)$.
See Figure~\ref{fig:patchwork}.
\end{definition}

\begin{figure}[htbp]
  \begin{center}
    \includegraphics[scale=1]{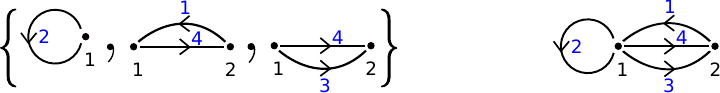}
    \caption{A patchwork $P$ of excess $2$, and the multigraph $\mg(P)$. 
    Observe that several patchworks can lead to the same multigraph.
    Here, $\LD(\mg(P)) \neq P$,
    since the double edge $(\{1,2\}, \{(2,1,1),(1,2,3)\})$ is missing from $P$.}
    \label{fig:patchwork}
  \end{center}
\end{figure}

In particular, all pairs $(V_i, E_i)$ are distinct,
$\mg(P)$ has minimum degree at least $2$,
and two edges in $E_i$, $E_j$ having the same label
must link the same vertices with the same orientation.
We use for patchwork generating functions
the same conventions as for multigraphs,
introducing an additional variable $u$ to mark the number of parts
\[
    \patch(z,w,u) =
    \sum_{\text{patchwork $P$}} u^{|P|}
    \frac{w^{m(P)}}{2^{m(P)} m(P)!}
    \frac{z^{n(P)}}{n(P)!}.
\]
The generating function of patchworks of excess $k$
is defined as
\[
    \patch_k(z,u) = [y^k] \patch(z/y, y, u),
\]
so
\[
    \patch(z,w,u) = \sum_{k \geq 0} \patch_k(zw,u) w^k.
\]
A patchwork $P$ is \emph{positive}
if all the connected components of $\mg(P)$
have positive excess.

\begin{lemma} \label{th:patchworks}
The generating function $\pospatch_k(z,u)$
of positive patchworks of excess $k$ is a multinomial,
which expression is derived in the appendix, in Lemma~\ref{th:pospatch}.
The generating function of patchworks of excess $k$ is equal to
\[
    \patch_k(z,u) = \pospatch_k(z,u) e^{u \frac{z}{2} + u \frac{z^2}{4}}.
\]
In particular, we have
\[
    \patch_0(z,u) = e^{u \frac{z}{2} + u \frac{z^2}{4}}.
\]
\end{lemma}

\begin{proof}
We first prove the third point of the lemma.
By definition, each vertex of a patchwork $P$ has degree at least $2$.
Thus, no component of $\mg(P)$ is a tree,
and the components with a nonpositive excess must be isolated cycles.
By definition again, the only possible
isolated cycles in $\mg(P)$
are the loops and the double edges,
which generating functions are, respectively,
\[
    \operatorname{Loop}(z,w,u) = u \frac{w^1}{2^1 1!} \frac{z^1}{1!},
    \qquad
    \operatorname{DoubleEdges}(z,w,u) = 4 u \frac{w^2}{2^2 2!} \frac{z^2}{2!},
\]
because there are $4$ different double edges
(see Definition~\ref{def:loops_double_edges}).
Patchworks of excess $0$ are sets of loops and double edges,
so their generating function is
\[
    \patch_0(z,u)
    =
    e^{\operatorname{Loop}(z,1,u) + \operatorname{DoubleEdges}(z,1,u)}
    =
    e^{u \frac{z}{2} + u \frac{z^2}{4}}.
\]
For the second point of the lemma, observe that
any patchwork of excess $k$ can be uniquely decomposed as
a positive patchwork of excess $k$,
and a patchwork of excess $0$, so
\[
    \patch_k(z,u) = \pospatch_k(z,u) e^{u \frac{z}{2} + u \frac{z^2}{4}}.
\]
Finally, for the first point of the lemma,
consider a positive patchwork $P$ of excess $k$.
In $\mg(P)$, each vertex of degree $2$
belongs to exactly one double edge and no loop.
The number of such double edges is at most $k$,
because each increases the global excess by $1$.
If we remove them, the corresponding multigraph
has minimum degree at least $3$ and excess at most $k$.
It is well known that there is a finite number of such multigraphs.
Indeed, consider any multigraph with minimum degree at least $3$
with $n$ vertices, $m$ edges, and excess $k=m-n$.
The sum of the degrees is at least $3n$.
Since this sum is twice the number of edges, we obtain
$2m \geq 3n$, which implies $n \leq 2k$ and $m \leq 3k$.
Thus, the family $\pospatch_k$ is finite,
and its generating function $\pospatch_k(z,u)$ is a multinomial.
\end{proof}

The multinomials $\pospatch_k(z,u)$ can be computed
by enumeration of all multigraphs with minimum degree at least $3$.
However, this is both inefficient and hard to compute.
An explicit expression is provided in Lemma~\ref{th:pospatch}.

    \subsection{Connected graphs}

The first part of the proof of the following proposition
relies on the work of \cite{EdPR16}.

\begin{proposition} \label{th:core}
The generating function of \emph{cores},
\ie graphs with minimum degree at least $2$, of excess $k$ is
\[
    \core_k(z) =
    \sum_{\ell = 0}^k
    (2(k-\ell)-1)!!
    [x^{2(k-\ell)}]
    \frac{\patch_{\ell}(z e^x,-1)}{\big(1-z \frac{e^x-1-x}{x^2/2}\big)^{k-\ell+1/2}}.
\]
\end{proposition}

\begin{proof}
Let $\mcore$ denote the set of multicores,
\ie multigraphs with minimum degree at least $2$,
and set
\[
    \mcore(z,w,u) =
    \sum_{\text{multicore $G$}} u^{\ld(G)} \frac{w^{m(G)}}{2^{m(G)} m(G)!} \frac{z^{n(G)}}{n(G)!},
\]
where $\ld(G)$ denotes the number of loops and double edges in $G$.
According to Lemma~\ref{th:multigraphs_to_graphs},
we have $\core(z,w) = \mcore(z,w,0)$.
To express the generating function of multicores, 
the inclusion-exclusion method (see \cite[Section III.7.4]{FS09})
advises us to consider $\mcore(z,w,u+1)$ instead.
This is the generating function of the set $\mcore^{\star}$ of multicores
where each loop and double edge is either marked by $u$ or left unmarked.
The set of marked loops and double edges form, by definition, a patchwork.
One can cut each unmarked edge into two labeled half-edges.
Observe that the degree constraint implies
that each vertex outside the patchwork contains at least two half-edges.
Reversely, as illustrated in Figure~\ref{fig:multicores}, 
any multicore from $\mcore^{\star}$ can be uniquely build following the steps:
\begin{enumerate}
\item
start with a patchwork $P$, which will be the final set of marked loops and double edges,
\item
add a set of isolated vertices,
\item
add to each vertex a set of labeled half-edges,
such that each isolated vertex receives at least two of them.
The total number of half-edges must be even, and is denoted by $2m$,
\item
add to the patchwork the $m$ edges obtained by
linking the half-edges with consecutive labels ($1$ with $2$, $3$ with $4$ and so on).
\end{enumerate}
\begin{figure}[htbp]
  \begin{center}
    \includegraphics[scale=0.9]{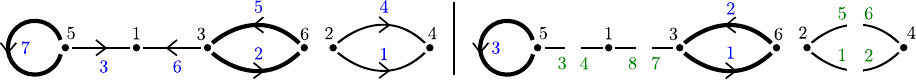}
    \caption{Left, a multigraph from $\mcore^{\star}$ (the marked loops and double edges are bold).
        Right, the corresponding multigraph with labeled half-edges, build in step $3$ of the proof of Theorem~\ref{th:core}.}
    \label{fig:multicores}
  \end{center}
\end{figure}
Observe that a relabeling of the vertices (\resp the edges) occurs at step $2$ (\resp $4$).
This construction implies, by application of the species theory (\cite{BLL97})
or the symbolic method (\cite{FS09}),
the generating function relation
\[
    \mcore(z,w,u+1) =
    \sum_{m \geq 0} (2m)! [x^{2m}] \patch(z e^x,w,u) e^{z (e^x-1-x)} \frac{w^m}{2^m m!},
\]
where
\begin{itemize}
\item
the variable $x$ marks the half-edges,
\item
$\patch(z e^x,w,u)$ is the generating function of patchworks
where a set of half-edges is added to each vertex,
\item
$e^{z (e^x-1-x)}$ is a set of vertex,
to each of which a set of at least $2$ half-edges is attached,
\item
the product by $\frac{w^m}{2^m m!}$ and the coefficient extraction $(2m)! [x^{2m}]$
represent the replacement of the $2m$ half-edges with $m$ edges.
\end{itemize}
For $u=-1$, applying Lemma~\ref{th:multigraphs_to_graphs},
we obtain the expression of $\core(z,w) = \mcore(z,w,0)$.

The end of the proof of the proposition is the same as for Proposition~\ref{th:mcorepos}.
After developing the exponential as a sum over $n$
and applying the change of variable $m \leftarrow k+n$, we obtain
\[
    \core(z,w) =
    \sum_{k \geq 0}
    [x^{2k}]
    \patch(z e^x,w,-1)
    \sum_{n \geq 0}
    (2(k+n)-1)!!
    \frac{\big(z w \frac{e^x-1-x}{x^2} \big)^n}{n!}
    w^k.
\]
The sum over $n$ is replaced by its closed form
\[
    \core(z,w) =
    \sum_{k \geq 0}
    [x^{2k}]
    \patch(z e^x,w,-1)
    \frac{(2k-1)!!}{\big(1- z w \frac{e^x-1-x}{x^2/2} \big)^{k+1/2}}
    w^k.
\]
We now decompose the generating function of patchworks
according to their excess
\[
    \patch(z,w,u) =
    \sum_{\ell \geq 0} \patch_k(z w,u) w^\ell.
\]
The generating function of cores of excess $k$ is then
\[
    \core_k(z) = [y^k] \core(z/y,y) =
    [y^k]
    \sum_{\ell \geq 0}
    (2\ell-1)!!
    [x^{2\ell}]
    \frac{\sum_{j \geq 0} \patch_j(z,u) y^{j}}{\big(1- z \frac{e^x-1-x}{x^2/2} \big)^{\ell+1/2}}
    y^\ell,
\]
and the coefficient extraction $[y^k]$ gives the result
\[
    \core_k(z) =
    \sum_{\ell = 0}^k
    (2(k-\ell)-1)!!
    [x^{2(k-\ell)}]
    \frac{\patch_{\ell}(z e^x,-1)}{\big(1-z \frac{e^x-1-x}{x^2/2}\big)^{k-\ell+1/2}}.
\]
\end{proof}

In the previous proof,
the inclusion-exclusion technique was applied
to enumerate multigraphs where loops and double edges subgraphs are forbidden.
Other subgraphs could be removed as well,
using the same approach.
This topic has been investigated by \cite{EdPCGGR17}.
We now apply the previous results to analyze positive graphs.
Again, two expressions are derived:
one suited to the analysis of constant excesses,
the other one to large excesses.

\begin{proposition} \label{th:sgpos}
The generating function of positive graphs of excess $k$
has the following two expressions
\begin{align*}
    \sgpos_k(z)
    &=
    \frac{\sk_k(T(z))}{(1-T(z))^{3k}},
    \\    
    \sgpos_k(z)
    &=
    \sum_{\ell = 0}^k 
    (2(k-\ell)-1)!!
    [x^{2k}] C_{\ell}(z,x)
    B(z,x)^k,
\end{align*}
where the auxiliary functions $B(z,x)$ and $C_{\ell}(z,x)$ and the polynomial $\sk_k(T)$ are equal to
\begin{align*}
    B(z,x) &= \left(1 - T(z) \frac{e^x-1-x}{x^2/2} \right)^{-1},
    \\
    C_{\ell}(z,x) &=
    x^{2\ell} \pospatch_{\ell}(T(z) e^x,-1) \sqrt{1-T(z)}
    e^{- T(z)\frac{e^x-1}{2}  - T(z)^2 \frac{e^{2x} - 1}{4}}
    B(z,x)^{-\ell+1/2},
    \\
    \sk_k(T) &=
    (1-T)^{3k}
    \sum_{\ell = 0}^k
    (2(k-\ell)-1)!!
    [x^{2(k-\ell)}]
    \frac{\pospatch_{\ell}(T e^x,-1) e^{- T \frac{e^x-1}{2} - T^2 \frac{e^{2x}-1}{4}}}
    {(1-T)^{k-\ell} \big(1- \frac{T}{1-T} \frac{e^x-1-x-x^2/2}{x^2/2}\big)^{k-\ell+1/2}}
\end{align*}
and $\pospatch_{\ell}(T)$ is a polynomial expressed in Section~\ref{sec:explicite_patchworks}.
\end{proposition}

\begin{proof}
In a core, the components of nonpositive excess
are the isolated cycles.
In a multigraph, such a cycle could be
of length $1$ (a loop),
or of length $2$ (a double edge).
This is forbidden in simple graphs,
so cycles have minimum length $3$,
and their generating function is
\[
    \sum_{n \geq 3}
    \frac{(n-1)!}{2} w^n \frac{z^n}{n!}
    =
    \frac{1}{2} \log \left( \frac{1}{1-z w} \right)
    - \frac{w z}{2} - \frac{w^2 z^2}{4}.
\]
The univariate generating function of isolated cycles
is obtained by setting $w$ to $1$
\[
    \frac{1}{2} \log \left( \frac{1}{1-z} \right) - \frac{z}{2} - \frac{z^2}{4}.
\]
A core of excess $k$ is a positive core of excess $k$
with a set of isolated cycles (which have excess $0$), so
\[
    \core_k(z)
    =
    \poscore_k(z) e^{\frac{1}{2} \log \left( \frac{1}{1-z} \right) - \frac{z}{2} - \frac{z^2}{4}}.
\]
Injecting the expression of $\core_k(z)$ from Proposition~\ref{th:core}
and solving this expression, we obtain
\[
    \poscore_k(z) =
    \sum_{\ell = 0}^k
    (2(k-\ell)-1)!!
    [x^{2(k-\ell)}]
    \frac{\patch_{\ell}(z e^x,-1) \sqrt{1-z} e^{\frac{z}{2} + \frac{z^2}{4}}}
    {\big(1-z \frac{e^x-1-x}{x^2/2}\big)^{k-\ell+1/2}}.
\]
A positive graph is a positive core,
where a rooted tree has been attached to each vertex.
This operation does not affect the excess,
as the same number of vertices and edges is added, so
\[
    \sgpos_k(z) =
    \poscore_k(T(z)) =
    \sum_{\ell = 0}^k
    (2(k-\ell)-1)!!
    [x^{2(k-\ell)}]
    \frac{\patch_{\ell}(T(z) e^x,-1) \sqrt{1-T(z)} e^{\frac{T(z)}{2} + \frac{T(z)^2}{4}}}
    {\big(1-T(z) \frac{e^x-1-x}{x^2/2}\big)^{k-\ell+1/2}}.
\]

To prove the first result,
we replace $1-T(z) \frac{e^x-1-x}{x^2/2}$ with
$(1-T(z)) \left( 1 - \frac{T(z)}{1-T(z)} \frac{e^x-1-x-x^2/2}{x^2/2} \right)$,
and $\patch_{\ell}(z,u)$ with $\pospatch_{\ell}(z,u) e^{u \frac{z}{2} + u \frac{z^2}{4}}$
(see Lemma~\ref{th:patchworks})
\[
    \sgpos_k(z) =
    \sum_{\ell = 0}^k
    (2(k-\ell)-1)!!
    [x^{2(k-\ell)}]
    \frac{\pospatch_{\ell}(T(z) e^x,-1) e^{- T(z) \frac{e^x-1}{2} - T(z)^2 \frac{e^{2x}-1}{4}}}
    {(1-T(z))^{k-\ell} \big(1- \frac{T(z)}{1-T(z)} \frac{e^x-1-x-x^2/2}{x^2/2}\big)^{k-\ell+1/2}}.
\]
Let us rewrite this expression as
\[
    \sgpos_k(z) = P \left( T(z), \frac{1}{1-T(z)} \right),
\]
where
\[
    P(T,S) =
    \sum_{\ell = 0}^k
    (2(k-\ell)-1)!!
    [x^{2(k-\ell)}]
    \frac{S^{k-\ell} \pospatch_{\ell}(T e^x,-1) e^{- T \frac{e^x-1}{2} - T^2 \frac{e^{2x}-1}{4}}}
    {\big(1- T S \frac{e^x-1-x-x^2/2}{x^2/2}\big)^{k-\ell+1/2}}.
\]
Since $\frac{e^x-1-x-x^2/2}{x^2/2}$, $e^x-1$ and $e^{2x}-1$ have valuation $1$,
and since $\pospatch_{\ell}$ is a polynomial,
$P(T,S)$ is a multinomial in the variables $T$ and $S$,
of degree at most $3k$ in $S$.
This implies that
\[
    \sk_k(T) = (1-T)^{3k} P\left( T, \frac{1}{1-T} \right)
\]
is a polynomial in $T$, and
\[
    \sgpos_k(z) = \frac{\sk_k(T(z))}{(1-T(z))^{3k}}.
\]
\end{proof}

We have proven that, for any fixed $k$,
the generating function of positive graphs of excess $k$
is a rational function in $T(z)$.
A more combinatorial proof is provided in Section~\ref{sec:kernels}.
The second expression of $\sgpos_k(z)$
involves a sum for $\ell$ from $0$ to $k$,
which is unbounded when $k$ goes to infinity.
Hence, although each summand is amenable to asymptotic analysis
using the saddle-point lemma~\ref{th:large_powers},
the asymptotics of the coefficients of the sum is not immediately available.
The next lemma establishes that
only the first terms of this sum
have a nonnegligible contribution
to the coefficients.

\begin{lemma} \label{th:IE}
Let $A(z)$ denote a series of radius of convergence greater than $\zeta$,
then
\[
    n! [z^n] \sgpos_k(z) A(z) =
    n!
    \sum_{\ell=0}^{d}
    (2(k-\ell)-1)!! [z^n x^{2k}]
    A(z)
    C_{\ell}(z,x) B(z,x)^k
    +
    \bigO\left(n! (2(k-d)-3)!! \frac{B(\zeta, \lambda)^k}{\zeta^n \lambda^{2k}} \right)
\]
where the auxiliary function $C_{\ell}(z,x)$
is defined in Proposition~\ref{th:sgpos}.
\end{lemma}

\begin{proof}
We first present the proof in the particular case $A(z) = 1$.
%
Let $\IEd$ denote the \emph{bounded inclusion-exclusion} operator,
which inputs a multigraph family $\mF$, and outputs the value
\[
    \IEd(\mF) =
    \sum_{G \in \mF} \sum_{\substack{P \subset \LD(G)\\ k(P) \leq d}} (-1)^{|P|},
\]
where the second sum is taken over all patchworks
contained in the multigraph $G$ and of excess at most $d$.
In this proof, let $\mgpos_{n,k}$ denote the set (instead of the number)
of positive multigraphs with $n$ vertices and excess $k$.
Let also $\mgpos_{n,k, \leq d}$ (\resp $\mgpos_{n,k, > d}$)
denote the set of positive multigraphs with $n$ vertices and excess $k$,
where the set of loops and double edges form a patchwork
of excess at most $d$ (\resp greater than $d$).
By definition, we have
\[
    \mgpos_{n,k} = \mgpos_{n,k, \leq d} \uplus \mgpos_{n,k, > d}.
\]
Applying the bounded inclusion-exclusion operator to this relation, we obtain
\begin{equation} \label{eq:ied}
    \IEd(\mgpos_{n,k}) = \IEd(\mgpos_{n,k, \leq d}) + \IEd(\mgpos_{n,k, > d}).
\end{equation}
We now derive expressions for the first two terms, and bound the third.
\\ \noindent \textbf{First term.}
\[
    \sum_{G \in \mgpos_{n,k,\leq d}} \sum_{\substack{P \subset \LD(G)\\ k(P) \leq d}} u^{|P|}
\]
is the generating function
of multigraphs with $n$ vertices and excess $k$,
where each loop and double edge is either marked by the variable $u$,
or left unmarked, and the marked loops and double edges
form a patchwork of excess at most $d$.
Following the proof of Proposition~\ref{th:sgpos}, this quantity is equal to
\[
    n!
    2^{n+k} (n+k)!
    \sum_{\ell=0}^{d}
    (2(k-\ell)-1)!! [z^n x^{2(k-\ell)}]
    \patch_{\ell}(T(z) e^x,u) \sqrt{1-T(z)}
    e^{\frac{T(z)}{2} + \frac{T(z)^2}{4}}
    B(z,x)^{k-\ell+1/2}.
\]
For $u=-1$, we obtain
\[
    \IEd(\mgpos_{n,k})
    =
    n! 2^{n+k} (n+k)! (2k-1)!!
    \sum_{\ell=0}^{d}
    \frac{(2(k-\ell)-1)!!}{(2k-1)!!} [z^n x^{2k}]
    C_{\ell}(z,x) B(z,x)^k.
\]
\noindent \textbf{Second term.}
By definition, this term is equal to
\[
    \IEd(\mgpos_{n,k,\leq d}) =
    \sum_{G \in \mgpos_{n,k,\leq d}} \sum_{\substack{P \subset \LD(G)\\ k(P) \leq d}} (-1)^{|P|}.
\]
Since the multigraphs from $\mgpos_{n,k,\leq d}$
contain only patchworks of excess at most $d$,
the second condition of the second sum is redundant.
Applying the classical relation $\sum_{A \subset B} x^{|A|} = (1+x)^{|B|}$,
valid for any finite set $B$,
we obtain
\[
    \IEd(\mgpos_{n,k,\leq d}) = \sum_{G \in \mgpos_{n,k,\leq d}} 0^{|\LD(G)|}.
\]
The summand vanishes, unless $\LD(G)$ is empty, because $0^0 = 1$.
According to Lemma~\ref{th:multigraphs_to_graphs},
the number of multigraphs from $\mgpos_{n,k}$
without loops and double edges
is equal to $2^{n+k} (n+k)!$ times
the number $\sgpos_{n,k}$ of positive graphs
with $n$ vertices and excess $k$, so
\[
    \IEd(\mgpos_{n,k,\leq d}) = 2^{n+k} (n+k)! \sgpos_{n,k}.
\]
Injecting the expressions of the first and second term into Equation~\eqref{eq:ied}, we obtain
\[
    2^{n+k} (n+k)!
    \bigg| \sgpos_{n,k} -
    n! (2k-1)!!
    \sum_{\ell=0}^{d}
    \frac{(2(k-\ell)-1)!!}{(2k-1)!!} [z^n x^{2k}]
    C_{\ell}(z,x) B(z,x)^k
    \bigg|
    \leq
    \big| \IEd(\mgpos_{n,k, > d}) \big|.
\]
Thus, the proof of the lemma is complete
once we establish the existence of a constant $D$,
that depends only on $d$ an not on $n$ or $k$, such that
\[
    \IEd(\mgpos_{n,k, > d}) \leq
    D n! 2^{n+k} (n+k)! (2(k-d)-3)!! \frac{B(\zeta, \lambda)^k}{\zeta^n \lambda^{2k}}.
\]
\noindent \textbf{Third term.}
The triangle inequality implies the following bound
\[
    \IEd(\mgpos_{n,k,> d})
    \leq
    \sum_{G \in \mgpos_{n,k,>d}} \sum_{\substack{P \subset \LD(G)\\ k(P) \leq d}} |(-1)^{|P|}|
    \leq
    \sum_{G \in \mgpos_{n,k,>d}} \big| \{P \subset \LD(G)\ |\ k(P) \leq d\} \big|
\]
Any multigraph from $\mgpos_{n,k, > d}$
contains a positive patchwork of excess exactly $d+1$.
Thus, $\IEd(\mgpos_{n,k, > d})$ is bounded
by the number of positive multigraphs
with $n$ vertices and excess $k$,
where one positive patchwork $P$ of excess $d+1$ is distinguished,
and a patchwork $Q$ (not necessarily disjoint of $P$)
of excess at most $d$ is distinguished as well.
A simpler upper bound follows from considering that $P$ and $Q$ are disjoint,
but each loop and double edge from $P$ brings a factor $2$,
to take into account that it might belong, or not, to $Q$.
Following the proof of Proposition~\ref{th:sgpos}, this bound is equal to
\begin{align*}
    n!
    2^{n+k} (n+k)!
    \sum_{\ell=0}^{d}
    &
    (2(k-\ell-d-1)-1)!! [z^n x^{2(k-\ell-d-1)}]
    \pospatch_{d+1}(T(z) e^x,2)
    \\
    & \times
    \patch_{\ell}(T(z) e^x, 1) \sqrt{1-T(z)}
    e^{\frac{T(z)}{2} + \frac{T(z)^2}{4}}
    B(z,x)^{k-\ell-d-1+1/2}.
\end{align*}
According to Lemma~\ref{th:saddle_point_bound},
for any fixed $d$, there is a constant $D$ this quantity is bounded by
\[
    D
    n! 2^{n+k} (n+k)!
    (2(k-d)-3)!!
    \frac{B(\zeta, \lambda)^k}{\zeta^n \lambda^{2k}}.
\]
Therefore, we have proven the inequality
\begin{equation} \label{eq:IE_bound}
    \bigg|
    [z^n] \sgpos(z)
    - \sum_{\ell=0}^d
    (2(k-\ell)-1)!! [z^n x^{2k}] C_{\ell}(z,x) B(z,x)^k
    \bigg|
    \leq
    D (2(k-d)-3)!! \frac{B(\zeta,\lambda)^k}{\zeta^n \lambda^{2k}}.
\end{equation}
We now turn to the proof of the general case of a function
\[
    A(z) = \sum_{r \geq 0} a_r z^r
\]
of radius of convergence greater than $\zeta$.
We have
\[
    [z^n] \sgpos_k(z) A(z) =
    \sum_{r=0}^n a_r [z^{n-r}] \sgpos_k(z),
\]
and
\[
    \sum_{\ell=0}^{d}
    (2(k-\ell)-1)!! [z^n x^{2k}]
    A(z)
    C_{\ell}(z,x) B(z,x)^k
    =
    \sum_{\ell=0}^{d}
    \sum_{r=0}^n a_r (2(k-\ell)-1)!! [z^{n-r} x^{2k}]
    C_{\ell}(z,x) B(z,x)^k.
\]
Applying the triangle inequality
and the bound from Equation~\eqref{eq:IE_bound},
we conclude
\[
    \bigg|
    [z^n] \sgpos_k(z) A(z)
    -
    \sum_{\ell=0}^{d}
    (2(k-\ell)-1)!! [z^n x^{2k}]
    A(z)
    C_{\ell}(z,x) B(z,x)^k
    \bigg|
    \leq
    \sum_{r=0}^n
    |a_r| D (2(k-d)-3)!! \frac{B(\zeta, \lambda)^k}{\zeta^{n-r} \lambda^{2k}}.
\]
In the right hand-side, we apply the fact that
$A(z)$ has radius of convergence greater than $\zeta$,
so
$
    \sum_{r \geq 0} |a_r| \zeta^r
$
is a convergent series.
Thus, the right hand-side is a
$
    \bigO \left((2(k-d)-3)!! \frac{B(\zeta, \lambda)^k}{\zeta^{n} \lambda^{2k}} \right).
$
Multiplying by $n!$ concludes the proof.
\end{proof}

\cite{ER60} proved that when $k = \exactbigO(n)$,
a typical random positive graph
with $n$ vertices and excess $k$ is connected.
This implies the asymptotic equivalence
\[
    \csg_{n,k} \sim \sgpos_{n,k},
\]
used by \cite{PW05}.
The last lemma provides the asymptotics
\[
    \sgpos_{n,k}
    = n! [z^n] \sgpos_k(z)
    \sim n! (2k-1)!! [z^n x^{2k}]
    C_{0}(z,x) B(z,x)^k.
\]
If we were only interested
into the first order of the asymptotics of connected graphs,
we would apply Lemma~\ref{th:mgpos_asymptotics}
to derive the asymptotics of $\sgpos_{n,k}$,
and hence of $\csg_{n,k}$.
The next proposition recalls the link between
the generating functions of connected graphs
and of positive graphs.

\begin{proposition} \label{th:csg_exact}
For any positive $k$,
the generating function of connected graphs of excess $k$
is equal to
\[
    \csg_k(z) =
    [y^k] \log \bigg(
    1 + \sum_{\ell \geq 1} \sgpos_{\ell}(z) y^{\ell}
    \bigg).
\]
\end{proposition}

\begin{proof}
This is the exact same proof as for Proposition~\ref{th:cmg_exact}.
\end{proof}

Following the proof we did for the asymptotics of connected multigraphs,
we now apply Lemma~\ref{th:bender_like}
to express, up to a negligible term,
the number of connected graphs
using the generating function of positive graphs.

\begin{lemma} \label{th:csg_bender_applied}
Consider a positive integer $d$,
and two integers $n$, $k$ going to infinity
such that $\alpha := k/n$ stays in a closed interval $\mK$ of $\reals_{>0}$.
Then the number of connected graphs
with $n$ vertices and excess $k$
is equal to
\[
    \csg_{n,k} =
    n! [z^n]
    \sum_{r=0}^{d}
    \sgpos_{k-r}(z)
    \frac{S_r(T(z))}{(1-T(z))^{3r}}
    + \bigO \left( n! (2(k-d)-3)!! \frac{B(\zeta, \lambda)^k}{\zeta^n \lambda^{2k}} \right),
\]
where $\zeta$, $\lambda$ are defined as in Lemma~\ref{th:mgpos_asymptotics},
the auxiliary polynomial $S_r(T)$ and the series $B(z,x)$ are equal to
\begin{align*}
    B(z,x) &= \left(1-T(z) \frac{e^x-1-x}{x^2/2}\right)^{-1},
    \\
    S_r(T) &= [y^r] \bigg(1 + \sum_{\ell=1}^r \sk_{\ell}(T) y^{\ell} \bigg)^{-1},
\end{align*}
and $\sk_{\ell}(T)$ is defined in Proposition~\ref{th:sgpos}.
\end{lemma}

\begin{proof}
We start with the result of Proposition~\ref{th:csg_exact}
\[
    \csg_k(z) = [y^k] \log \bigg( 1 + \sum_{\ell \geq 1} \sgpos_{\ell}(z) y^{\ell} \bigg),
\]
and we plan to apply Lemma~\ref{th:bender_like}, with
$g_{\ell}(z) = \sgpos_{\ell}(z)$, $\beta = 1/2$, and $E = 2 B(\zeta, \lambda)/\lambda^2$.
In order to do so, we need to bound the value $\sgpos_{\ell}(\zeta)$.
Any positive graph with $n$ vertices and excess $\ell$
(hence with $n+\ell$ edges)
matches $2^{n+\ell} (n+\ell)!$ positive multigraphs
with $n$ vertices and excess $\ell$,
obtained by orienting and labeling the edges of the graphs.
Thus, denoting by $\sgpos_{n,\ell}$ (\resp $\mgpos_{n,\ell}$)
the number of such positive graphs (\resp multigraphs),
we have
\[
    \sgpos_{n,\ell} \leq \frac{\mgpos_{n,\ell}}{2^{n+\ell} (n+\ell)!}.
\]
This implies the following inequality between the generating functions
evaluated at the positive value $\zeta$
\[
    \sgpos_{\ell}(\zeta)
    =
    \sum_{n \geq 0} \sgpos_{n,\ell} \frac{\zeta^n}{n!}
    \leq
    \sum_{n \geq 0} \frac{\mgpos_{n,\ell}}{2^{n+\ell} (n+\ell)!} \frac{\zeta^n}{n!}
    =
    \mgpos_{\ell}(\zeta).
\]
The expression $\mgpos_{\ell}(z)$ was derived in Proposition~\ref{th:mgpos}
\[
    \mgpos_{\ell}(z) = (2\ell-1)!! [x^{2\ell}] \sqrt{1-T(z)} B(z,x)^{\ell+1/2}.
\]
Using the identity
$
    (2\ell-1)!! = \frac{2^{\ell}}{\sqrt{\pi}} \Gamma(\ell + 1/2)
$
and Lemma~\ref{th:saddle_point_bound}, we conclude
\[
    \sgpos_{\ell}(\zeta) \leq
    \frac{2^{\ell}}{\sqrt{\pi}}
    \Gamma(\ell+1/2)
    \frac{\sqrt{1-T(\zeta)} B(\zeta, \lambda)^{\ell+1/2}}{\lambda^{2\ell}}.
\]
Thus, the hypothesis of Lemma~\ref{th:bender_like} are satisfied,
and it implies
\[
    \csg_{n,k} =
    n! [z^n]
    \sum_{r=0}^{d}
    \sgpos_{k-r}(z)
    [y^r] \bigg(1 + \sum_{\ell = 1}^r \sgpos_{\ell}(z) y^{\ell} \bigg)^{-1}
    + \bigO \left( n! (2(k-d)-3)!! \frac{B(\zeta, \lambda)^k}{\zeta^n \lambda^{2k}} \right).
\]
In Proposition~\ref{th:sgpos}, we obtained the expression
\[
    \sgpos_{\ell}(z) = \frac{\sk_{\ell}(T(z))}{(1-T(z))^{3\ell}},
\]
which implies
\[
    [y^r] \bigg(1 + \sum_{\ell = 1}^r \sgpos_{\ell}(z) y^{\ell} \bigg)^{-1}
    =
    \frac{S_r(T(z))}{(1-T(z))^{3r}},
\]
and the result of the lemma follows.
\end{proof}

\begin{theorem} \label{th:csg_asymptotics}
Consider a positive integer $d$,
and two integers $n$, $k$ going to infinity
such that $\alpha := k/n$ stays in a closed interval $\mK$ of $\reals_{>0}$.
Then the number of connected graphs
with $n$ vertices and excess $k$
has the following asymptotic expansion,
uniformly for $\alpha$ in $\mK$
\[
    \csg_{n,k} =
    n! (2k-1)!!
    \frac{B(\zeta, \lambda)^k}{2 \pi k \zeta^n \lambda^{2k}}
    \bigg(
    \sum_{r=0}^{d-1}
    c_r k^{-r}
    + \bigO(k^{-d})
    \bigg),
\]
where $\lambda$ is the unique positive solution of the equation
\[
    \frac{\lambda}{2} \frac{e^{\lambda}+1}{e^{\lambda}-1} = \alpha + 1,
\]
the values of $B(\zeta, \lambda)$ and $\zeta$ are equal to
\[
    B(\zeta, \lambda) = \frac{\lambda}{2 \alpha},
    \quad
    \zeta = e^{-\alpha-1} \sqrt{(\alpha+1)^2 - (\lambda/2)^2},
\]
the first constant $c_0$ is equal to
\[
    c_0 =
    e^{- \frac{\lambda}{2} (\alpha+2)}
    \frac{\alpha(\lambda/2-\alpha)}{\sqrt{\lambda^2/2 - 2 \alpha^2 - 2 \alpha}},
\]
and the expression of the other $(c_r)$ is given in the proof.
Each of $c_r$, $\lambda$ and $\zeta$ is a smooth function of $\alpha$.
\end{theorem}

\begin{proof}
We start with the result of Lemma~\ref{th:csg_bender_applied}
\[
    \csg_{n,k} =
    n! [z^n]
    \sum_{r=0}^{d}
    \sgpos_{k-r}(z)
    \frac{S_r(T(z))}{(1-T(z))^{3r}}
    + \bigO \left( n! (2(k-d)-3)!! \frac{B(\zeta, \lambda)^k}{\zeta^n \lambda^{2k}} \right),
\]
and apply Lemma~\ref{th:IE}, so $\csg_{n,k}$ is equal to
\[
    n!
    \sum_{r=0}^{d}
    \sum_{\ell=0}^{d-r}
    (2(k-r-\ell)-1)!!
    [z^n x^{2(k-r)}]
    \frac{S_r(T(z))}{(1-T(z))^{3r}}
    C_{\ell}(z,x)
    B(z,x)^k
    + \bigO \left( n! (2(k-d)-3)!! \frac{B(\zeta, \lambda)^k}{\zeta^n \lambda^{2k}} \right).
\]
Dividing and multiplying by $(2k-1)!!$,
and replacing the coefficient extraction
$[x^{2(k-r)}]$ with $[x^{2k}] x^{2r}$,
the expression becomes 
\begin{equation} \label{eq:csg_asymptotics}
    \csg_{n,k} =
    n! (2k-1)!!
    \bigg(
    \sum_{r=0}^{d}
    \sum_{\ell=0}^{d-r}
    \frac{(2(k-r-\ell)-1)!!}{(2k-1)!!}
    [z^n x^{2k}]
    A_{r,\ell}(z,x)
    B(z,x)^k
    + \bigO \left( k^{-d-1} \frac{B(\zeta, \lambda)^k}{\zeta^n \lambda^{2k}} \right)
    \bigg),
\end{equation}
where $A_{r,\ell}(z,x) = x^{2r} \frac{S_r(T(z))}{(1-T(z))^{3r}} C_{\ell}(z,x)$.
The asymptotic expansion of the quotient of double factorials
has been derived in the proof of Theorem~\ref{th:cmg_asymptotics},
and is equal to
\[
    \frac{(2(k-r-\ell)-1)!!}{(2k-1)!!}
    =
    \sum_{t=0}^{d-1}
    a_{r+\ell,t} k^{-t}
    + \bigO(k^{-d}),
\]
where each $a_{r,t}$ is equal to
\[
    a_{r,t} =
    2^{-t}
    \sum_{t_0 + \cdots + t_{r-1} = t-r}
    \prod_{s=0}^{r-1}
    (2s+1)^{t_s}.
\]
We now turn to the asymptotic expansion of the coefficient extraction
$[z^n x^{2k}] A_{r,\ell}(z,x) B(z,x)^k$,
applying Lemma~\ref{th:mgpos_asymptotics}.
With the notations $\psi(x,y)$ and $J_{\psi}(x,y)$
from this lemma,
we have the following asymptotic expansion
\[
    [z^n x^{2k}] A_{r,\ell}(z,x) B(z,x)^k =
    \frac{B(\zeta, \lambda)^k}{2 \pi k \zeta^n \lambda^{2k}}
    \bigg(
    \sum_{s=0}^{d-1}
    b_{r,\ell,s} k^{-s}
    + \bigO(k^{-d})
    \bigg),
\]
where each $b_{r,\ell,s}$ is computed using the formula
\[
    b_{r,\ell,s} =
    \sum_{a=0}^s
    (2a-1)!! (2(s-a)-1)!!
    [x^{2a} y^{2(s-a)}]
    A_{r,\ell}(\zeta e^{i \psi_1(x,y)}, \lambda e^{i \psi_2(x,y)})
    \det(J_{\psi}(x,y)).
\]
Injecting the asymptotic expansions of $\frac{(2(k-r-\ell)-1)!!}{(2k-1)!!}$
and $[z^n x^{2k}] A_{r,\ell}(z,x) B(z,x)^k$ into Equation~\eqref{eq:csg_asymptotics},
we obtain for $\csg_{n,k}$
\[
    n! (2k-1)!!
    \bigg(
    \sum_{r=0}^{d}
    \sum_{\ell=0}^{d-r}
    \bigg(
    \sum_{t=0}^{d-1} a_{r+\ell,t} k^{-t} + \bigO(k^{-d})
    \bigg)
    \frac{B(\zeta,\lambda)^k}{2 \pi k \zeta^n \lambda^{2k}}
    \bigg(
    \sum_{s=0}^{d-1} b_{r,\ell,s} k^{-s} + \bigO(k^{-d})
    \bigg)
    + \bigO \left(k^{-d-1} \frac{B(\zeta,\lambda)^k}{\zeta^n \lambda^{2k}} \right)
    \bigg),
\]
which reduces to
\[
    \csg_{n,k} =
    n! (2k-1)!!
    \frac{B(\zeta, \lambda)^k}{2 \pi k \zeta^n \lambda^{2k}}
    \bigg(
    \sum_{s=0}^{d-1}
    c_s k^{-s}
    + \bigO(k^{-d})
    \bigg),
\]
where
\[
    c_s =
    \sum_{r=0}^{d}
    \sum_{\ell=0}^{d-r}
    \sum_{t=0}^s
    a_{r+\ell,s-t}
    b_{r,\ell,t}.
\]
Since $s < d$ and $a_{r+\ell,s-t}$ vanishes when $s-t$ is smaller than $r+\ell$,
the coefficient $c_s$ is equal to
\[
    c_s =
    \sum_{t=0}^s
    \sum_{r=0}^{s-t}
    \sum_{\ell=0}^{s-t-r}
    a_{r+\ell,s-t}
    b_{r,\ell,t}.
\]
In particular, $c_0 = \sum_{r=0}^d a_{r,0} b_{r,0,0}$.
Since $a_{r,0}$ vanishes unless $r=0$,
we have $c_0 = b_{0,0,0}$,
which is equal to
\[
    b_{0,0,0} =
    A_{0,0}(\zeta, \lambda)
    \sqrt{
    \frac{\alpha^3}{\lambda}
    \frac{\lambda/2-\alpha}{\lambda^2/4 - \alpha^2 - \alpha}}
\]
according to Lemma~\ref{th:large_powers}.
The expression of $A_{0,0}(\zeta, \lambda)$ is,
using the notation $C_{\ell}(z,x)$ from Proposition~\ref{th:sgpos},
\[
    A_{0,0}(\zeta, \lambda) =
    C_0(\zeta,\lambda) B(\zeta, \lambda)^{1/2} =
    e^{- T(\zeta) \frac{e^{\lambda}-1}{2} - T(\zeta)^2 \frac{e^{2 \lambda}-1}{4}}
    \sqrt{(1-T(\zeta) B(\zeta,\lambda)}.
\]
Injecting the values
\[
    e^{\lambda} = \frac{\alpha+1+\frac{\lambda}{2}}{\alpha+1-\frac{\lambda}{2}},
    \qquad
    T(\zeta) = \alpha+1-\frac{\lambda}{2},
    \qquad
    B(\zeta,\lambda) = \frac{\lambda}{2\alpha}
\]
derived from the characterizations of $\lambda$ and $\zeta$,
we obtain
\[
    - T(\zeta) \frac{e^{\lambda}-1}{2} - T(\zeta)^2 \frac{e^{2 \lambda}-1}{4}
    =
    - \frac{\lambda}{2} (\alpha+2)
\]
and
\[
    c_0 = b_{0,0,0} =
    e^{- \frac{\lambda}{2} (\alpha+2)}
    \frac{\alpha(\lambda/2-\alpha)}{\sqrt{\lambda^2/2 - 2 \alpha^2 - 2 \alpha}}.
\]
\end{proof}

Comparing the dominant terms of the asymptotics of connected graphs
with the one obtained in Theorem~\ref{th:cmg_asymptotics} on connected multigraphs,
we conclude that the asymptotic probability for a random
multigraph with $n$ vertices and excess $k$,
with $\alpha = k/n$ in a closed interval of $\reals_{>0}$,
to contain neither loops nor multiple edges
is $e^{-\frac{\lambda}{2} (\alpha + 2)}$.

Step by step instructions for the computation of the coefficients $(c_r)$
are provided in Appendix~\ref{sec:computation}.

\section{Connected graphs and multigraphs with fixed excess}
\label{sec:fixed_k}

In the first subsection, we study the link between positive graphs and multigraphs,
and \emph{kernels}, which are multigraphs with minimum degree at least $3$.
%
In the second subsection, the asymptotics of connected graphs of multigraphs
of fixed excess is derived.
This result was first obtained by \cite{W80},
and our contribution is to present a different formula
for the constant term.

    \subsection{Positive (multi)graphs and kernels} \label{sec:kernels}

In Propositions~\ref{th:mgpos} and~\ref{th:sgpos},
we proved that the generating functions
of positive graphs and multigraphs of excess $k$
are rational functions in $T(z)$.
In this subsection, we provide, for the sake of completeness,
the combinatorial proof for this property
first derived by \cite{W80}.
The proof is based on the reduction of positive multigraphs and graphs
to their \emph{kernels},
which are multigraphs with minimum degree at least $3$,
and is illustrated in Figure~\ref{fig:path_of_trees}.

\begin{lemma} \label{th:finite_kernels}
There is a finite number of kernels of excess $k$.
Furthermore, such kernels contain at most $2k$ vertices, and $3k$ edges.
\end{lemma}

\begin{proof}
Consider a kernel with $n$ vertices, $m$ edges, and excess $k = m-n$.
The sum of the degrees of a multigraph is equal to twice its number of edges.
Since each degree is at least $3$, we conclude
\[
    3n \leq \sum_{v \in V(G)} \deg(v) = 2m,
\]
which implies $n \leq 2k$ and $m \leq 3k$.
\end{proof}

The previous lemma implies that the generating function $\kernel_k(z)$
of kernels of excess $k$ is a polynomial.

\begin{proposition}[already contained in Proposition~\ref{th:mgpos}] \label{th:mgpos_kernels}
For each positive integer $k$,
there is a polynomial $\mk_k(T)$ of degree $2k$
such that the generating function of positive multigraphs of excess $k$
is equal to
\[
    \mgpos_k(z) = \frac{\mk_k(T(z))}{(1-T(z))^{3k}}.
\]
\end{proposition}

\begin{proof}
We have seen that removing the vertices of degree $1$
in a positive multigraphs produces a positive multicore.
If furthermore the pairs of edges sharing a vertex of degree $2$
are merged to form only one edge,
and this operation is applied until there are
no more vertices of degree $2$,
a kernel $K$ is obtained
(see Figure~\ref{fig:path_of_trees}).
During this process, observe that
the numbers of vertices and edges removed are equal,
so the excess of the multigraph and of its kernel are the same.
The removed vertices can be divided into two groups:
\begin{itemize}
\item
a set of trees, rooted at the vertices of the kernel,
\item
a set of \emph{paths of trees}, defined below,
which replace the edges of the kernel.
\end{itemize}
A path of trees is as a sequence
\[
    (e_0, T_1, e_1, T_2, e_2, \ldots, T_t, e_{t}),
\]
where each $e_i$ is an edge label,
and each $T_i$ is a rooted tree.
An equivalent formulation is that a path of trees
is an unrooted tree, where two distinct vertices
have been ordered and removed,
while the edges linked to those two vertices are preserved.
A path of trees can be decomposed as an edge followed
by a sequence of pairs (rooted tree, edge),
\[
    (e_0, (T_1, e_1), (T_2, e_2), \ldots, (T_t, e_{t}))
\]
so the univariate generating function of paths of trees is
\[
    P(z) = \frac{1}{1-T(z)},
\]
where $T(z)$ denotes the univariate generating function of rooted trees.
As illustrated in Figure~\ref{fig:path_of_trees},
each positive multigraph of excess $k$
can be decomposed as a kernel of excess $k$,
where each edge is replaced by a path of trees,
and each vertex by a rooted tree
(specifically, the $r$th path of tree
replaces the kernel's edge of label $r$).
This implies
\[
    \mgpos_k(z) =
    \sum_{G \in \kernel_k}
    \frac{P(z)^{n(G) + k}}{2^{n(G) + k} (n(G)+k)!}
    \frac{T(z)^{n(G)}}{n(G)!}.
\]
Injecting the expression of the generating function
of paths of trees, this expression becomes
\[
    \mgpos_k(z) =
    \frac{1}{(1-T(z))^k}
    \sum_{G \in \kernel_k}
    \frac{1}{2^{n(G) + k} (n(G)+k)!}
    \frac{\left( \frac{T(z)}{1-T(z)} \right)^{n(G)}}{n(G)!}
    =
    \frac{\kernel_k \left( \frac{T(z)}{1-T(z)} \right)}{(1-T(z))^k}.
\]
Since $\kernel_k(T)$ is a polynomial of degree $2k$,
this last expression is equal to
\[
    \mgpos_k(z) =
    \frac{\mk_k(T(z))}{(1-T(z))^{3k}}
\]
if we set $\mk_k(T) = (1-T)^{2k} \kernel_k \left( \frac{T}{1-T} \right)$.
\end{proof}

\begin{figure}
\begin{center}
\includegraphics[scale=0.8]{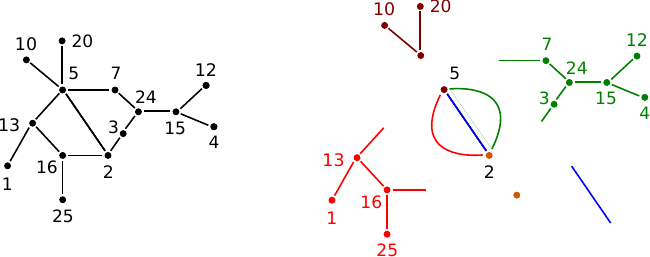}
\caption{For the sake of the clarity of the figure,
the orientations and labels of the edges
have been omitted.
Decomposition of a positive multigraph as a kernel,
where a rooted tree is attached to each vertex,
and each edge is replaced by a path of trees.
The rooted trees might have size $1$
(\eg orange tree attached to the vertex $2$),
and the path of tree might be reduced to one edge
(\eg the blue edge linking the vertices $2$ and $5$).}
\label{fig:path_of_trees}
\end{center}
\end{figure}

If we start with a graph,
remove the vertices of degree $0$ and $1$,
and merge into one edge each pair of edges sharing a vertex of degree $2$,
the result might contain loops and multiple edges,
as illustrated in Figure~\ref{fig:path_of_trees}.
When expressing the generating function of positive graphs
using the generating function of kernels,
we have to be careful about those loops and multiple edges,
since they are forbidden in graphs.
Given a multigraph $G$ and an integer $s \geq 2$,
let us define an \emph{induced $s$-multiple edge}
as a set of $s$ edges linking the same two distinct vertices,
and such that $G$ contains no other
edge linking those two vertices.
For any $s \geq 2$, the number of $s$-multiple edges of $G$ is denoted by $m_s(G)$.
The number of loops in $G$ in denoted by $\ell(G)$.
In this section, we add to the generating function of kernels
variables to mark the number of loops and induced $s$-multiple edges
\[
    \kernel(z,w,\nu,(\omega_s)_{s \geq 2}) =
    \sum_{G \in \kernel}
    \nu^{\ell(G)}
    \bigg( \prod_{s \geq 2} \omega_s^{m_s(G)} \bigg)
    \frac{w^{m(G)}}{2^{m(G)} m(G)!}
    \frac{z^{n(G)}}{n(G)!}.
\]
As usual, the generating function
of the corresponding family of excess $k$ is defined as
\[
    \kernel_k(z,\nu,(\omega_s)_{s \geq 2}) =
    [y^k] \kernel(z/y,y,\nu,(\omega_s)_{s \geq 2}).
\]

\begin{proposition}[already contained in Proposition~\ref{th:sgpos}] \label{th:sgpos_kernels}
For each positive $k$,
there is a polynomial $\sk_k(T)$
such that the generating function of positive graphs of excess $k$ is
\[
    \sgpos_k(z) = \frac{\sk(T(z))}{(1-T(z))^{3k}}.
\]
\end{proposition}

\begin{proof}
We use the same decomposition as in Lemma~\ref{th:mgpos_kernels}.
A positive graph of excess $k$ is a kernel of excess $k$,
where each edge is replaced by a path of trees,
and a rooted tree is added to each vertex.
The generating function of paths of trees
has been derived in the proof of Lemma~\ref{th:mgpos_kernels}
\[
    P(z) = \frac{1}{1-T(z)},
\]
and the generating functions of nonempty paths of trees
and paths of trees containing at least $2$ trees are
\[
    P^{>0}(z) = P(z) - 1 = \frac{T(z)}{1-T(z)},
    \qquad
    P^{>1}(z) = P(z) - 1 - T(z) = \frac{T(z)^2}{1-T(z)}.
\]
To ensure that the loops and double edges from the kernels
are not present in the positive graph,
the following conditions are added:
\begin{itemize}
\item
each loop is replaced by a path of trees containing at least $2$ trees,
\item
in each $s$-multiple edges, at least $s-1$ edges
are replaced by nonempty paths of trees.
\end{itemize}
This implies that the generating function $\sgpos_k(z)$ is equal to
\[
    \sum_{G \in \kernel_k}
    P^{>1}(z)^{\ell(G)}
    \bigg( \prod_{s \geq 2} \left(P^{>0}(z)^s + s P^{>0}(z)^{s-1}\right)^{m_s(G)} \bigg)
    \frac{P(z)^{n(G) + k - \ell(G) - \sum_{s \geq 2} m_s(G) s}}{2^{n(G) + k} (n(G)+k)!}
    \frac{T(z)^{n(G)}}{n(G)!},
\]
where $n(G) + k - \ell(G) - \sum_{s \geq 2} m_s(G) s$
is the number of edges from $G$
that are neither loops, nor multiple edges.
Injecting the expressions of the generating functions of paths of trees
into the previous equations gives
\[
    \sgpos_k(z) =
    \sum_{G \in \kernel_k}
    T(z)^{2 \ell(G)}
    \bigg( \prod_{s \geq 2} \left(T(z)^s + s T(z)^{s-1} (1-T(z)) \right)^{m_s(G)} \bigg)
    \frac{(1-T(z))^{-n(G) - k}}{2^{n(G) + k} (n(G)+k)!}
    \frac{T(z)^{n(G)}}{n(G)!},
\]
which is equal to
\[
    \sgpos_k(z) =
    \frac{\kernel_k \left(
        \frac{T(z)}{1-T(z)},
        T(z)^2,
        \left( T(z)^s + s T(z)^{s-1} (1-T(z)) \right)_{s \geq 2}
    \right)}
    {(1-T(z))^k}.
\]
According to Lemma~\ref{th:finite_kernels},
there is a finite number of kernels of excess $k$,
and they contain at most $2k$ vertices.
Thus, $\kernel(z,w,\nu,(\omega_s)_{s \geq 2})$ is a multinomial
of degree at most $2k$ in $z$.
Therefore,
\[
    (1-T)^{2k}
    \kernel_k \left(
        \frac{T}{1-T},
        T^2,
        \left( T^s + s T^{s-1} (1-T) \right)_{s \geq 2}
    \right)
\]
is a polynomial with respect to $T$, denoted by $\sk_k(T)$, and
\[
    \sgpos_k(z) = \frac{\sk(T(z))}{(1-T(z))^{3k}}.
\]
\end{proof}

    \subsection{Asymptotics of connected graphs and multigraphs with fixed excess}

A positive graph (or multigraph) is connected
if and only if its kernel is connected.
Thus, for any positive $k$,
the last two results of the previous section are also valid
for the generating functions of connected graphs (or multigraphs) of excess $k$.

\begin{proposition} \label{th:cskcmk}
For any positive integer $k$,
there exist polynomials $\csk_k(T)$ and $\cmk_k(T)$
equal to
\[
    \csk_k(T)
    =
    [y^k] \log \bigg( 1 + \sum_{\ell \geq 1} \sk_{\ell}(T) y^{\ell} \bigg),
    \qquad
    \cmk_k(T)
    =
    [y^k] \log \bigg( 1 + \sum_{\ell \geq 1} \mk_{\ell}(T) y^{\ell} \bigg),
\]
where the formulas for the polynomials $\sk_{\ell}(T)$ and $\mk_{\ell}(T)$
are provided in Propositions~\ref{th:sgpos} and~\ref{th:mgpos},
such that the generating function of connected graphs of excess $k$
and connected multigraphs of excess $k$ are equal to
\[
    \csg_k(z)
    =
    \frac{\csk_k(T(z))}{(1-T(z))^{3k}},
    \qquad
    \cmg_k(z)
    =
    \frac{\cmk_k(T(z))}{(1-T(z))^{3k}}.
\]
\end{proposition}

\begin{proof}
We start with the result of Proposition~\ref{th:csg_exact}
\[
    \csg_k(z) =
    [y^k] \log \bigg(
    1 + \sum_{\ell \geq 1} \sgpos_{\ell}(z) y^{\ell}
    \bigg).
\]
Applying Proposition~\ref{th:sgpos_kernels} to expression $\sgpos_{\ell}(z)$,
we obtain
\[
    \csg_k(z) =
    [y^k] \log \bigg(
    1 + \sum_{\ell \geq 1} \frac{\sk_{\ell}(T(z))}{(1-T(z))^{3\ell}} y^{\ell}
    \bigg)
\]
which, after the change of variable $y \mapsto (1-T(z)) y$, becomes
\[
    \csg_k(z) =
    \frac{[y^k] \log \bigg(
    1 + \sum_{\ell \geq 1} \sk_{\ell}(T(z)) y^{\ell}
    \bigg)}{(1-T(z))^{3k}}.
\]
The same proof applies to multigraphs,
starting with the result of Proposition~\ref{th:cmg_exact}
\[
    \cmg_k(z) =
    [y^k] \log \bigg(
    1 + \sum_{\ell \geq 1} \mgpos_{\ell}(z) y^{\ell}
    \bigg).
\]
\end{proof}

\cite{W80} proved that the generating function of connected graphs
of fixed excess is a rational function in $T(z)$.
He also derived a differential recurrence
characterizing the sequence of polynomials $(\csk_k(T))_k$.
This formula was the starting point of the proof of \cite{BCM90}
for the asymptotics of connected graphs,
that covers the case where the excess grows linearly with the number of vertices.
Our contribution here is to provide a new direct expression for those polynomials.
Let us recall here the chain of formulas applied to compute them.
\begin{itemize}
\item
The polynomial $\csk_k(T)$ is expressed in Proposition~\ref{th:cskcmk},
using the polynomials $(\sk_{\ell}(T))$, as
\[
    \csk_k(T)
    =
    [y^k] \log \bigg( 1 + \sum_{\ell \geq 1} \sk_{\ell}(T) y^{\ell} \bigg).
\]
\item
The polynomial $\sk_k(T)$ is expressed in Proposition~\ref{th:sgpos},
using the polynomials $(\pospatch_{\ell}(z,-1))$, as
\[
    \sk_k(T) =
    (1-T)^{3k}
    \sum_{\ell = 0}^k
    (2(k-\ell)-1)!!
    [x^{2(k-\ell)}]
    \frac{\pospatch_{\ell}(T e^x,-1) e^{- T \frac{e^x-1}{2} - T^2 \frac{e^{2x}-1}{4}}}
    {(1-T)^{k-\ell} \big(1- \frac{T}{1-T} \frac{e^x-1-x-x^2/2}{x^2/2}\big)^{k-\ell+1/2}}.
\]
\item
The formula for the polynomial $\pospatch_k(z,-1)$ is provided by Lemma~\ref{th:pospatch}
\[
    \pospatch_k(z,-1) =
    \sum_{n=0}^{3k}
    p_{n,k} z^n,
\]
where each coefficient $p_{n,k}$ has the following expression
\[
    p_{n,k} =
    \sum_{\ell,r,s,t}
    \binom{\binom{\ell}{2}}{k+\ell+t-r}
    \binom{s}{n-\ell-s-t}
    \frac{(-1)^{r+t} \ell^{2r}}{\ell! r! s! t! 2^{n+r-\ell-t}}.
\]
\end{itemize}
Applying classical tools from analytic combinatorics,
the asymptotics of connected graphs and multigraphs
with fixed excess is derived in the next theorem.

\begin{theorem}
The asymptotic numbers of connected graphs and multigraphs
with $n$ vertices and a fixed positive excess $k$ are
\begin{align*}
    \csg_{n,k} &=
    \frac{\sqrt{2 \pi} \csk_k(1)}{2^{3k/2} \Gamma(3k/2)}
    n^{n + 3k/2 - 1/2}
    \big( 1 + \bigO(n^{-1/2}) \big),
    \\
    \cmg_{n,k} &=
    \frac{2 \pi \cmk_k(1)}{\Gamma(3k/2)}
    \frac{2^{n-k/2}}{e^n}
    n^{2n + 5k/2}
    \big( 1 + \bigO(n^{-1/2}) \big),
\end{align*}
where the polynomials $\csk_k(T)$ and $\cmk_k(T)$ are defined in Proposition~\ref{th:cskcmk}.
\end{theorem}

\begin{proof}
It is a classic result (see \eg \cite{FS09}) that
the Cayley tree function $T(z)$ is analytic in the domain
\[
    \{z\ |\ |z| < R,\ z \neq 1/e,\ |\arg(z-1/e)| > \phi\}
\]
for some $R > 1/e$ and $0 < \phi < \pi/2$,
and has the following approximation
\[
    T(z) = 1 - \sqrt{2} \sqrt{1-ez} + \bigO(1-ez).
\]
Injecting this approximation into the expression
of the generating function of connected graphs from Proposition~\ref{th:cskcmk},
we obtain
\[
    \csg_k(z) = \frac{\csk_k(1)}{(2 (1-ez))^{3k/2}} \left( 1 + \bigO(\sqrt{1-ez})\right).
\]
Applying the singularity analysis \cite[Theorems VI.2 and VI.3]{FS09},
the asymptotics of the coefficient is extracted
\[
    \csg_{n,k} = n! [z^n] \csg_k(z) =
    n! \csk_k(1) \frac{e^n}{2^{3k/2}} \frac{n^{3k/2-1}}{\Gamma(3k/2)}
    \left( 1 + \bigO(n^{-1/2}) \right).
\]
Applying Stirling's formula to replace $n!$ with $n^n e^{-n} \sqrt{2 \pi n}$
concludes the proof for connected graphs.

The number of connected multigraphs with $n$ vertices and excess $k$ is
\[
    \cmg_{n,k} = n! 2^{n+k} (n+k)! [z^n] \frac{\mk_k(T(z))}{(1-T(z))^{3k}}.
\]
Applying the same proof as for graphs, we obtain
\[
    \cmg_{n,k} = n! 2^{n+k} (n+k)!
    \cmk_k(1) \frac{e^n}{2^{3k/2}} \frac{n^{3k/2-1}}{\Gamma(3k/2)}
    \left( 1 + \bigO(n^{-1/2}) \right).
\]
Again, an application of Stirling formula gives, for fixed $k$,
\[
    2^{n+k} (n+k)! \sim
    2^{n+k} n^{n+k} e^{-n} \sqrt{2 \pi n},
\]
and concludes the proof.
\end{proof}

A complete asymptotic expansion could be derived as well:
first compute more terms in the Newton-Puiseux expansion of the Cayley tree function
\[
    T(z) = \sum_{j = 0}^{d-1} e_j (1-ez)^{j/2} + \bigO(1-ez)^d,
\]
then inject them in the expressions from Proposition~\ref{th:cskcmk},
and apply a singularity analysis to derive $d$ terms
of the asymptotic expansion, in increasing powers of $n^{-1/2}$.
An alternative approach for the computation of
the complete asymptotic expansion of connected graphs with fixed excess
is provided by \cite{FSS04}.

\section{Conclusion}
\label{sec:conclusion}

\cite{JKLP93} and \cite{FSS04} started their analysis of multigraphs and graphs
with the expression of their generating functions
\begin{align*}
    \mg(z,w)
    &= \sum_{n \geq 0} e^{w n^2/2} \frac{z^n}{n!},
    \\
    \sg(z,w)
    &=
    \sum_{n \geq 0}
    (1+w)^{\binom{n}{2}}
    \frac{z^n}{n!}.
\end{align*}
As already mentioned, any graph (resp.\ multigraph) can be uniquely decomposed as
a set of trees, and a core (resp.\ multicore) where vertices are replaced by rooted trees.
If the manipulations of generating functions with negative exponents were following the same rules
as their nonnegative counterparts,
this decomposition would translate into the following generating function relations
\begin{align*}
    \mg(z/y,y)
    =
    \sum_{n \geq 0} e^{y n^2/2} \frac{(z/y)^n}{n!}
    &=
    e^{y^{-1} U(z)}
    \sum_{k \geq 0}
    (2k-1)!! [x^{2k}] \frac{y^k}{\left( 1 - T(z) \frac{e^x-1-x}{x^2/2} \right)^{k+1/2}},
    \\
    \sg(z/y,y)
    =
    \sum_{n \geq 0}
    (1+y)^{\binom{n}{2}}
    \frac{(z/y)^n}{n!}
    &=
    e^{y^{-1} U(z)}
    \sum_{k \geq 0}
    \sum_{\ell=0}^k
    (2(k-\ell)-1)!! [x^{2(k-\ell)}] \frac{P_{\ell}(T(z),-1) y^k}{\left( 1 - T(z) \frac{e^x-1-x}{x^2/2} \right)^{k-\ell+1/2}}.
\end{align*}
The contribution of trees is then clearly separated from the rest,
contrary to the first expressions.
Those two relations might be of mathematical interest on their own
(should their domain of validity be defined).
It would be interesting to prove them analytically,
continuing the work of \cite{FSS04}.

    \paragraph{Acknowledgments.}

We thank the two anonymous referees for their attentive proofreading of the paper.
Their helpful remarks improved the structure and clarity of the paper,
and they corrected errors in the formulation of the asymptotics.
We also thank Michael Borinsky for his helpful suggestions
on the asymptotic expansion of the coefficients
of divergent series (Appendix~\ref{sec:divergent_series}),
and Adrien Sauvaget and Nathana\"el Fijalkow
for the proof of Lemma~\ref{th:morse}.

\bibliographystyle{abbrvnat}
\bibliography{/home/elie/research/articles/bibliography/biblio}

\begin{thebibliography}{26}
\providecommand{\natexlab}[1]{#1}
\providecommand{\url}[1]{\texttt{#1}}
\expandafter\ifx\csname urlstyle\endcsname\relax
  \providecommand{\doi}[1]{doi: #1}\else
  \providecommand{\doi}{doi: \begingroup \urlstyle{rm}\Url}\fi

\bibitem[Bender(1975)]{Be75}
E.~A. Bender.
\newblock An asymptotic expansion for the coefficients of some formal power
  series.
\newblock \emph{Journal of the London Mathematical Society}, 2\penalty0
  (3):\penalty0 451–458, 1975.

\bibitem[Bender et~al.(1990)Bender, Canfield, and McKay]{BCM90}
E.~A. Bender, E.~R. Canfield, and B.~D. McKay.
\newblock The asymptotic number of labeled connected graphs with a given number
  of vertices and edges.
\newblock \emph{Random Structures and Algorithm}, 1:\penalty0 129--169, 1990.

\bibitem[Bergeron et~al.(1997)Bergeron, Labelle, and Leroux]{BLL97}
F.~Bergeron, G.~Labelle, and P.~Leroux.
\newblock \emph{Combinatorial Species and Tree-like Structures}.
\newblock Cambridge University Press, 1997.

\bibitem[Bollob\'as(1980)]{Bo80}
B.~Bollob\'as.
\newblock A probabilistic proof of an asymptotic formula for the number of
  labelled regular graphs.
\newblock \emph{European Journal of Combinatorics}, 1:\penalty0 311--316, 1980.

\bibitem[Borinsky(2017{\natexlab{a}})]{Bo16}
M.~Borinsky.
\newblock Generating asymptotics for factorially divergent sequences.
\newblock \emph{Formal Power Series and Algebraic Combinatorics (FPSAC)},
  2017{\natexlab{a}}.

\bibitem[Borinsky(2017{\natexlab{b}})]{Bo17}
M.~Borinsky.
\newblock Renormalized asymptotic enumeration of feynman diagrams.
\newblock \emph{Annals of Physics}, 385:\penalty0 95--135, 2017{\natexlab{b}}.

\bibitem[Collet et~al.(2017)Collet, de~Panafieu, Gardy, Gittenberger, and
  Ravelomanana]{EdPCGGR17}
G.~Collet, E.~de~Panafieu, D.~Gardy, B.~Gittenberger, and V.~Ravelomanana.
\newblock Threshold functions for small subgraphs: an analytic approach.
\newblock \emph{Eurocomb}, 2017.

\bibitem[de~Panafieu(2014)]{ElieThesis}
E.~de~Panafieu.
\newblock \emph{Analytic Combinatorics of Graphs, Hypergraphs and Inhomogeneous
  Graphs}.
\newblock PhD thesis, Universit\'e Paris-Diderot, Sorbonne Paris-Cit\'e, 2014.

\bibitem[de~Panafieu(2015{\natexlab{a}})]{EdP15}
E.~de~Panafieu.
\newblock Enumeration and structure of inhomogeneous graphs.
\newblock \emph{proceedings of the International Conference on Formal Power
  Series and Algebraic Combinatorics (FPSAC), poster presentation}, page~12,
  2015{\natexlab{a}}.

\bibitem[de~Panafieu(2015{\natexlab{b}})]{EdPhp15}
E.~de~Panafieu.
\newblock Phase transition of random non-uniform hypergraphs.
\newblock \emph{Journal of Discrete Algorithms}, 31\penalty0 (0):\penalty0 26
  -- 39, 2015{\natexlab{b}}.

\bibitem[de~Panafieu(2016)]{EdP16}
E.~de~Panafieu.
\newblock Counting connected graphs with large excess.
\newblock \emph{proceedings of the International Conference on Formal Power
  Series and Algebraic Combinatorics (Fpsac16)}, 2016.

\bibitem[de~Panafieu and Ramos(2016)]{EdPR16}
E.~de~Panafieu and L.~Ramos.
\newblock Graphs with degree constraints.
\newblock \emph{proceedings of the Meeting on Analytic Algorithmics and
  Combinatorics (Analco16)}, 2016.

\bibitem[de~Panafieu and Ravelomanana(2015)]{PR14}
E.~de~Panafieu and V.~Ravelomanana.
\newblock Analytic description of the phase transition of inhomogeneous
  multigraphs.
\newblock \emph{European Journal of Combinatorics}, 2015.

\bibitem[Erd\H{o}s and R\'enyi(1960)]{ER60}
P.~Erd\H{o}s and A.~R\'enyi.
\newblock On the evolution of random graphs.
\newblock \emph{Publication of the Mathematical Institute of the Hungarian
  Academy of Sciences}, 5:\penalty0 17, 1960.

\bibitem[Flajolet and Sedgewick(2009)]{FS09}
P.~Flajolet and R.~Sedgewick.
\newblock \emph{Analytic Combinatorics}.
\newblock Cambridge University Press, 2009.

\bibitem[Flajolet et~al.(1989)Flajolet, Knuth, and Pittel]{FKP89}
P.~Flajolet, D.~E. Knuth, and B.~Pittel.
\newblock The first cycles in an evolving graph.
\newblock \emph{Discrete Mathematics}, 75\penalty0 (1-3):\penalty0 167--215,
  1989.

\bibitem[Flajolet et~al.(2004)Flajolet, Salvy, and Schaeffer]{FSS04}
P.~Flajolet, B.~Salvy, and G.~Schaeffer.
\newblock {A}iry phenomena and analytic combinatorics of connected graphs.
\newblock \emph{Electronic Journal of Combinatorics}, 11\penalty0 (1), 2004.

\bibitem[Janson et~al.(1993)Janson, Knuth, \L{}uczak, and Pittel]{JKLP93}
S.~Janson, D.~E. Knuth, T.~\L{}uczak, and B.~Pittel.
\newblock The birth of the giant component.
\newblock \emph{Random Structures and Algorithms}, 4\penalty0 (3):\penalty0
  233--358, 1993.

\bibitem[\L{}uczak(1990)]{L90}
T.~\L{}uczak.
\newblock On the number of sparse connected graphs.
\newblock \emph{Random Structures and Algorithms}, 2:\penalty0 171--173, 1990.

\bibitem[Pemantle and Wilson(2013)]{PW13}
R.~Pemantle and M.~C. Wilson.
\newblock \emph{Analytic Combinatorics in Several Variables}.
\newblock Cambridge University Press, New York, NY, USA, 2013.

\bibitem[Pittel and Wormald(2005)]{PW05}
B.~Pittel and N.~C. Wormald.
\newblock Counting connected graphs inside-out.
\newblock \emph{Journal of Combinatorial Theory, Series B}, 93\penalty0
  (2):\penalty0 127--172, 2005.

\bibitem[R\'enyi(1959)]{R59}
A.~R\'enyi.
\newblock On connected graphs {I}.
\newblock \emph{publication of the mathematical institute of the hungarian
  academy of sciences}, 4\penalty0 (159):\penalty0 385--388, 1959.

\bibitem[SageMath(2016)]{sagemath}
SageMath.
\newblock \emph{{S}ageMath, the {S}age {M}athematics {S}oftware {S}ystem
  ({V}ersion 7.1)}, 2016.
\newblock {\tt http://www.sagemath.org}.

\bibitem[van~der Hofstad and Spencer(2006)]{HS06}
R.~van~der Hofstad and J.~Spencer.
\newblock Counting connected graphs asymptotically.
\newblock \emph{European Journal on Combinatorics}, 26\penalty0 (8):\penalty0
  1294--1320, 2006.

\bibitem[Wormald(1978)]{Wo78}
N.~Wormald.
\newblock \emph{Some problems in the enumeration of labelled graphs}.
\newblock Newcastle University, 1978.

\bibitem[Wright(1980)]{W80}
E.~M. Wright.
\newblock The number of connected sparsely edged graphs~{III}: Asymptotic
  results.
\newblock \emph{Journal of Graph Theory}, 4\penalty0 (4):\penalty0 393--407,
  1980.

\end{thebibliography}
\appendix

\section{Saddle-point method}
\label{sec:saddle_point}

In this section, vectors are denoted using bold letters, such as $\vz$.
We use the classical notation
\[
    \vz^{\vm} = \prod_j z_j^{m_j},
\]
where the vectors $\vz$ and $\vm$
are assumed to have the same length.

The main result of this section is Lemma~\ref{th:appendix_mgpos_asymptotics}.
It is a corollary of Theorem~\ref{th:large_powers},
which provides a complete asymptotic expansion
for coefficient extractions of the form
\[
    [\vz^{\vm}] A(\vz) B(\vz)^n
\]
when the coefficients of the vector $\vm$
go to infinity linearly with $n$.
This theorem is an application of the results from Section~5 of \cite{PW13},
and some classical properties on analytic functions
(Morse and Daffodil Lemmas).

    \subsection{Preliminaries} \label{sec:preliminaries}

Let us first recall the Morse Lemma,
which can be found in the book \cite{PW13} for example.

\begin{lemma}
When $\phi(\vx)$ is a function analytic in a neighborhood $\mC$ of $\vzero$,
such that $\phi(\vzero) = 0$, $\partial_{x_j} \phi(\vzero) = 0$ for all $j$
and the Hessian matrix $\mH_{\phi}(\vzero)$ of $\phi$ at $\vzero$ is nonsingular,
then there exists a biholomorphic change of variable $\vx = \psi(\vy)$
that maps $\mC$ to a neighborhood of $\vzero$,
and such that
\[
    \phi(\psi(\vy)) = \frac{1}{2} \sum_{j=1}^t y_j^2.
\]
The Jacobian matrix $J_{\psi}(\vy)$ of $\psi$
is linked to the Hessian matrix of $\phi$ by the relation
\[
    \det(J_{\psi}(\vzero))^2 = | \det( \mH_{\phi}(\vzero) ) |^{-1}.
\]
\end{lemma}

We will need the following parameterized generalization.
The first part of the proof follows the one of \cite{PW13}.
We thank Adrien Sauvaget and Nathanaël Fijalkow for their help
on the proof of the second part.

\begin{lemma} \label{th:morse}
When $\phi(\vx, \vlambda)$ is a function analytic for $\vx$ a neighborhood $\mC \subset \reals^t$ of $\vzero$,
and $\vlambda$ in a compact set $\mK$,
such that for all $\vlambda$ in $\mK$,
\begin{itemize}
\item
$\phi(\vzero, \vlambda) = 0$,
\item
$\partial_{x_j} \phi(\vzero, \vlambda) = 0$ for all $j$,
\item
and the Hessian matrix $\mH_{\phi}(\vzero, \vlambda)$ of $\vx \mapsto \phi(\vx, \vlambda)$ at $\vzero$ is nonsingular,
\end{itemize}
then there exists a function $\psi(\vy, \vlambda)$
analytic for $\vy $ in a neighborhood of $\vzero$ and $\vlambda$ in $\mK$,
such that for each $\vlambda$ in $\mK$,
$\vx = \psi(\vy, \vlambda)$ is a biholomorphic change of variable 
that maps $\mC$ to a neighborhood of $\vzero$,
and such that
\[
    \phi(\psi(\vy,\vlambda),\vlambda) = \frac{1}{2} \sum_{j=1}^t y_j^2.
\]
The Jacobian matrix $J_{\psi}(\vy,\vlambda)$ of $\vx \mapsto \phi(\vx, \vlambda)$
is linked to the Hessian matrix of $\vx \mapsto \phi(\vx, \vlambda)$ by the relation
\[
    \det(J_{\psi}(\vzero,\vlambda))^2 = | \det( \mH_{\phi}(\vzero,\vlambda) ) |^{-1}.
\]
Furthermore, for any neighborhood $U$ of $\vzero$,
there is a neighborhood $V$ of $\vzero$
such that for all $\vlambda$ in $\mK$,
we have $\psi(V, \vlambda) \subset U$.
\end{lemma}

\begin{proof}
The first step is to construct analytic functions
$(\phi_{j,k}(\vx, \vlambda))_{1 \leq j, k \leq t}$ such that
\begin{equation} \label{eq:def_phi_j_k}
  \phi(\vx, \vlambda) = \sum_{1 \leq j, k \leq t} x_j x_k \phi_{j,k}(\vx, \vlambda).
\end{equation}
To do so, we set for all $1 \leq j, k \leq t$
\[
  x_j x_k \phi_{j,k}(\vx, \vlambda) =
  \sum_{r_1 + \cdots + r_t \geq 2}
  \frac{r_j (r_k - \delta_{j,k})}{(r_1 + \cdots + r_t) (r_1 + \cdots + r_t - 1)}
  [\vy^{\vr}] \phi(\vy, \vlambda)
  \vx^{\vr},
\]
where $\delta_{j,k}$ is equal to $1$ if $j=k$, and to $0$ otherwise.
Observe that the right hand-side is divisible by $x_j x_k$,
so $\phi_{j,k}(\vx, \vlambda)$ is indeed analytic,
and for all $j$, $k$, we have $\phi_{j,k} = \phi_{k,j}$.
Furthermore, since
\[
  \sum_{1 \leq j, k \leq t}
  \frac{r_j (r_k - \delta_{j,k})}{(r_1 + \cdots + r_t) (r_1 + \cdots + r_t - 1)}
  = 1,
\]
Equation~\eqref{eq:def_phi_j_k} is satisfied.
Observe that Equation~\eqref{eq:def_phi_j_k} implies
\[
  \phi_{j,k}(\vzero, \vlambda)  = \frac{1}{2} \mH_{j,k}(\vzero, \vlambda)
\]
for all $j$, $k$.

The second step of the proof is an induction.
Let us assume that $\phi_{j,j}(\vzero, \vlambda)$
does not vanish for any $\vlambda$ from $\mK$ and $1 \leq j \leq t$.
Then $\phi_{1,1}(\vx, \vlambda)^{-1}$
and a branch of $\phi_{1,1}(\vx, \vlambda)^{1/2}$ are analytic.
Set
\[
  y_1 :=
  \sqrt{2 \phi_{1,1}(\vx, \vlambda)}
  \bigg(
  x_1 + \sum_{k = 2}^t \frac{x_k \phi_{1,k}(\vx, \vlambda)}{\phi_{1,1}(\vx, \vlambda)}
  \bigg).
\]
Expanding the square of the right side,
we see that $y_1^2 / 2$ and $\phi(\vx, \vlambda)$
agree on all terms of total degree at most one in $x_2, \ldots, x_t$.
Thus, there are analytic functions $h_{j,k}(\vx, \vlambda)$ such that
\[
  \phi(\vx, \vlambda) =
  \frac{1}{2} y_1^2 + \sum_{j,k \geq 2} x_j x_k h_{j,k}(\vx, \vlambda),
\]
with $h_{j,k}(\vzero, \vlambda) = \mH_{j,k}(\vzero, \vlambda) / 2$.
Similarly, if
\[
  \phi(\vx, \vlambda) =
  \frac{1}{2} \sum_{j=1}^{r-1} y_j^2
  + \sum_{j,k \geq r} x_j x_k h_{j,k}(\vx, \vlambda),
\]
then setting
\[
  y_r :=
  \sqrt{2 \phi_{r,r}(\vx, \vlambda)}
  \bigg(
  x_r + \sum_{k = r+1}^t \frac{x_k h_{r,k}(\vx, \vlambda)}{h_{r,r}(\vx, \vlambda)}
  \bigg)
\]
gives
\[
  \phi(\vx, \vlambda) =
  \frac{1}{2} \sum_{j=1}^{r} y_j^2
  + \sum_{j,k \geq r+1} x_j x_k \tilde{h}_{j,k}(\vx, \vlambda)
\]
for some analytic functions $\tilde{h}_{j,k}(\vx, \vlambda)$.
By induction, we arrive at
\[
  \phi(\vx, \vlambda) = \frac{1}{2} \sum_{j=1}^t y_j^2,
\]
finishing the proof of the first two assertions of the lemma in the case where $\mH_{j,j}(\vzero, \vlambda)$
does not vanish for any $\lambda$ in $\mK$.

If some $\mH_{j,j}(\vzero, \vlambda)$ vanishes,
because $\mH(\vzero, \vlambda)$ is nonsingular,
we may always find some unitary map $U(\vlambda)$
such that the Hessian
\[
  U(\vlambda)^T \mH(\vzero, \vlambda) U(\vlambda)
\]
of $\phi(U(\vlambda) \vx, \vlambda)$ has no vanishing diagonal entries.
We know there is a $\psi_0$ such that
\[
  \phi(U(\vlambda) \psi_0(\vy, \vlambda), \vlambda) =
  \frac{1}{2} \sum_{j=1}^t y_j^2,
\]
and taking $\psi(\vy, \vlambda) = U(\vlambda) \psi_0(\vy, \vlambda)$
finishes the construction of $\psi$ in this case.

Now, let us prove the last assertion of the Lemma.
Let $d \psi(\vzero,\vlambda)$ denote the differential
of $\vy \mapsto \psi(\vy, \vlambda)$ at $\vzero$.
It is nonsingular for all $\vlambda \in \mK$.
Since $\psi$ is analytic, the function
\[
  \vlambda \mapsto
  \sup_{\|\vx\| \leq 1}
  \frac{\| d\psi(\vzero, \vlambda) \cdot \vx \|}{\|\vx\|}
\]
is continuous on the compact $\mK$,
so it has a maximum $M$. 
The Taylor expansion of $\vy \mapsto \psi(\vy, \vlambda)$ at the origin is
\begin{equation} \label{eq:taylorpsi}
  \psi(\vy, \vlambda) =
  \psi(\vzero, \vlambda) +
  d\psi(\vzero, \vlambda) \cdot \vy + R(\vy, \vlambda),
\end{equation}
where $R(\vy, \vlambda)$ denotes the rest of the expansion.
This rest has the following bound
\[
  \|R(\vy, \vlambda)\| \leq
  c_{\vlambda} \|\vy\|^2,
\]
where $c_{\vlambda}$ is a linear combination
of the second derivatives of $\vy \mapsto \psi(\vy, \vlambda)$.
Since $\vlambda$ stays in the compact set $\mK$,
the constant
\[
  C = \sup_{\vlambda \in \mK} c_{\vlambda}
\]
is well defined.
Injecting this bound in Equation~\eqref{eq:taylorpsi},
and the relation $\psi(\vzero, \vlambda) = \vzero$,
we obtain
\[
  \| \psi(\vy, \vlambda) \|
  \leq
  \| d \psi(\vzero, \vlambda) \cdot \vy \| + C \|\vy\|^2
  \leq
  M \|\vy\| + C \|\vy\|^2,
\]
where the right hand-side is independent of $\vlambda$.
Let $r > 0$ be the radius of an open ball
centered at the origin and contained in $U$.
Then, according to the previous inequality,
for any $\vx$ satisfying
\[
  \| \vx \| < \min \left(1, \frac{r}{M + C} \right),
\]
we have $\|\psi(\vx, \vlambda)\| < r$,
so $\psi(\vx, \vlambda) \in U$.
Thus, defining $V$ as the open ball
of radius $\min \left(1, \frac{r}{M + C} \right)$
and centered at the origin,
we have, for all $\vlambda$ in $\mK$,
$\psi(V, \vlambda) \subset U$.
\end{proof}

Our next lemma provides an effective way to locate
the maximum of the absolute value of a multivariate analytic function
with nonnegative coefficients at the origin.
It is inspired by the Daffodil Lemma (Lemma~IV.1 from \cite{FS09}).
The \emph{support} $\supp(B)$ of a function $B(\vz)$ of $t$ variables
is the set of exponent vectors of nonzero monomials
in its series expansion at the origin
\[
    \supp(B) = \{\vn \in \naturals^t\ |\ [\vz^{\vn}] B(\vz) \neq 0\}.
\]
We say that the support \emph{spans the identity matrix}
if it contains $2t$ vectors $\vn_1, \ldots, \vn_{2t}$
(not necessarily distinct) such that the matrix
\[
    \begin{pmatrix}
    \vn_1 - \vn_2\\
    \vn_3 - \vn_4\\
    \vdots\\
    \vn_{2t-1} - \vn_{2t}
    \end{pmatrix}
\]
is the identity matrix.

\begin{lemma} \label{th:multivariate_daffodil}
Let $B(\vz)$ denote a function of $t$ variables,
analytic in a neighborhood $\mC$ of the origin,
with nonnegative coefficients at the origin,
and which support spans the identity matrix.
Then on any closed torus $\mathcal{T}$ of radii $\vzeta$ contained in $\mC$,
\[
    \mathcal{T} = \{ \vz \in \complex^t\ |\ \forall j,\ |z_j| \leq \zeta_j\},
\]
$|B(\vz)|$ reaches its unique maximum at the point $\vzeta$.
\end{lemma}

\begin{proof}
Consider a vector $\vz$ in the torus of radii $\vzeta$,
then, by the triangle inequality,
\[
    |B(\vz)| =
    \bigg|
    \sum_{\vn \in \supp(B)} b_{\vn} \vz^{\vn}
    \bigg|
    \leq
    \sum_{\vn \in \supp(B)} b_{\vn} |\vz^{\vn}|
    \leq
    \sum_{\vn \in \supp(B)} b_{\vn} \vzeta^{\vn}
    =
    B(\vzeta).
\]
Now assume that $|B(\vz)| = B(\vzeta)$, so
\[
    \bigg|
    \sum_{\vn \in \supp(B)} b_{\vn} \vz^{\vn}
    \bigg|
    =
    \sum_{\vn \in \supp(B)} b_{\vn} \vzeta^{\vn}.
\]
According to the strong triangle inequality,
this implies $|\vz^{\vn}| = \vzeta^{\vn}$
for all $\vn$ in $\supp(B)$,
and the complex numbers $\vz^{\vn}$ are aligned.
Since there exist $2t$ vectors $\vn_1, \ldots, \vn_{2t}$ in $\supp(B)$
that span the identity matrix,
for any $1 \leq j \leq t$, there is a vector $\vn$ in $\supp(B)$
with a positive $j$th coefficient $n_j > 0$.
Since $\vz$ is in $\mathcal{T}$, we have
\[
    |\vz^{\vn}|
    =
    \prod_{k=1}^t |z_k|^{n_k}
    \leq
    |z_j|^{n_j} \prod_{k \neq j} \zeta_k^{n_k}.
\]
Since $|\vz^{\vn}| = \vzeta^{\vn}$, this implies $|z_j| = \zeta_j$,
and this holds for any $1 \leq j \leq t$.
Thus, $\vz$ is on the boundary of $\mathcal{T}$.
Let $\vtheta$ denote the vector $(\arg(z_1), \ldots, \arg(z_t))$,
where $\arg(z)$ is the principal argument of $z$.
The alignment implies that for all $\vn$, $\vm$ in $\supp(B)$,
$\vn \vtheta$ is equal to $\vm \vtheta$ modulo $2\pi$.
Thus, we have
\[
    \begin{pmatrix}
    \vn_1 - \vn_2\\
    \vn_3 - \vn_4\\
    \vdots\\
    \vn_{2t-1} - \vn_{2t}
    \end{pmatrix}
    \vtheta
    = 0 \mod 2 \pi.
\]
The matrix on the left side is the identity matrix,
and $\vtheta$ has its coefficients in $[0,2\pi[$,
so the previous relation implies $\vtheta = \vzero$.
Thus $\vz = \vzeta$, and $B$ reaches its unique maximum at $\vzeta$.
\end{proof}

    \subsection{Laplace-Fourier integrals and Large Powers Theorem}

The following Lemma is a combination
of Theorems $5.1.1$ and $5.1.2$ from \cite{PW13},
and of their proofs.
Given an integer vector $\vr = (r_1, \ldots, r_t)$,
we introduce the notation
\[
    (2 \vr - 1)!! = \prod_{j=1}^t (2 r_j - 1)!! =
    \prod_{j=1}^t \frac{(2r_j)!}{2^{r_j} r_j!}.
\]

\begin{lemma}
Consider a compact neighborhood $\mN$ of $\vzero$ in $\reals^t$,
and two analytic functions $A(\vx)$ and $\phi(\vx)$ on $\mN$.
Assume $A(\vzero) \neq 0$,
$\phi(\vzero) = 0$, $\partial_j \phi(\vzero) = 0$ for all $1 \leq j \leq t$,
and that the real part of $\phi(\vx)$
is positive on $\mN \setminus \{\vzero\}$.
Assume that the Hessian matrix $\mH_{\phi}(\vzero)$ of $\phi$ at $\vzero$ is nonsingular,
and let us use the notations $\psi$ and $J_{\psi}$ from Lemma~\ref{th:morse}.
For any integer $d$, we have the following asymptotic expansion
\[
    \int_{\mN} A(\vx) e^{- n \phi(\vx)} d \vx =
    \left( \frac{2 \pi}{n} \right)^{t/2} 
    \Bigg(
    \sum_{r=0}^{d-1}
    c_r
    n^{-r}
    + \bigO(n^{-d})
    \Bigg),
\]
where each coefficient $c_r$ is equal to
\[
    c_r =
    \sum_{\substack{r_1 + \cdots + r_t = r\\ \forall j,\ r_j \geq 0}}
    (2 \vr-1)!!
    [\vy^{2\vr}] A(\psi(\vy)) \det( J_{\psi}(\vy) ),
\]
and in particular $c_0 = A(0) |\det(\mH_{\phi}(\vzero))|^{-1/2}$.
\end{lemma}

We will need the following parameterized generalization.
The proof mainly follows the ones of Theorems $5.1.1$ and $5.1.2$
from \cite{PW13}.

\begin{lemma} \label{th:laplace_fourier}
Consider a compact neighborhood $\mN$ of $\vzero$ in $\reals^t$,
a compact set $\mK$ of some power of $\reals$,
and two analytic functions $A(\vx, \vlambda)$ and $\phi(\vx, \lambda)$ on $\mN \times \mK$.
Assume that for all $\vlambda$ in $\mK$, we have
$A(\vzero, \vlambda) \neq 0$,
$\phi(\vzero, \vlambda) = 0$, $\partial_{x_j} \phi(\vzero, \vlambda) = 0$ for all $1 \leq j \leq t$,
and that the real part of $\phi(\vx, \vlambda)$
is positive for any $\vx$ in $\mN \setminus \{\vzero\}$.
Assume that the Hessian matrix $\mH_{\phi}(\vzero, \vlambda)$ of $\vx \mapsto \phi(\vx, \vlambda)$ at $\vx = \vzero$ is nonsingular,
and let us use the notations $\psi$ and $J_{\psi}$ from Lemma~\ref{th:morse}.
For any integer $d$, we have the following asymptotic expansion,
which holds uniformly for $\vlambda$ in $\mK$
\[
    \int_{\mN} A(\vx, \vlambda) e^{- n \phi(\vx, \vlambda)} d \vx =
    \left( \frac{2 \pi}{n} \right)^{t/2} 
    \Bigg(
    \sum_{r=0}^{d-1}
    c_r(\vlambda)
    n^{-r}
    + \bigO(n^{-d})
    \Bigg),
\]
where each coefficient $c_r(\vlambda)$ is a function analytic on $\mK$, equal to
\[
    c_r(\vlambda) =
    \sum_{\substack{r_1 + \cdots + r_t = r\\ \forall j,\ r_j \geq 0}}
    (2 \vr-1)!!
    [\vy^{2\vr}] A(\psi(\vy, \vlambda), \vlambda) \det( J_{\psi}(\vy, \vlambda) ),
\]
and in particular $c_0(\vlambda) = A(0, \vlambda) |\det(\mH_{\phi}(\vzero, \vlambda))|^{-1/2}$.
\end{lemma}

\begin{proof}
Since the real part of $\phi(\vx, \vlambda)$ is positive
for any $\vlambda$ in $\mK$ and any $\vx$ in $\mN \setminus \{\vzero\}$,
and $\phi(\vzero, \vlambda) = 0$,
the function $\vx \mapsto |e^{- \phi(\vx, \vlambda)}|$
reaches its maximum $1$ only at $\vx = \vzero$.
Thus, we can find a small enough neighborhood $\mM$ of the origin and a value $0 < \delta < 1$
such that
\begin{itemize}
\item
for all $\vlambda$ in $\mK$ and $\vx$ in $\mN \setminus \mM$, we have $|e^{- \phi(\vx, \vlambda)}| < \delta$,
\item
the change of variable $\vx = \psi(\vy, \vlambda)$ is applicable on $\mM$ and send it to a neighborhood $\mL_{\vlambda}$ of the origin,
\end{itemize}
so
\[
  \int_{\mN} A(\vx, \vlambda) e^{- n \phi(\vx, \vlambda)} d \vx =
  \int_{\mM} A(\vx, \vlambda) e^{- n \phi(\vx, \vlambda)} d \vx
  + \int_{\mN \setminus \mM} A(\vx, \vlambda) e^{- n \phi(\vx, \vlambda)} d \vx.
\]
Since $A(\vx, \vlambda)$ is bounded on the compact set $\mN \times \mK$,
he absolute value of the second integral is bounded by
\[
  \int_{\mN \setminus \mM} \left| A(\vx, \vlambda) e^{- n \phi(\vx, \vlambda)} \right| d \vx
  = \bigO(\delta^n),
\]
which is exponentially small compared to the final result,
hence negligible in the asymptotic expansion.
The change of variable $\vx = \psi(\vy, \vlambda)$ is applied in the first integral,
which becomes
\[
  \int_{\mL_{\vlambda}} A(\psi(\vy, \vlambda), \vlambda) \det(J_{\psi}(\vy, \vlambda)) e^{- \frac{n}{2} \sum_{j=1}^t y_j^2} d \vy
\]
According to the last assertion of Lemma~\ref{th:morse},
their is a neighborhood $\mL$ of the origin such that
\[
  \mL \subset \bigcap_{\vlambda \in \mK} \mL_{\vlambda}.
\]
Up to an exponentially small additional term,
we can again reduce the domain of summation to $\mL$.
We choose $\mL$ small enough to ensure
that the following Taylor expansion holds in $\mL$
uniformly with respect to $\vlambda$ in $\mK$
\[
  A(\psi(\vy, \vlambda), \vlambda) \det(J_{\psi}(\vy, \vlambda)) =
  \sum_{k_1 + \cdots + k_t < 2 d}
  \left( [\vx^{\vk}] A(\psi(\vx, \vlambda), \vlambda) \det(J_{\psi}(\vx, \vlambda)) \right)
  \vy^{\vk}
  + \bigO \bigg(
  \sum_{k_1 + \cdots + k_t = 2 d}
  \vy^{\vk}
  \bigg).
\]
According to Lemma~\ref{th:morse}, the constant term of this expansion is equal to
$A(\vzero, \vlambda) | \det(\mH_{\phi}(\vzero, \vlambda)) |^{-1/2}$.
Injecting this expansion in the integral, we obtain for
$\int_{\mN} A(\vx, \vlambda) e^{- n \phi(\vx, \vlambda)} d \vx$
the expression
\begin{equation}
\label{eq:sum_int_laplace}
  \sum_{k_1 + \cdots + k_t < 2 d}
  [\vx^{\vk}] A(\psi(\vx, \vlambda), \vlambda) \det(J_{\psi}(\vx, \vlambda))
  \int_{\mL} \vy^{\vk} e^{- \frac{n}{2} \sum_{j=1}^t y_j^2} \vy
  +
  \int_{\mL}
  \bigO \bigg(
  \sum_{k_1 + \cdots + k_t = 2 d}
  \vy^{\vk}
  \bigg)
  e^{- \frac{n}{2} \sum_{j=1}^t y_j^2} \vy
\end{equation}
up to the addition of exponentially small terms.
On any neighborhood $U$ of $0$, when $k$ is a fixed integer and $n$ tends to infinity,
the following truncated Gaussian-like integral is equal, up to an exponentially small term, to
\[
  \int_{U} x^k e^{- \frac{n}{2} x^2} dx =
  \begin{cases}
  0 & \text{if $k$ is odd,}\\
  (2r-1)!! n^{-r} & \text{if $k = 2 r$,}
  \end{cases}
  + \text{(exponentially small term).}
\]
Thus, up to an exponentially negligible term, we have
\[
  \int_{\mL} \vy^{\vk} e^{- \frac{n}{2} \sum_{j=1}^t y_j^2} \vy =
  \begin{cases}
  0 & \text{if at least one of the $k_j$ is odd,}\\
  (2 \vr-1)!! n^{-(r_1 + \cdots + r_t)} & \text{if $\vk = 2 \vr$.}
  \end{cases}  
\]
Injecting this result in Equation~\eqref{eq:sum_int_laplace} concludes the proof.
\end{proof}

\begin{theorem} \label{th:large_powers}
Consider two fixed integers $d$ and $t$,
a compact subset $\mK$ of $\reals_{>0}^t$,
and integers $n$ and $\vm := (m_1, \ldots, m_t)$ going to infinity
such that $\vm / n$ stays in $\mK$.
Let $B(\vz)$ be a function satisfying the following conditions.
\begin{itemize}
\item
$B(\vz)$ is analytic on a closed torus $\mC$ from $\complex^t$ containing the origin.
\item
Its series expansion at the origin has nonnegative coefficients.
\item
Its support spans the identity matrix (see the definition before Lemma~\ref{th:multivariate_daffodil}).
\item
There is an analytic function $\vzeta(\vm/n) := (\zeta_1(\vm/n), \ldots, \zeta_t(\vm/n))$
from $\mK$ to $\reals_{\geq 0}^t \cap \mC$ such that
for all $1 \leq j \leq t$,
\[
    \frac{\zeta_j(\vm/n) \partial_{\zeta_j(\vm/n)} B(\vzeta(\vm/n))}{B(\vzeta(\vm/n))} = m_j/n.
\]
In the following, $\vzeta$ stands for $\vzeta(\vm/n)$.
\item
The function $\vx \mapsto \phi(\vx, \vm/n)$, defined for $\vm/n$ in $\mK$ and $\vx$ in a neighborhood of $\vzero$ from $\reals^t$ by
\[
    \phi(\vx, \vm/n) =
    - \log \left(
    \frac{B(\zeta_1 e^{i x_1}, \ldots, \zeta_t e^{i x_t})}{B(\vzeta)}
    \right)
    + i \sum_{j=1}^t \frac{m_j}{n} x_j,
\]
has an Hessian matrix $\mH_{\phi}(\vx, \vm/n)$ that is nonsingular at $\vx = \vzero$,
for all $\vm/n$ in $\mK$.
\end{itemize}
Let $\psi(\vy, \vm/n) = (\psi_1(\vy, \vm/n), \ldots, \psi_t(\vy, \vm/n))$ denote
the biholomorphic change of variable $\vx = \psi(\vy, \vm/n)$ from Lemma~\ref{th:morse}
that maps $\mC$ to a neighborhood of $\vzero$,
and such that
\[
    \phi(\psi(\vy,\vm/n), \vm/n) = \frac{1}{2} \sum_{j=1}^t y_j^2.
\]
Its Jacobian matrix is denoted by $J_{\psi}(\vy,\vm/n)$.
Let $A(\vz)$ denote a function analytic on $\mC$
that does not vanish at $\vzeta$ for any $\vm/n$ in $\mK$.
Then the following asymptotic expansion holds uniformly for $\vm/n$ in $\mK$
\[
    [\vz^{\vm}] A(\vz) B(\vz)^n =
    \frac{B(\vzeta)^n}{\vzeta^{\vm} (2 \pi n)^{t/2}}
    \Bigg(
    \sum_{r=0}^{d-1}
    c_r(\vm/n)
    n^{-r}
    + \bigO(n^{-d})
    \Bigg),
\]
where each $c_r$ is an analytic function, equal to
\[
    c_r(\vm/n) =
    \sum_{\substack{r_1 + \cdots + r_t = r\\ \forall j,\ r_j \geq 0}}
    (2 \vr - 1)!!
    [\vy^{2\vr}]
    A(\zeta_1 e^{i \psi_1(\vy)}, \ldots, \zeta_t e^{i \psi_t(\vy)}) \det( J_{\psi}(\vy, \vm/n) ),
\]
and, in particular, $c_0(\vm/n) = A(\vzeta) |\det(\mH_{\phi}(\vzero, \vm/n))|^{-1/2}$.
\end{theorem}

\begin{proof}
The coefficient extraction is expressed as a Cauchy integral
\[
    [\vz^{\vm}] A(\vz) B(\vz)^n =
    \frac{1}{(2 i \pi)^t}
    \oint
    A(\vz) \frac{B(\vz)^n}{\vz^{\vm}}
    \frac{d \vz}{z_1 \cdots z_t}.
\]
We choose for the contour of integration
the torus of radii $\vzeta$,
so $z_j = \zeta_j e^{i x_j}$ for all $1 \leq j \leq t$
\[
    [\vz^{\vm}] A(\vz) B(\vz)^n =
    \frac{1}{(2 \pi)^t}
    \int_{\vx \in ]-\pi,\pi]^t}
    A(\zeta_1 e^{i x_1}, \ldots, \zeta_t e^{i x_t})
    \frac{B(\zeta_1 e^{i x_1}, \ldots, \zeta_t e^{i x_t})^n}{\vzeta^{\vm} e^{i \sum_{j=1}^t m_j x_j}}
    d \vx.
\]
We would like to inject in the previous expression the function $\phi$,
but it might not be analytic everywhere on $]-\pi,\pi]^t$.
Since the support of $B$ spans the identity matrix, $B$ is not the zero function.
Its coefficients are nonnegative and $\vzeta$ has positive coefficients,
so $B(\vzeta)$ is a positive real number.
According to Lemma~\ref{th:multivariate_daffodil},
on the torus of radii $\vzeta$,
$|B(\vz)|$ reaches its unique maximum at $\vzeta$.
Therefore, there is a compact neighborhood $\mN$ of the origin in $\reals^t$
such that for all $\vm/n$ in $\mK$,
\begin{enumerate}
\item
$\phi$ is analytic on $\mN$,
\item \label{it:unimodular}
the Hessian matrix of $\phi$ is nonsingular on $\mN \setminus \{\vzero\}$
and Lemma~\ref{th:morse} is applicable,
\item 
for any $\vx$ in $\mN$ and $\vy$ in $]-\pi,\pi]^t \setminus \mN$, we have
\[
  \left| B(\zeta_1 e^{i y_1}, \ldots, \zeta_t e^{i y_t}) \right|
  \leq
  \left| B(\zeta_1 e^{i x_1}, \ldots, \zeta_t e^{i x_t}) \right|,
\]
\end{enumerate}
Thus, there is a value $\delta$ in $]0,1[$ such that
for all $\vx$ in $[-\pi,\pi[ \setminus \mN$,
\[
    |B(\zeta_1 e^{i x_1}, \ldots, \zeta_t e^{i x_t})| \leq B(\vzeta) \delta.
\]
This implies that the integral outside $\mN$ is exponentially negligible compared to the final result, because
\[
    \frac{1}{(2 \pi)^t}
    \int_{\vx \in ]-\pi,\pi]^t \setminus \mN}
    \left|
    A(\zeta_1 e^{i x_1}, \ldots, \zeta_t e^{i x_t})
    \frac{B(\zeta_1 e^{i x_1}, \ldots, \zeta_t e^{i x_t})^n}{\vzeta^{\vm} e^{i \sum_{j=1}^t m_j x_j}}
    \right|
    d \vx
    =
    \bigO \left( \delta^n \frac{B(\vzeta)^n}{\vzeta^{\vm}} \right).
\]
Therefore, we reduce the domain of integration to $\mN$
\[
    [\vz^{\vm}] A(\vz) B(\vz)^n =
    \frac{1}{(2 \pi)^t}
    \int_{\vx \in \mN}
    A(\zeta_1 e^{i x_1}, \ldots, \zeta_t e^{i x_t})
    \frac{B(\zeta_1 e^{i x_1}, \ldots, \zeta_t e^{i x_t})^n}{\vzeta^{\vm} e^{i \sum_{j=1}^t m_j x_j}}
    d \vx
    + \bigO \left( \delta^n \frac{B(\vzeta)^n}{\vzeta^{\vm}} \right).
\]
Multiplying and dividing by $B(\vzeta)^n e^{i n \sum_{j=1}^t \frac{m_j}{n} x_j}$,
and injecting the notation $\phi$ from the theorem, we obtain
\[
    [\vz^{\vm}] A(\vz) B(\vz)^n =
    \frac{B(\vzeta)^n}{\vzeta^{\vm} (2 \pi)^t}
    \int_{\vx \in \mN}
    A(\zeta_1 e^{i x_1}, \ldots, \zeta_t e^{i x_t})
    e^{- n \phi(\vx, \vm/n)}
    d \vx
    + \bigO \left( \delta^n \frac{B(\vzeta)^n}{\vzeta^{\vm}} \right).
\]
An application of Lemma~\ref{th:laplace_fourier} concludes the proof.
\end{proof}

\begin{lemma} \label{th:appendix_mgpos_asymptotics}
Consider a positive integer $d$
and integers $n$ and $k$ going to infinity
such that $\alpha := k/n$ stays in a compact set $\mK$ of $\reals_{>0}$.
Let $\lambda$ and $\zeta$ denote the unique positive solutions of the equations
\[
    \frac{\lambda}{2} \frac{e^{\lambda}+1}{e^{\lambda}-1} = \alpha + 1,
    \quad
    T(\zeta) = \alpha+1-\frac{\lambda}{2},
\]
$A(z,x)$ a bivariate function 
analytic on the closed torus of radii $(\zeta, \lambda)$,
and
\[
    B(z,x) = \frac{1}{1 - T(z) \frac{e^x-1-x}{x^2/2}}.
\]
Then the following asymptotic expansion holds uniformly for $\alpha$ in $\mK$
\[
    [z^n x^{2k}]
    A(z,x)
    B(z,x)^k
    =
    \frac{B(\zeta, \lambda)^k}{2 \pi k \zeta^n \lambda^{2k}}
    \bigg(
    \sum_{r=0}^{d-1}
    c_r(\alpha) k^{-r}
    + \bigO(k^{-d})
    \bigg),
\]
with
\[
    B(\zeta, \lambda) = \frac{\lambda}{2 \alpha},
    \quad
    \zeta = e^{-\alpha-1} \sqrt{(\alpha+1)^2 - (\lambda/2)^2},
    \quad
    c_0(\alpha) =
    A(\zeta, \lambda)
    \sqrt{\frac{\alpha^3}{\lambda}
    \frac{\lambda/2-\alpha}{\lambda^2/4-\alpha^2-\alpha}}
\]
and the formula for the other $(c_r(\alpha))$ is as follows.
There is a biholomorphic function $(x,y) \mapsto \psi(x,y,\alpha) = (\psi_1(x,y,\alpha), \psi_2(x,y,\alpha))$
defined a neighborhood of $(0,0)$ and sending $(0,0)$ to $(0,0)$
such that, with the notation $\psi_j = \psi_j(x,y,\alpha)$ for $j$ equal to $1$ or $2$,
\[
    - \log \left( \frac{B(\zeta e^{i \psi_1}, \lambda e^{i \psi_2})}{B(\zeta,\lambda)} \right)
    + i \left( \frac{\psi_1}{\alpha} + 2 \psi_2 \right) =
    \frac{x^2+y^2}{2}.
\]
Its Jacobian matrix is denoted by $J_\psi(x,y,\alpha)$, and we have
\[
    c_r(\alpha) =
    \sum_{t=0}^r
    (2t-1)!! (2(r-t)-1)!! [x^{2t} y^{2(r-t)}]
    A(\zeta e^{i \psi_1(x,y)}, \lambda e^{i \psi_2(x,y)}) \det(J_{\psi}(x,y,\alpha)).
\]
Each $c_r$, $\lambda$ and $\zeta$ are smooth functions of $\alpha$.
\end{lemma}

\begin{proof}
To apply Theorem~\ref{th:large_powers}
with $t=2$, $\lambda_1 = 1/\alpha$, $\lambda_2 = 2$,
we check its hypothesis.
First, the system of equations
\[
    \frac{\zeta \partial_1 B(\zeta, \lambda)}{B(\zeta, \lambda)} = \frac{1}{\alpha},
    \quad
    \frac{\lambda \partial_2 B(\zeta, \lambda)}{B(\zeta, \lambda)} = 2
\]
needs to be solved. Using the classical relation from Lemma~\ref{th:cayley})
\[
    z T'(z) = \frac{T(z)}{1-T(z)},
\]
it is equivalent with
\[
    \frac{\lambda}{2} \frac{e^{\lambda}+1}{e^{\lambda}-1} = \alpha + 1,
    \quad
    T(\zeta) = \alpha + 1 - \frac{\lambda}{2}. 
\]
Since the function $\frac{x}{2} \frac{e^{x}+1}{e^{x}-1}$
is strictly increasing on $\reals_{>0}$ and has value $1$ at $0$,
there is a unique positive solution $\lambda$ to the first equation,
and the implicit function theorem states that $\lambda$
is a smooth function of $\alpha$.
This equation implies
\[
    \frac{e^{\lambda}-1}{\lambda} = \frac{1}{\alpha+1-\frac{\lambda}{2}}.
\]
The left hand-side is clearly greater than $1$,
so $\alpha < \lambda/2$.
The function $T(x)$ is strictly increasing on $[0,1/e[$, $T(0) = 0$
and $\lim_{x \to 1/e} T(x) = +\infty$,
so there is a unique positive solution $\zeta$ to the second equation,
which is, again, a smooth function of $\alpha$ by the implicit function theorem.
Using the relation $z = T(z) e^{-T(z)}$ from Lemma~\ref{th:cayley},
the previous system admits the following useful reformulation
\[
    e^{\lambda} = \frac{\alpha+1+\frac{\lambda}{2}}{\alpha+1-\frac{\lambda}{2}},
    \quad
    \zeta = e^{-\alpha-1} \sqrt{(\alpha+1)^2 - \left(\lambda/2 \right)^2}
\]
With the help of the computer algebra system \cite{sagemath},
we find that the Hessian matrix $\mH_{\phi}(x,y)$ of
\[
    \phi(x,y) =
    - \log \left( \frac{B(\zeta e^{i x}, \lambda e^{i y})}{B(\zeta, \lambda)} \right)
    + i \left( \frac{x}{\alpha} + 2 y \right)
\]
has coefficients at the origin equal to
\begin{align*}
    \mH_{1,1}(0,0) &= \frac{1}{(1-T(\zeta))^2} \frac{1}{\alpha} + \frac{1}{\alpha^2},
    \\
    \mH_{1,2}(0,0) &= \mH_{2,1}(0,0) = \frac{2}{1-T(\zeta)} + \frac{2}{\alpha},
    \\
    \mH_{2,2}(0,0) &= \lambda ( 1 - T(\zeta)) \frac{1}{\alpha} + 2 \lambda,
\end{align*}
so its determinant is equal to
\[
    \det(\mH(0,0)) =
    \frac{\lambda}{\alpha^3}
    \frac{\lambda^2/4-\alpha^2-\alpha}{\lambda/2-\alpha}.
\]
In the numerator $\lambda^2/4-\alpha^2-\alpha$,
we replace $\alpha$ with its expression $\frac{\lambda}{2} \frac{e^{\lambda}+1}{e^{\lambda}-1} -1$, and obtain
\[
    - \frac{\lambda^2}{(e^{\lambda}-1)^2}
    + \frac{\lambda(1-\lambda)}{e^{\lambda}-1}
    + \frac{\lambda}{2},
\]
which is a function of $\lambda$, positive for any $\lambda > 0$.
Since $\alpha < \lambda/2$,
the determinant $\det(\mH(0,0))$ is strictly positive
and the matrix is nonsingular.
Thus, according to Lemma~\ref{th:morse},
there is a biholomorphic function
$\psi(x,y) = (\psi_1(x,y), \psi_2(x,y))$
that maps a neighborhood of $(0,0)$
to a neighborhood of $(\zeta, \lambda)$, such that
\[
    \phi(\psi(x,y)) = \frac{x^2+y^2}{2}.
\]
Let $J_{\psi}(x,y)$ denote its Jacobian matrix.
The hypothesis of Theorem~\ref{th:large_powers} are satisfied, so
\[
    [z^n x^{2k}]
    A(z,x)
    B(z,x)^k
    =
    \frac{B(\zeta, \lambda)^k}{2 \pi k \zeta^n \lambda^{2k}}
    \bigg(
    \sum_{r=0}^{d-1}
    c_r k^{-r}
    + \bigO(k^{-d})
    \bigg),
\]
where each $c_r$ is equal to
\[
    c_r =
    \sum_{a=0}^r
    (2a-1)!! (2(r-a)-1)!!
    [x^{2a} y^{2(r-a)}]
    A(\zeta e^{i \psi_1(x,y)}, \lambda e^{i \psi_2(x,y)})
    \det(J_{\psi}(x,y)),
\]
and in particular, $c_0 = A(\zeta,\lambda) \det(\mH(0,0))^{-1/2}$.
Injecting the values of $\det(\mH(0,0))$,
of $T(\zeta) = \alpha + 1 - \lambda/2$,
of $e^{\lambda} = 1 + \lambda / (\alpha+1-\lambda/2)$
and of
\[
    B(\zeta, \lambda)
    =
    \left( 1 - T(\zeta) \frac{e^{\lambda}-1-\lambda}{\lambda^2/2} \right)^{-1}
    =
    \frac{\lambda}{2 \alpha}
\]
concludes the proof.
\end{proof}

In this paper, many coefficient extractions of the form
\[
    [z^n x^{2k}] F_k(T(z),x)
\]
appeared.
We chose not to apply the natural change of variable $t = T(z)$,
because the resulting generating functions
might not have nonnegative coefficients anymore
-- a property that simplifies the asymptotic analysis
and is required for the application of Theorem~\ref{th:large_powers}.

\section{Asymptotic expansion of the coefficients of divergent series}
\label{sec:divergent_series}

The number of connected multigraphs with $n$ vertices and excess $k$
is expressed using the generating function of positive multigraphs
\[
    \cmg_{n,k} = 2^{n+k} (n+k)! n! [z^n y^k] \log \bigg( 1 + \sum_{\ell \geq 1} \mgpos_{\ell}(z) y^{\ell} \bigg).
\]
A similar expression links connected graphs to positive graphs.
Given the rapid growth of the asymptotics of $\mgpos_{n,k}$
with respect to $k$, the work of \cite{Be75} comes to mind
to extract the asymptotic expansion of $\cmg_{n,k}$.
Informally, his Theorem~$1$ states that
when the sequence $(g_k)_{k \geq 1}$ grows fast enough to infinity
and $f(y)$ has a nonzero radius of convergence, then
\[
    [y^k] f\bigg(\sum_{\ell \geq 1} g_{\ell} y^{\ell}\bigg) =
    \sum_{r=0}^d
    g_{k-r}
    [y^r] f'\bigg(\sum_{\ell \geq 1} g_{\ell} y^{\ell}\bigg)
    + \bigO(g_{k-d-1}).
\]
In this expression, observe that there is a finite number of summands,
indexed by $r$, and that
$[y^r] f'(\sum_{\ell \geq 1} g_{\ell} y^{\ell})$
is a finite sum of product of terms from $(g_{\ell})_{\ell \geq 1}$.
Thus, the asymptotics of the $r$th summand when $k$ is large
is driven by $g_{k-r}$.
To apply this theorem, we would set
\[
    f(y) = \log(1+y)
    \quad
    \text{ and }
    \quad
    g_k = \mgpos_k(z).
\]
Because we are dealing with this extra variable $z$,
our problem does not fit as it is in the theorems of \cite{Be75}.
However, we can follow his proof to obtain a similar result,
stated in Lemma~\ref{th:bender_like}.
This lemma is applied in Lemmas~\ref{th:bender_applied}
and~\ref{th:csg_bender_applied}.

This section is thus strongly inspired by \cite{Be75},
and our presentation follows a recent extension of his work
provided by \cite{Bo16}.
In particular, Lemmas~\ref{th:df_two},
\ref{th:df_two_stripping},
\ref{th:df_poly},
and~\ref{th:df_product}
are copied from~\cite{Bo16},
recalled here for completeness,
and the other lemmas are simple corollaries.

\begin{lemma}[\cite{Bo16}] \label{th:df_two}
For any positive value $\beta$,
there exists a positive value $C = (2 + \beta) \Gamma(\beta)$ such that
\[
    \sum_{\ell=0}^k \Gamma(\ell + \beta) \Gamma(k-\ell+\beta)
    \leq
    C \Gamma(k+\beta)
\]
for all nonnegative integers $k$.
\end{lemma}

\begin{proof}
The inequality holds for $k=0$, so we consider the case $k \geq 1$.
The function $\Gamma$ is log-convex in $\reals_{>0}$.
The log-convexity is preserved under shifts, reflections,
and multiplication by a log-convex function,
so the function
\[
    G_k(\ell) = \Gamma(\ell + \beta) \Gamma(k-\ell+\beta)
\]
is log-convex for $\ell$ in $[0,k]$.
Any log-convex function reaches its maximum on the boundary of its domain.
Furthermore, $G$ has the reflexion property $G_k(k-\ell) = G_k(\ell)$ for all $0 \leq \ell \leq k$.
Therefore, $G_k(\ell) \leq G_k(1)$ for all $1 \leq \ell \leq k-1$.
Stripping off the two boundary terms, we obtain
\[
    \sum_{\ell=0}^k G_k(\ell)
    \leq 2 G_k(0) + (k-1) G_k(1)
    \leq 2 G_k(0) + (k-1+\beta) G_k(1).
\]
It follows from $n\Gamma(n) = \Gamma(n+1)$ that $G_k(1) = G_k(0) \frac{\beta}{k-1+\beta}$, so
\[
    \sum_{\ell=0}^k G_k(\ell)
    \leq
    (2 + \beta) G_k(0)
    \leq
    (2 + \beta) \Gamma(\beta) \Gamma(k+\beta),
\]
which concludes the proof.
\end{proof}

\begin{lemma}[\cite{Bo16}] \label{th:df_two_stripping}
For any positive value $\beta$ and nonnegative integer $d$, we have
\[
    \sum_{\ell=d}^{k-d}
    \Gamma(\ell+\beta) \Gamma(k-\ell+\beta)
    = \bigO(\Gamma(k-d+\beta)).
\]
\end{lemma}

\begin{proof}
The left hand side is rewritten as
\[
    \sum_{\ell=0}^{k-2d}
    \Gamma(\ell+d+\beta) \Gamma(k-\ell+d+\beta),
\]
where Lemma~\ref{th:df_two} is applied
with the substitutions $\beta \to d+\beta$ and $k \to k - 2d$.
\end{proof}

\begin{lemma}[\cite{Bo16}] \label{th:df_poly}
For any positive value $\beta$, nonnegative integer $d$,
real value $C$, and polynomial $P$, we have
\[
    \sum_{\ell=d}^{k}
    C^{\ell} P(\ell)
    \Gamma(k-\ell+\beta)
    = \bigO(\Gamma(k-d+\beta)).
\]
\end{lemma}

\begin{proof}
There exists a value $D$ such that
$|C^{\ell} P(\ell)|$ is bounded by $D \Gamma(\ell+\beta)$
for all nonnegative $\ell$, so
\[
    \bigg|
    \sum_{\ell=d}^{k-d}
    C^{\ell} P(\ell)
    \Gamma(k-\ell+\beta)
    \bigg|
    \leq
    D
    \sum_{\ell=d}^{k-d}
    \Gamma(\ell+\beta)
    \Gamma(k-\ell+\beta)
\]
Lemma~\ref{th:df_two_stripping} then implies
\[
    \sum_{\ell=d}^{k-d}
    C^{\ell} P(\ell)
    \Gamma(k-\ell+\beta)
    = \bigO(\Gamma(k-d+\beta)).
\]
The remainder of the sum
\[
    \sum_{\ell=k-d+1}^{k}
    C^{\ell} P(\ell)
    \Gamma(k-\ell+\beta)
\]
is rewritten, after the change of variable $m = k-\ell$, as
\[
    \sum_{m=0}^{d-1}
    C^{k-m} P(k-m)
    \Gamma(m+\beta),
\]
which is also a $\bigO(\Gamma(k-d+\beta))$.
\end{proof}

\begin{lemma}[\cite{Bo16}] \label{th:df_product}
For any positive value $\beta$,
there exists a positive value $C$ such that
for all integers $k \geq 0$ and $q \geq 1$, we have
\[
    \sum_{\substack{k_1 + \cdots + k_q = k\\ \forall j,\ k_j \geq 0}}
    \prod_{j=1}^q
    \Gamma(k_j + \beta)
    \leq
    C^q \Gamma(k+\beta).
\]
\end{lemma}

\begin{proof}
The value $C$ is the same as Lemma~\ref{th:df_two}.
We prove the lemma by induction on $q$.
For $q=1$, since $C \geq 1$, the inequality holds.
For any $q \geq 1$, let us separate the sum on $k_{q+1}$
\[
    \sum_{\substack{k_1 + \cdots + k_{q+1} = k\\ \forall j,\ k_j \geq 0}}
    \prod_{j=1}^{q+1}
    \Gamma(k_j + \beta)
    =
    \sum_{k_{q+1} = 0}^k
    \Gamma(k_{q+1}+\beta)
    \sum_{\substack{k_1 + \cdots + k_q = k-k_{q+1}\\ \forall j,\ k_j \geq 0}}
    \prod_{j=1}^q
    \Gamma(k_j + \beta).
\]
If the lemma holds at $q$, we obtain
\[
    \sum_{\substack{k_1 + \cdots + k_{q+1} = k\\ \forall j,\ k_j \geq 0}}
    \prod_{j=1}^{q+1}
    \Gamma(k_j + \beta)
    \leq
    \sum_{k_{q+1} = 0}^k
    \Gamma(k_{q+1}+\beta)
    C^q \Gamma(k-k_{q+1}+\beta).
\]
An application of Lemma~\ref{th:df_two} concludes the proof.
\end{proof}

\begin{lemma} \label{th:df_stripping_product}
For any positive value $\beta$,
there exists a positive value $C$ such that
for all integers $k \geq 0$ and $q \geq 1$, we have
\[
    \sum_{\substack{k_1 + \cdots + k_q = k\\ \forall j,\ k_j \geq 1}}
    \prod_{j=1}^q
    \Gamma(k_j + \beta)
    \leq
    C^q \Gamma(k-q+1+\beta).
\]
\end{lemma}

\begin{proof}
After the change of variable $k_j \to k_j-1$,
the left hand side becomes
\[
    \sum_{\substack{k_1 + \cdots + k_q = k\\ \forall j,\ k_j \geq 1}}
    \prod_{j=1}^q
    \Gamma(k_j + \beta)
    =
    \sum_{\substack{k_1 + \cdots + k_q = k-q\\ \forall j,\ k_j \geq 0}}
    \prod_{j=1}^q
    \Gamma(k_j + 1 + \beta).
\]
To conclude the proof,
Lemma~\ref{th:df_product} is applied
with the substitution $\beta \to 1 + \beta$ and $k \to k-q$.
\end{proof}

\begin{lemma} \label{th:df_exp_small}
Given a positive value $\beta$, a value $C$,
and a positive integer $d$,
there exists a value $D$ such that
for any positive integer $k$, we have
\[
    \sum_{q \geq 1}
    C^q
    \sum_{\substack{k_1 + \cdots + k_q = k\\ \forall j,\ 1 \leq k_j < d}}
    \prod_{j=1}^q \Gamma(k_j+\beta)
    \leq
    D^k.
\]
\end{lemma}

\begin{proof}
The inequality is satisfied for any $q > k$,
because the left hand side vanishes in that case.
Thus, we consider $1 \leq q \leq k$.
Stirling's bound implies
\[
    \Gamma(k+\beta) = \exactbigO(k^{k+\beta+1/2} e^{-k}),
\]
so there is a constant $E$ such that
for any $k_1 + \cdots + k_q = k$,
with $k_j < d$ for all $j$, we have
\[
    \prod_{j=1}^q \Gamma(k_j+\beta)
    \leq
    E \prod_{j=1}^q k_j^{k_j+\beta+1/2} e^{-k_j}
    \leq
    E d^{k+q(\beta+1/2)} e^{-k}
    \leq
    E d^{k(\beta+3/2)} e^{-k},
\]
because $q \leq k$.
The number of compositions of the integer $k$ is $2^{k-1}$, so
\[
    \sum_{q \geq 1}
    C^q
    \sum_{\substack{k_1 + \cdots + k_q = k\\ \forall j,\ 1 \leq k_j < d}}
    \prod_{j=1}^q \Gamma(k_j+\beta)
    \leq
    2^{k-1}
    |C|^k
    E d^{k(\beta+3/2)} e^{-k},
\]
which grows only exponentially fast.
\end{proof}

\begin{lemma} \label{th:S_k_d}
For any positive value $\beta $, value $C$,
and nonnegative integer $d$, we have
\[
    \sum_{q=1}^k
    C^q
    \sum_{\substack{k_1+\cdots+k_q = k\\ \forall j,\ 1 \leq k_j \leq k-d}}
    \prod_{j=1}^q
    \Gamma(k_j + \beta)
    =
    \bigO(\Gamma(k - d + \beta)).
\]
\end{lemma}

\begin{proof}
The sum is cut into two,
depending on whether $\max_j(k_j)$ is smaller than or not smaller than $d+1$
\[
    \sum_{q=1}^k
    C^q
    \sum_{\substack{k_1+\cdots+k_q = k\\ \forall j,\ 1 \leq k_j < d+1}}
    \prod_{j=1}^q
    \Gamma(k_j+\beta)
    +
    \sum_{q=1}^k
    C^q
    \sum_{\substack{k_1+\cdots+k_q = k\\ \forall j,\ d+1 \leq \max_j(k_j) \leq k-d}}
    \prod_{j=1}^q
    \Gamma(k_j+\beta).
\]
According to Lemma~\ref{th:df_exp_small},
there is a constant $D$ such that the first term
is bounded in absolute value by $D^k$,
which is negligible compared to $\Gamma(k - d + \beta)$.
Thus, we focus on the second term.
Up to a symmetry of order at most $q$,
the maximum is reached by $\ell = k_q$,
so the second term is bounded in absolute value by
\[
    \sum_{q=1}^k
    q |C|^q
    \sum_{\ell = d+1}^{k-d}
    \Gamma(\ell + \beta)
    \sum_{\substack{k_1 + \cdots + k_{q-1} = k-\ell\\ \forall j,\ k_j \geq 1}}
    \prod_{j=1}^{q-1}
    \Gamma(k_j + \beta).
\]
The sums over $q$ and $\ell$ are switched,
and the conditions $k_1 + \cdots + k_{q-1} = k - \ell$
and $k_j \geq 1$ for all $1 \leq j \leq q-1$
imply $q \leq k- \ell + 1$
\[
    \sum_{\ell = d+1}^{k-d}
    \Gamma(\ell + \beta)
    \sum_{q=1}^{k-\ell+1}
    q |C|^q
    \sum_{\substack{k_1 + \cdots + k_{q-1} = k-\ell\\ \forall j,\ k_j \geq 1}}
    \prod_{j=1}^{q-1}
    \Gamma(k_j + \beta).
\]
The summand corresponding to $\ell=k-d$
is a $\bigO(\Gamma(k-d+\beta))$,
so the term can be rewritten
\[
    \sum_{\ell = d+1}^{k-d-1}
    \Gamma(\ell + \beta)
    \sum_{q=1}^{k-\ell+1}
    q |C|^q
    \sum_{\substack{k_1 + \cdots + k_{q-1} = k-\ell\\ \forall j,\ k_j \geq 1}}
    \prod_{j=1}^{q-1}
    \Gamma(k_j + \beta).
    +
    \bigO(\Gamma(k-d+\beta)).
\]
According to Lemma~\ref{th:df_stripping_product},
applied with $k \mapsto k-\ell$ and $q \mapsto q-1$,
there is a constant $D$ such that
\[
    \sum_{\substack{k_1 + \cdots + k_{q-1} = k-\ell\\ \forall j,\ k_j \geq 1}}
    \prod_{j=1}^{q-1}
    \Gamma(k_j + \beta)
    \leq
    D^{q-1}
    \Gamma(k - \ell- q + 2 + \beta),
\]
so the previous expression is bounded by
\[
    \sum_{\ell = d+1}^{k-d-1}
    \Gamma(\ell + \beta)
    \sum_{q=1}^{k-\ell+1}
    q |C|^q D^{q-1}
    \Gamma(k - \ell - q + 2 + \beta)
    +
    \bigO(\Gamma(k-d+\beta)).
\]
There is a constant $E$ such that $E^q > q |C|^q D^{q-1}$
for all $q \geq 1$.
According to Lemma~\ref{th:df_poly},
applied with $k \mapsto k - \ell + 1$ and $\beta \mapsto \beta + 1$,
there is a constant $F$ such that
the second sum is bounded by $F \Gamma(k-\ell+1+\beta)$,
so the previous expression is bounded by
\[
    F
    \sum_{\ell = d+1}^{k-d-1}    
    \Gamma(\ell + \beta)
    \Gamma(k-\ell+1+\beta).
\]
This is a $\bigO(\Gamma(k-d+\beta))$
according to Lemma~\ref{th:df_two_stripping}.
\end{proof}

The following lemma is the main result of this section.
It considers a bivariate formal series
\[
    g(z,y) = \sum_{\ell \geq 1} g_{\ell}(z) y^{\ell},
\]
where $g_{\ell}(z)$ is analytic at the origin,
but the sequence $g_{\ell}(\zeta)$
(for some positive value $\zeta$)
grows factorially,
so that $g(\zeta, y)$ has a zero radius of convergence.

\begin{lemma}
Consider a formal bivariate series
\[
    g(z,y) = \sum_{\ell \geq 1} g_{\ell}(z) y^{\ell}
\]
with nonnegative coefficients,
and assume there exist positive constants $C$, $E$, $\beta$ and $\zeta$
such that for each $\ell$,
\[
    g_{\ell}(\zeta) \leq C E^{\ell} \Gamma(\ell + \beta),
\]
assuming that the radius of convergence of each $g_{\ell}(z)$ is greater than $\zeta$.
Let $f(z)$ be a function analytic at the origin,
then for any positive integer $d$, we have
\[
    [z^n y^k] f(g(z,y)) =
    [z^n]
    \sum_{r=0}^{d-1}
    g_{k-r}(z)
    [y^r]
    f'(g(z,y))
    +
    \bigO
    \left(
    \frac{E^k}{\zeta^n}
    \Gamma(k-d+\beta)
    \right).    
\]
\end{lemma}

\begin{proof}
Replacing $f(z)$ by its series expansion $\sum f_q z^q$, we obtain
\[
    f \bigg( \sum_{\ell \geq 1} g_{\ell}(z) y^{\ell} \bigg)
    =
    \sum_{q \geq 1}
    f_q
    \bigg( \sum_{\ell \geq 1} g_{\ell}(z) y^{\ell} \bigg)^q.
\]
The coefficient extraction $[z^n y^k]$ then gives
\[
    [z^n y^k]
    f \bigg( \sum_{\ell \geq 1} g_{\ell}(z) y^{\ell} \bigg)
    =
    [z^n]
    \sum_{q \geq 1}
    f_q
    \sum_{k_1 + \cdots + k_q = k}
    \prod_{j=1}^q
    g_{k_j}(z),
\]
with the convention $g_0(z) = 0$.
Let us cut this sum in two,
depending on whether $\max_j(k_j)$ is or is not greater than $k-d$
\begin{equation} \label{eq:large_or_small_k}
    [z^n]
    \sum_{q \geq 1}
    f_q
    \sum_{\substack{k_1 + \cdots + k_q = k\\ \forall j,\ \max_j(k_j) > k-d}}
    \prod_{j=1}^q
    g_{k_j}(z)
    +
    [z^n]
    \sum_{q \geq 1}
    f_q
    \sum_{\substack{k_1 + \cdots + k_q = k\\ \forall j,\ 1 \leq k_j \leq k-d}}
    \prod_{j=1}^q
    g_{k_j}(z).
\end{equation}
When $k$ is large compared to $d$ and $\max_j(k_j) > k-d$,
the maximum is reached by a unique $k_j$.
Up to a symmetry of order $q$, we can assume $k_q = \max_j(k_j)$.
Introduction the notation $r = k-k_q$,
the first sum becomes
\[
    [z^n]
    \sum_{q \geq 1}
    q
    f_q
    \sum_{r=0}^{d-1}
    g_{k-r}(z)
    \sum_{k_1 + \cdots + k_{q-1} = r}
    \prod_{j=1}^{q-1}
    g_{k_j}(z).
\]
Switching the sums over $q$ and $r$,
and interpreting the sum over $q$
as a coefficient extraction, we obtain
\[
    [z^n]
    \sum_{r=0}^{d-1}
    g_{k-r}(z)
    [y^r]
    f'(g(z,y)).
\]
It is left to bound the second term of Expression~\eqref{eq:large_or_small_k}.
For any series $h(z)$ with nonnegative coefficients
and radius of convergence at least $\zeta$,
and any integer $n$, we have
\[
    h(\zeta)
    =
    \sum_{j \geq 0} ([z^j] h(z)) \zeta^j
    \geq
    ([z^n] h(z)) \zeta^n,
\]
which implies
\[
    [z^n] h(z) \leq \frac{h(\zeta)}{\zeta^n}.
\]
By the hypothesis of the lemma on $g_{\ell}(\zeta)$,
this implies
\[
    [z^n] \prod_{j=1}^q g_{k_j}(z)
    \leq
    \frac{1}{\zeta^n}
    \prod_{j=1}^q
    g_{k_j}(\zeta)
    \leq
    \frac{1}{\zeta^n}
    \prod_{j=1}^q
    C E^{k_j} \Gamma(k_j + \beta).    
\]
Thus, the second term of Expression~\eqref{eq:large_or_small_k} is bounded by
\[
    \frac{E^k}{\zeta^n}
    \sum_{q \geq 1}
    f_q
    C^q
    \sum_{\substack{k_1 + \cdots + k_q = k\\ \forall j,\ 1 \leq k_j \leq k-d}}
    \prod_{j=1}^q
    \Gamma(k_j + \beta).
\]
Since $f(z)$ has a nonzero radius of convergence,
there exists a positive constant $D$ such that
$f_q \leq D^q$ for all $q$.
An application of Lemma~\ref{th:S_k_d}
then bounds the term as a
\[
    \bigO
    \left(
    \frac{E^k}{\zeta^n}
    \Gamma(k-d+\beta)
    \right),
\]
which achieves the proof.
\end{proof}

\section{Computation of the asymptotic expansion}
\label{sec:computation}

In this section, we provide step by step instructions
to compute the coefficients of the asymptotic expansion
of connected graphs.
To illustrate their practicability,
the first two coefficients are computed at the end of this section,
using the computer algebra system \cite{sagemath}.
Let us first recall the long chain of dependence for their computation.
According to Theorem~\ref{th:csg_asymptotics}, we have
\[
    \csg_{n,k} =
    n! (2k-1)!!
    \frac{B(\zeta, \lambda)^k}{2 \pi k \zeta^n \lambda^{2k}}
    \bigg(
    \sum_{r=0}^{d-1}
    c_r k^{-r}
    + \bigO(k^{-d})
    \bigg),
\]
where $\lambda$ is the unique positive solution of the equation
\[
    \frac{\lambda}{2} \frac{e^{\lambda}+1}{e^{\lambda}-1} = \alpha + 1,
\]
the values of $B(\zeta, \lambda)$ and $\zeta$ are equal to
\[
    B(\zeta, \lambda) = \frac{\lambda}{2 \alpha},
    \quad
    \zeta = e^{-\alpha-1} \sqrt{(\alpha+1)^2 - (\lambda/2)^2}.
\]
The coefficients $(c_s)$ are given by the following expression
(see the proof of the theorem)
\begin{equation} \label{eq:cs}
    c_s =
    \sum_{t=0}^s
    \sum_{r=0}^{s-t}
    \sum_{\ell=0}^{s-t-r}
    a_{r+\ell,s-t}
    b_{r,\ell,t}.
\end{equation}
where the coefficients $(a_{r,t})$ are known
\[
    a_{r,t}
    =
    2^{-t}
    \sum_{t_0 + \cdots + t_{r-1} = t-r}
    \prod_{s=0}^{r-1}
    (2s+1)^{t_s},
\]
and the coefficients $(b_{r,\ell,s})$ are equal to
\begin{equation} \label{eq:brells}
    b_{r,\ell,s}
    =
    \sum_{a=0}^s
    (2a-1)!! (2(s-a)-1)!!
    [x^{2a} y^{2(s-a)}]
    A_{r,\ell}(\zeta e^{i \psi_1(x,y)}, \lambda e^{i \psi_2(x,y)})
    \det(J_{\psi}(x,y)).
\end{equation}
The function $\psi(x,y) = (\psi_1(x,y), \psi_2(x,y))$
is defined in Lemma~\ref{th:mgpos_asymptotics}
using the following function $\phi(x,y)$
\begin{align*}
    \phi(x,y)
    &=
    - \log \left( \frac{B(\zeta e^{i x}, \lambda e^{i y})}{B(\zeta, \lambda)} \right)
    + i \left( \frac{x}{\alpha} + 2 y \right),
    \\
    \phi(\psi(x,y))
    &=
    \frac{x^2+y^2}{2},
\end{align*}
and $J_{\psi}(x,y)$ denotes the Jacobian matrix of $\psi(x,y)$.
The function $A_{r,\ell}(z,x)$ from the proof of Theorem~\ref{th:csg_asymptotics}
is equal to
\[
    A_{r,\ell}(z,x) = x^{2r} \frac{S_r(T(z))}{(1-T(z))^{3r}} C_{\ell}(z,x),
\]
where the polynomial $S_r(T)$ from Lemma~\ref{th:csg_bender_applied},
the polynomial $\sk_k(T)$ and the function $C_{\ell}(z,x)$
from Proposition~\ref{th:sgpos}
are expressed as
\begin{align}
    S_r(T) &= [y^r] \bigg(1 + \sum_{\ell=1}^r \sk_{\ell}(T) y^{\ell} \bigg)^{-1},
    \label{eq:SrT}
    \\
    \sk_k(T) &=
    (1-T)^{3k}
    \sum_{\ell = 0}^k
    (2(k-\ell)-1)!!
    [x^{2(k-\ell)}]
    \frac{\pospatch_{\ell}(T e^x,-1) e^{- T \frac{e^x-1}{2} - T^2 \frac{e^{2x}-1}{4}}}
    {(1-T)^{k-\ell} \big(1- \frac{T}{1-T} \frac{e^x-1-x-x^2/2}{x^2/2}\big)^{k-\ell+1/2}},
    \label{eq:skkT}
    \\
    C_{\ell}(z,x) &=
    x^{2\ell} \pospatch_{\ell}(T(z) e^x,-1) \sqrt{1-T(z)}
    e^{- T(z)\frac{e^x-1}{2}  - T(z)^2 \frac{e^{2x} - 1}{4}}
    B(z,x)^{-\ell+1/2}.
    \label{eq:cellzx}
\end{align}
Finally, a formula for the polynomial $\pospatch_{\ell}(z,u)$
is provided in the next section.
In Section~\ref{sec:explicite_patchworks},
we compute the polynomials $\pospatch_k(z,u)$, $\sk_k(T)$ and $S_r(T)$.
Section~\ref{sec:Arl} provides the Taylor expansion
of $A_{r,\ell}(\zeta e^{i x}, \lambda e^{i y})$.
In Section~\ref{sec:phi}, we explain how to compute
the Taylor expansion of $\phi(x,y)$, $\psi(x,y)$ and $J_{\psi}(x,y)$.
Combining those results in Equation~\eqref{eq:brells},
the coefficients $b_{r,\ell,s}$ are computed,
and finally the coefficients $c_s$, using Equation~\eqref{eq:cs}.

To facilitate the computation of the Taylor expansions,
we apply in this section the following simple lemma.

\begin{lemma} \label{th:composed_taylor}
Consider a function $f(t,u)$ analytic at $(T(\zeta), \lambda)$.
Then the following Taylor expansion holds
\[
    f(T(\zeta e^{i x}), \lambda e^{i y}) =
    \sum_{a,b \geq 0}
    \left( \frac{T(\zeta)}{1-T(\zeta)} \partial_{T(\zeta)} \right)^a
    \left(\lambda \partial_{\lambda} \right)^b
    f(T(\zeta), \lambda)
    \frac{(i x)^a}{a!}
    \frac{(i y)^b}{b!}.
\]
\end{lemma}

\begin{proof}
The classical formula for the Taylor expansion is
\[
    f(T(\zeta e^{i x}), \lambda e^{i y})
    =
    \sum_{a,b \geq 0}
    \left(
    \partial_{x}^a
    \partial_{y}^b
    f(T(\zeta e^{i x}), \lambda e^{i y})
    \right)
    \frac{x^a}{a!}
    \frac{y^b}{b!}.
\]
After the changes of variable $\zeta e^{i x} \mapsto x$, $\lambda e^{i y} \mapsto y$,
the partial derivative becomes
\[
    \partial_{x}^a
    \partial_{y}^b
    f(T(\zeta e^{i x}), \lambda e^{i y})
    =
    i^{a+b}
    (\zeta \partial_{\zeta})^a
    (\lambda \partial_{\lambda})^b
    f(T(\zeta), \lambda).
\]
We apply the relation
$
    z T'(z) = \frac{T(z)}{1-T(z)}
$
from Lemma~\ref{th:cayley}
and a change of variable to obtain
\[
    \partial_{x}^a
    \partial_{y}^b
    f(T(\zeta e^{i x}), \lambda e^{i y})
    =
    i^{a+b}
    \left(\frac{T(\zeta)}{1-T(\zeta)}  \partial_{T(\zeta)}\right)^a
    (\lambda \partial_{\lambda})^b
    f(T(\zeta), \lambda).
\]
\end{proof}

    \subsection{Computation of the polynomials $\pospatch_k(z,u)$, $\sk_k(T)$ and $S_r(T)$} \label{sec:explicite_patchworks}

    \paragraph{Patchworks $\pospatch_k(z,u)$.}

In Lemma~\ref{th:patchworks}, we proved the existence of polynomials $\pospatch_k(z,u)$
(generating functions of the patchworks of excess $k$
that contain neither isolated loops nor isolated double edges)
such that
\[
    \patch_k(z,u) = e^{u z/2 + u z^2/4} \pospatch_k(z,u).
\]
The expressions of those polynomials were far from explicit,
and relied on the enumeration of all multigraphs with excess $k$ and minimum degree at least $3$.
In this section, we derive a formula that allows their computation,
using a computer algebra system, and provide the first few polynomials.

\begin{lemma} \label{th:explicite_patchworks}
Let the generating function of graphs be denoted by
\[
    \sg(z,w) = \sum_{n \geq 0} (1+w)^{\binom{n}{2}} \frac{z^n}{n!}.
\]
The generating function of patchworks is then equal to
\[
    \patch(z,w,u) = \sg \big(z e^{u w/2}, \sg(w,u) e^{-w}-1 \big) e^{-z}.
\]
\end{lemma}

\begin{proof}
Let $\mD$ denote the family of patchworks on two vertices $\{1,2\}$
that contains only double edges (\textit{i.e.}\ no loop).
Two parts of such a patchwork $P$ may share at most one edge
(since all parts are distinct).
We now describe a bijection between $\mD$
and the non-empty graphs without isolated vertices.
Let $P$ denote a patchwork from $\mD$,
and $G$ the corresponding graph:
\begin{itemize}
\item
each edge of $\mg(P)$ is represented by a vertex of $G$,
\item
each part of $P$ is a double edge, 
and corresponds to an edge of $G$.
\end{itemize}
There are no loops in $G$ because
each double edge of $P$ contains two distinct edges.
There are no multiple edges in $G$
because the parts of $P$ are distinct.
No vertex of $G$ can be isolated
since all edges of $P$ belong to at least one part.
The generating function of non-empty graphs without isolated vertices is
\[
    \sg(z,w) e^{-z} - 1,
\]
so the generating function of $\mD$ is (without taking into account the two vertices)
\[
    \sg(w,u) e^{-w} - 1.
\]
Any patchwork where a set of isolated vertices has been added 
can be uniquely described as a graph $G$, 
where each edge is replaced with a patchwork from $\mD$,
and a set of loops is added to each vertex.
Therefore, the generating function of patchworks satisfies
\[
  \patch(z,w,u) e^z = \sg \big(z e^{u w/2}, \sg(w,u) e^{-w} - 1\big).
\]
\end{proof}

According to Lemma~\ref{th:patchworks},
generating function of the patchworks of excess $k$
that contain neither isolated loops nor isolated double edges
is a multinomial, equal to
\[
    \pospatch_k(z,u) = [y^k] \patch(z/y,y,u) e^{-u z/2 - u z^2/4}.
\]
In all the other expression involving those multinomial,
the variable $u$ is set to $-1$.
In the following lemma, we derive an exact expression
for the polynomials $(\pospatch_k(z,-1))$.

\begin{lemma} \label{th:pospatch}
For any positive $k$,
the generating function $\pospatch_k(z,-1)$
is a polynomial of degree $3k$
\[
    \pospatch_k(z,-1) =
    \sum_{n=0}^{3k}
    p_{n,k} z^n,
\]
where each coefficient $p_{n,k}$ has the following expression
\[
    p_{n,k} =
    \sum_{\ell,r,s,t}
    \binom{\binom{\ell}{2}}{k+\ell+t-r}
    \binom{s}{n-\ell-s-t}
    \frac{(-1)^{r+t} \ell^{2r}}{\ell! r! s! t! 2^{n+r-\ell-t}},
\]
where the domain of summation of $\ell$, $s$ and $t$ is $\{0,1,\ldots, n\}$,
and the domain of summation of $r$ is $\{0,1,\ldots, n+k\}$.
\end{lemma}

\begin{proof}
The expression of the generating function of patchworks
from Lemma~\ref{th:explicite_patchworks},
evaluated at $u=-1$, is equal to
\[
    \patch(z,w,-1) = \sum_{\ell \geq 0} (1+w)^{\binom{\ell}{2}} e^{- \frac{\ell^2}{2} w} \frac{z^{\ell}}{\ell!} e^{-z}.
\]
The expression from Lemma~\ref{th:patchworks}
\[
    \pospatch_k(z,u) = [y^k] \patch(z/y,y,u) e^{-u z/2 - u z^2/4}
\]
is evaluated at $u=-1$,
and the coefficient $[z^n]$ is extracted
\[
    p_{n,k} = [z^n] \pospatch_k(z,-1) =
    [z^n w^{n+k}] \patch(z,w,-1) e^{z w/2 + z^2 w^2/4}.
\]
Injecting the expression of $\patch(z,w,-1)$, we obtain
\[
    p_{n,k} =
    [z^n w^{n+k}]
    \sum_{\ell \geq 0} (1+w)^{\binom{\ell}{2}} e^{- \frac{\ell^2}{2} w} \frac{z^{\ell}}{\ell!} e^{-z}
    e^{z w/2 + z^2 w^2/4}.
\]
The exponentials are expanded as sums and the coefficients $[z^n w^{n+k}]$ are extracted,
which leads to the result of the lemma
\[
    p_{n,k} =
    \sum_{\ell,r,s,t}
    \binom{\binom{\ell}{2}}{k+\ell+t-r}
    \binom{s}{n-\ell-s-t}
    \frac{(-1)^{r+t} \ell^{2r}}{\ell! r! s! t! 2^{n+r-\ell-t}}.
\]
It is not obvious from this expression
that $p_{n,k}$ vanishes for all $n \geq 3k$.
To prove that the generating function of positive patchworks of excess $k$
is a polynomial of degree $3k$,
we show that those patchworks contain at most $3k$ vertices.
To do so, we try to build a positive patchwork $P$ of excess $k$
with a maximal number of vertices.
If $P$ contains a loop, a vertex can be added
to transform this loop into a double edge,
without changing the excess, which is a contradiction.
If $P$ contains a triple edge linking the vertices $v$ and $w$,
we can remove one of those edges,
and add a vertex of degree $2$, linked to $v$ by a double edge,
without changing the excess.
Again, this is a contradiction with the definition of $P$.
Thus, $P$ contains only double edges,
so $\mg(P)$ can be obtained from a simple graph $G$,
by replacing each edge with a double edge.
If $G$ contained a cycle,
we could remove an edge $(\{v,w\}$ from this cycle,
and add a new vertex linked to $v$.
The corresponding patchwork would have the same excess as $P$,
but one more vertex, which is absurd.
Hence, all components of $G$ are trees.
%
A tree with $t$ vertices contains $t-1$ edges.
Let $t_1$, \ldots, $t_{c}$ denote the number of vertices
of the $c$ connected components of $\mg(P)$,
so the number of vertices of $P$ is $n = t_1 + \cdots + t_{c}$.
Then each component has excess $t_j-2$,
so the excess of $P$ is
\[
    k = t_1 + \cdots + t_{\ell} - 2c = n - 2c.
\]
Hence, the number of vertices of $P$ is maximized
when the number of trees is maximized.
Since each tree has excess at least $1$
(because $P$ is a positive patchwork),
we obtain $c \leq k$, which implies $n \leq 3k$.
This bound is reached by the patchwork
where each of the $k$ connected components
has three vertices,
and two double edges sharing a vertex.
\end{proof}

Using \cite{sagemath}, we computed the first few polynomials $\pospatch_k(z,-1)$
\begin{align*}
    \pospatch_0(z,-1) &=
    1,
    \\
    \pospatch_1(z,-1) &=
    1/8 z^3 + 5/12 z^2 + 1/8 z,
    \\
    \pospatch_2(z,-1) &=
    1/128 z^6 + 5/96 z^5 + 11/576 z^4 - 29/96 z^3 - 133/384 z^2 - 1/48 z,
    \\
    \pospatch_3(z,-1) &=
    1/3072 z^9 + 5/1536 z^8 + 13/9216 z^7 - 1253/20736 z^6 - 385/3072 z^5 + 655/4608 z^4\\
    &\ \ + 1337/3072 z^3 + 129/640 z^2 + 1/384 z.
\end{align*}


Applying Equation~\eqref{eq:skkT} and~\eqref{eq:SrT},
we compute the first few polynomials $\sk_k(T)$ and $S_r(T)$
\begin{align*}
    \sk_0(T) &= 1,
    \\
    \sk_1(T) &= -\frac{1}{24} \, {\left(T - 6\right)} T^{4},
    \\
    \sk_2(T) &= \frac{1}{1152} \, {\left(T^{6} - 12 \, T^{5} + 12 \, T^{4} + 216 \, T^{3} - 552 \, T^{2} + 672 \, T + 48\right)} T^{4},
    \\
    \sk_3(T) &= -\frac{1}{414720} \, \Big(5 \, T^{10} - 90 \, T^{9} + 180 \, T^{8} + 4320 \, T^{7} - 22248 \, T^{6} - 6048 \, T^{5} + 234072 \, T^{4} - 714528 \, T^{3}
    \\ &\ + 905472 \, T^{2} - 671040 \, T - 155520\Big) T^{5},
    \\
    S_0(T) &= 1,
    \\
    S_1(T) &= \frac{1}{24} \, T^{5} - \frac{1}{4} \, T^{4} - 1,
    \\
    S_2(T) &= \frac{1}{1327104} \, {\left(T^{10} - 12 \, T^{9} + 12 \, T^{8} + 216 \, T^{7} - 552 \, T^{6} + 624 \, T^{5} + 336 \, T^{4} + 1152\right)}^{2},
    \\
    S_3(T) &= \frac{1}{71328803586048000} \, \Big(5 \, T^{15} - 90 \, T^{14} + 180 \, T^{13} + 4320 \, T^{12} - 22248 \, T^{11} - 6408 \, T^{10} + 238392 \, T^{9}
    \\&\ - 718848 \, T^{8} + 827712 \, T^{7} - 472320 \, T^{6} - 380160 \, T^{5} - 120960 \, T^{4} - 414720\Big)^{3}.
\end{align*}

One way to check the computation is to verify that for any integer $n$,
$n! [T^n] \frac{\sk_k(T)}{(1-T)^{3 k}}$ is equal to the number of positive graphs
with minimum degree at least 2
with $n$ vertices and excess $k$.

    \subsection{Taylor expansion of $A_{r,\ell}(\zeta e^{i x},\lambda e^{i y})$} \label{sec:Arl}

The function $A_{r,\ell}(z,x)$ can be rewritten
as a function $f(T(z), x)$, analytic at $(T(\zeta), \lambda)$.
According to Lemma~\ref{th:composed_taylor},
the Taylor expansion of $A_{r,\ell}(\zeta e^{i x},\lambda e^{i y})$
is then equal to
\[
    \sum_{a,b \geq 0}
    \left( \frac{T(\zeta)}{1-T(\zeta)} \partial_{T(\zeta)} \right)^a
    \left(\lambda \partial_{\lambda} \right)^b
    f(T(\zeta), \lambda)
    \frac{(i x)^a}{a!}
    \frac{(i y)^b}{b!}.
\]

    \subsection{Taylor expansion of $\psi(x,y)$} \label{sec:phi}

We start with the function $\phi(x,y)$, which is defined as
\[
    \phi(x,y) =
    - \log \left( \frac{B(\zeta e^{i x}, \lambda e^{i y})}{B(\zeta, \lambda)} \right)
    + i \left( \frac{x}{\alpha} + 2 y \right).
\]

\begin{lemma} \label{th:phi}
The Taylor expansion of $\phi(x,y)$ is equal to
\[
    \phi(x,y) =
    \sum_{a+b \geq 2}
    \left( \frac{T(\zeta)}{1-T(\zeta)} \partial_{T(\zeta)} \right)^a
    (\lambda \partial_{\lambda})^b
    \log \left(
    1 - T(\zeta) \frac{e^{\lambda}-1-\lambda}{\lambda^2/2}
    \right)
    \frac{(i x)^a}{a!} \frac{(i y)^b}{b!}.
\]
\end{lemma}

\begin{proof}
By construction, $\phi(0,0) = 0$,
and $\zeta$, $\lambda$ have been chosen
to ensure that the partial derivatives
$\partial_1 \phi(0,0)$ and $\partial_2 \phi(0,0)$
vanish as well.
Therefore, the Taylor expansion of $\phi(x,y)$ at $(0,0)$ is
\[
    \phi(x,y) =
    - \sum_{a+b \geq 2}
    \partial_u^a \partial_v^b
    \log(B(\zeta e^{i u}, \lambda e^{i v}))_{|u=v=0}
    \frac{x^a}{a!} \frac{y^b}{b!}.
\]
Since
\[
    B(\zeta, \lambda) =
    \left(
    1 - T(\zeta) \frac{e^{\lambda}-1-\lambda}{\lambda^2/2}
    \right)^{-1}
\]
is a simple function of $T(\zeta)$ and $\lambda$,
Lemma~\ref{th:composed_taylor} is applicable and concludes the proof.
\end{proof}

%
%


Morse's Lemma ensure the existence
of a biholomorphic function $\psi = (\psi_1, \psi_2)$ of $(x,y)$
mapping $(0,0)$ to $(0,0)$ such that
\[
    \phi(\psi_1, \psi_2) = \frac{x^2 + y^2}{2}.
\]
We now compute the Taylors series at the origin of $\psi_1$ and $\psi_2$.
To do so, we follow the proof of Morse's Lemma.
Let us introduce the functions $\phi_1$, $\phi_2$ and $\phi_3$ of $(\psi_1,\psi_2)$
\begin{align*}
    \phi_1 &=
    \frac{1}{\psi_1^2}
    \sum_{a+b \geq 2}
    \frac{a(a-1)}{(a+b)(a+b-1)}
    [u^a v^b] \phi(u,v)
    \psi_1^a \psi_2^b,
    \\
    \phi_2 &=
    \frac{1}{\psi_1 \psi_2}
    \sum_{a+b \geq 2}
    \frac{2 a b}{(a+b)(a+b-1)}
    [u^a v^b] \phi(u,v)
    \psi_1^a \psi_2^b,
    \\
    \phi_3 &=
    \frac{1}{\psi_2^2}
    \sum_{a+b \geq 2}
    \frac{b(b-1)}{(a+b)(a+b-1)}
    [u^a v^b] \phi(u,v)
    \psi_1^a \psi_2^b.
\end{align*}
Since $\frac{a(a-1)}{(a+b)(a+b-1)} + \frac{2 a b}{(a+b)(a+b-1)} + \frac{b(b-1)}{(a+b)(a+b-1)} = 1$,
those definitions ensure
\[
    \phi(\psi_1,\psi_2) = \psi_1^2 \phi_1 + \psi_1 \psi_2 \phi_2 + \psi_2^2 \phi_3.
\]
This is rewritten as
\[
    \phi(\psi_1,\psi_2) = 
    \left( \psi_1 \sqrt{\phi_1} + \frac{\psi_2 \phi_2}{2 \sqrt{\phi_1}} \right)^2 +
    \psi_2^2 \phi_3 - \left( \frac{\psi_2 \phi_2}{2 \sqrt{\phi_1}} \right)^2.
\]
Since we want $\psi_1$, $\psi_2$ to be solutions of $\phi(\psi_1, \psi_2) = \frac{x^2+y^2}{2}$,
we set
\begin{align*}
    x &= \sqrt{2} \left( \psi_1 \sqrt{\phi_1} + \frac{\psi_2 \phi_2}{2 \sqrt{\phi_1}} \right),\\
    y &= \sqrt{ 2 \left( \psi_2^2 \phi_3 - \left( \frac{\psi_2 \phi_2}{2 \sqrt{\phi_1}} \right)^2 \right) },
\end{align*}
which is equivalent with
\begin{align*}
    \psi_1 &= \frac{x}{\sqrt{2 \phi_1}} - y \frac{\phi_2}{\sqrt{2 \phi_1 (4 \phi_1 \phi_3 - \phi_2^2)}},\\
    \psi_2 &= y \sqrt{\frac{2 \phi_1}{4 \phi_1 \phi_3 - \phi_2^2}}. \nonumber
\end{align*}

\begin{lemma} \label{th:psi}
We define a sequence of pairs of multinomial
$(\psi^{(j)})_j = (\psi_1^{(j)}, \psi_2^{(j)})_j$
in the variable $(x,y)$
using the initialization $\psi^{(0)} = (0,0)$,
and the recurrence relations
\begin{align} \label{eq:psi}
    \psi_1^{(j+1)} &= \frac{x}{\sqrt{2 \phi_1(\psi^{(j)}})} - y \frac{\phi_2(\psi^{(j)})}{\sqrt{2 \phi_1(\psi^{(j)}) (4 \phi_1(\psi^{(j)}) \phi_3(\psi^{(j)}) - \phi_2(\psi^{(j)})^2)}},\\
    \psi_2^{(j+1)} &= y \sqrt{\frac{2 \phi_1(\psi^{(j)})}{4 \phi_1(\psi^{(j)}) \phi_3(\psi^{(j)}) - \phi_2(\psi^{(j)})^2}}. \nonumber
\end{align}
Then for any nonnegative integers $a$, $b$ such that $a+b \leq j$,
$\psi_1^{(j)}$ (\resp $\psi_2^{(j)}$)
and the function $\psi_1$ (\resp $\psi_2$)
have the same monomial $x^a y^b$
in their Taylor expansion at the origin.
\end{lemma}

\begin{proof}
Since $\phi_1$ and $4 \phi_1 \phi_3 - \phi_2^2$
have positive constant terms,
the operator $A = (A_1, A_2)$ defined by
\begin{align*}
    A_1(u,v) &=
    \frac{x}{\sqrt{2 \phi_1(u,v})} - y \frac{\phi_2(u,v)}{\sqrt{2 \phi_1(u,v) (4 \phi_1(u,v) \phi_3(u,v) - \phi_2(u,v)^2)}}
    ,\\
    A_2(u,v) &= y \sqrt{\frac{2 \phi_1(u,v)}{4 \phi_1(u,v) \phi_3(u,v) - \phi_2(u,v)^2}}.
\end{align*}
is an operator from the set of pairs of functions
in the variables $(x,y)$ analytic at $(0,0)$ to itself.
Let us define as $a+b$ the \emph{total degree} of the monomial $x^a y^b$.
We prove the lemma by recurrence.
Since $\psi(0,0) = (0,0)$, the initialization is correct.
Now assume that $\psi$ and $\psi^{(j)}$
have equal coefficients up to the total degree $j$.
Since $A_1$ and $A_2$ have valuation $1$ in $(x,y)$,
all the coefficients of $A(\psi)$ and $A(\psi^{(j)})$ are equal
up to the total degree $j+1$.
We have $\psi^{(j+1)} = A(\psi^{(j)})$,
and we proved before the statement of the lemma
that $\psi$ is solution of the fixed point equation $\psi = A(\psi)$.
Thus, all the coefficients of $\psi$ and $\psi^{(j+1)}$ are equal
up to the total degree $j+1$.
\end{proof}

The computation of the Taylor expansion of
\[
    J_{\psi}(x,y) =
    (\partial_x \psi_1(x,y)) (\partial_y \psi_2(x,y))
    - (\partial_x \psi_2(x,y)) (\partial_y \psi_1(x,y))
\]
is easy once the Taylor expansion of $\psi(x,y)$ is done.

Combining those computations,
we obtain the coefficients of the asymptotic expansion.
For example, we have computed using \cite{sagemath}
\[
  c_0 =
  \frac{\alpha}{\sqrt{2}}
  \frac{\lambda - 2 \alpha}{\sqrt{\lambda^2 - 4 \alpha^2 - 4 \alpha}}
  e^{- (1 + \alpha / 2) \lambda}
\]
and
\begin{align*}
  c_1 =
  \frac{\alpha \sqrt{2}}{384}
  \frac{e^{- (1 + \alpha / 2) \lambda}}
    {\sqrt{\lambda^2 - 4 \alpha^2 - 4 \alpha}}
  \Bigg(
    &\alpha \lambda (\lambda - 2 \alpha) (6 \alpha + 3 \lambda + 8) (2 \alpha + \lambda + 8) (2 \alpha + \lambda + 2)
    \\&
    - \frac{2 \alpha \lambda (\lambda - 2 \alpha - 2)^4 (\lambda - 2 \alpha + 10)}{(\lambda - 2 \alpha)^2}
    - \frac{A}{(\lambda - 2 \alpha)^2 (\lambda^2 - 4 \alpha^2 - 4 \alpha)^3}
  \Bigg)
\end{align*}
where
\begin{align*}
  A &=
12288 \, \alpha^{13}\lambda 
- 18432 \, \alpha^{11}\lambda^{3} 
+ 11520 \, \alpha^{9}\lambda^{5} 
- 3840 \, \alpha^{7}\lambda^{7} 
+ 720 \, \alpha^{5}\lambda^{9} 
- 72 \, \alpha^{3}\lambda^{11} 
+ 3 \, \alpha\lambda^{13} 
\\&
+ 110592 \, \alpha^{12}\lambda 
- 24576 \, \alpha^{11}\lambda^{2} 
- 156672 \, \alpha^{10}\lambda^{3} 
+ 46080 \, \alpha^{9}\lambda^{4} 
+ 81408 \, \alpha^{8}\lambda^{5} 
- 27648 \, \alpha^{7}\lambda^{6} 
- 19584 \, \alpha^{6}\lambda^{7} 
\\&
+ 7296 \, \alpha^{5}\lambda^{8} 
+ 2160 \, \alpha^{4}\lambda^{9} 
- 864 \, \alpha^{3}\lambda^{10} 
- 84 \, \alpha^{2}\lambda^{11} 
+ 36 \, \alpha\lambda^{12} 
+ 393216 \, \alpha^{11}\lambda 
- 221184 \, \alpha^{10}\lambda^{2} 
\\&
- 396288 \, \alpha^{9}\lambda^{3} 
+ 251904 \, \alpha^{8}\lambda^{4} 
+ 129024 \, \alpha^{7}\lambda^{5} 
- 99840 \, \alpha^{6}\lambda^{6} 
- 12672 \, \alpha^{5}\lambda^{7} 
+ 16512 \, \alpha^{4}\lambda^{8} 
- 768 \, \alpha^{3}\lambda^{9} 
\\&
- 960 \, \alpha^{2}\lambda^{10} 
+ 132 \, \alpha\lambda^{11} 
+ 712704 \, \alpha^{10}\lambda 
- 688128 \, \alpha^{9}\lambda^{2} 
- 285696 \, \alpha^{8}\lambda^{3} 
+ 448512 \, \alpha^{7}\lambda^{4} 
- 26880 \, \alpha^{6}\lambda^{5} 
\\&
- 89856 \, \alpha^{5}\lambda^{6} 
+ 20928 \, \alpha^{4}\lambda^{7} 
+ 4608 \, \alpha^{3}\lambda^{8} 
- 1872 \, \alpha^{2}\lambda^{9} 
+ 144 \, \alpha\lambda^{10} 
+ 806912 \, \alpha^{9}\lambda 
- 1183744 \, \alpha^{8}\lambda^{2} 
\\&
+ 306176 \, \alpha^{7}\lambda^{3} 
+ 332800 \, \alpha^{6}\lambda^{4} 
- 211712 \, \alpha^{5}\lambda^{5} 
+ 12544 \, \alpha^{4}\lambda^{6} 
+ 19520 \, \alpha^{3}\lambda^{7} 
- 5440 \, \alpha^{2}\lambda^{8} 
+ 416 \, \alpha\lambda^{9} 
\\&
- 4096 \, \alpha^{9} 
+ 759808 \, \alpha^{8}\lambda 
- 1361920 \, \alpha^{7}\lambda^{2} 
+ 699392 \, \alpha^{6}\lambda^{3} 
+ 78848 \, \alpha^{5}\lambda^{4} 
- 170240 \, \alpha^{4}\lambda^{5} 
+ 42880 \, \alpha^{3}\lambda^{6} 
\\&
- 704 \, \alpha^{2}\lambda^{7} 
- 496 \, \alpha\lambda^{8} 
- 8 \,\lambda^{9} - 20480 \, \alpha^{8} 
+ 628736 \, \alpha^{7}\lambda 
- 1033216 \, \alpha^{6}\lambda^{2} 
+ 534016 \, \alpha^{5}\lambda^{3} 
- 33536 \, \alpha^{4}\lambda^{4} 
\\&
- 41344 \, \alpha^{3}\lambda^{5} 
+ 6720 \, \alpha^{2}\lambda^{6} 
+ 416 \, \alpha\lambda^{7} 
- 28672 \, \alpha^{7} 
+ 296960 \, \alpha^{6}\lambda 
- 436224 \, \alpha^{5}\lambda^{2} 
+ 203776 \, \alpha^{4}\lambda^{3} 
\\&
- 16128 \, \alpha^{3}\lambda^{4} 
- 6016 \, \alpha^{2}\lambda^{5} 
- 12288 \, \alpha^{6} 
+ 10240 \, \alpha^{5}\lambda 
- 72704 \, \alpha^{4}\lambda^{2} 
+ 35328 \, \alpha^{3}\lambda^{3} 
- 28672 \, \alpha^{4}\lambda
\end{align*}
The next terms are too big to fit in a page.

We thank an anonymous referee for providing us with tables of numbers of connected graphs.
This referee computed them using a recurrence obtained by counting
in two different ways connected graphs with one marked edge.
Part of those tables are presented in Figure~\ref{fig:numerical_verification}.
The correct digits obtained by using
a first or second order expansion of $\csg_{n,k}$
are highlighted.

\begin{figure}
\begin{center}
\begin{tabular}{|c|c|c|c|c|c|}
$n \diagdown \alpha$ & $1/10$ & $1/4$ & $1/2$ & $1$ & $2$\\ \hline
$100$ & \textcolor{blue}{7}\textcolor{red}{.02}48161e219 & \textcolor{blue}{1.0}\textcolor{red}{323}4865e247 & \textcolor{blue}{9}\textcolor{red}{.08}81507e287 & \textcolor{blue}{3}\textcolor{red}{.0}930091e361 & \textcolor{red}{6}.2268140e489\\
$200$ & \textcolor{blue}{4.8}\textcolor{red}{25}1554e507 & \textcolor{blue}{2}\textcolor{red}{.861118}04e571 & \textcolor{blue}{9}\textcolor{red}{.68}20759e668 & \textcolor{blue}{2}\textcolor{red}{.22}17143e847 & \textcolor{blue}{1}\textcolor{red}{.1}340686e1167\\
$300$ & \textcolor{blue}{3.7}\textcolor{red}{58}0934e820 & \textcolor{blue}{2.3}\textcolor{red}{03864}19e924 & \textcolor{blue}{1.4}\textcolor{red}{3}10237e1084 & \textcolor{blue}{5}\textcolor{red}{.01}77562e1378 & \textcolor{blue}{3}\textcolor{red}{.2}403053e1912\\
$400$ & \textcolor{blue}{5}\textcolor{red}{.294}5796e1149 & \textcolor{blue}{5}\textcolor{red}{.487049}75e1295 & \textcolor{blue}{3.0}\textcolor{red}{554}585e1521 & \textcolor{blue}{3}\textcolor{red}{.899}6489e1939 & \textcolor{blue}{1}\textcolor{red}{.79}51048e2702\\
$500$ & \textcolor{blue}{9}\textcolor{red}{.05}69880e1490 & \textcolor{blue}{7.0}\textcolor{red}{22662}24e1680 & \textcolor{blue}{1.9}\textcolor{red}{405}577e1975 & \textcolor{blue}{2}\textcolor{red}{.760}5105e2522 & \textcolor{blue}{8}\textcolor{red}{.4}704376e3524\\
$600$ & \textcolor{blue}{6}\textcolor{red}{.576}4746e1841 & \textcolor{blue}{7.8}\textcolor{red}{2032}799e2076 & \textcolor{blue}{1.6}\textcolor{red}{459}619e2442 & \textcolor{blue}{6}\textcolor{red}{.146}6205e3122 & \textcolor{red}{6.9}645842e4373\\
$700$ & \textcolor{blue}{4.8}\textcolor{red}{343}130e2200 & \textcolor{blue}{1.0}\textcolor{red}{866635}8e2482 & \textcolor{blue}{1.1}\textcolor{red}{45}6019e2920 & \textcolor{blue}{4}\textcolor{red}{.85}60190e3737 & \textcolor{blue}{3}\textcolor{red}{.79}57600e5244\\
$800$ & \textcolor{blue}{2.5}\textcolor{red}{299}174e2566 & \textcolor{blue}{9.2}\textcolor{red}{30277}63e2894 & \textcolor{blue}{1.7}\textcolor{red}{548}159e3407 & \textcolor{blue}{1}\textcolor{red}{.0942}901e4365 & \textcolor{red}{9.90}71844e6133\\
$900$ & \textcolor{blue}{1.29}\textcolor{red}{55}157e2938 & \textcolor{blue}{5.0}\textcolor{red}{27281}58e3314 & \textcolor{blue}{3.9}\textcolor{red}{548}560e3902 & \textcolor{blue}{1.9}\textcolor{red}{10}1224e5003 & \textcolor{blue}{5}\textcolor{red}{.54}99666e7039\\
$1000$ & \textcolor{blue}{1.39}\textcolor{red}{142}34e3315 & \textcolor{blue}{3.0}\textcolor{red}{510646}2e3740 & \textcolor{blue}{1.6}\textcolor{red}{0649}97e4405 & \textcolor{blue}{1.5}\textcolor{red}{722}527e5651 & \textcolor{blue}{1.0}\textcolor{red}{03}2961e7960
\end{tabular}
\caption{The exact number of connected graphs with $n$ vertices and $(\alpha + 1) n$ edges.
In blue, the correct digits obtained from using only the constant term $c_0$.
In red, the additional correct digits obtained when the first error term $c_1$ is used.}
\label{fig:numerical_verification}
\end{center}
\end{figure}

\end{document}